%% file: main.tex
\documentclass[10pt]{article}

\PassOptionsToPackage{table}{xcolor}
\newcommand{\ArticleTitle}{Lexicographic Minimax Load Balancing for T-Adaptive Segment Routing}
\newcommand{\ArticleAuthor}{Kaoutar Bouaachra}
\input{preamble}

\usepackage{thmtools}
\usepackage{thm-restate}
\usepackage{pifont}
\usepackage{arydshln}

\definecolor{babypink}{rgb}{0.96, 0.76, 0.76}

\definecolor{boxgrey}{RGB}{80,75,80}
\tikzstyle{problembox}  = [draw=boxgrey, fill=white, thin, rectangle,
                           rounded corners=2pt,
                           inner sep=9pt, inner ysep=10pt]
\tikzstyle{problemtitle}= [fill=white, text=boxgrey, font=\small\itshape]
\newcommand{\pboptdef}[3]{%
  \begin{center}
  \begin{tikzpicture}
    \node[problembox] (box){%
      \begin{minipage}{0.88\textwidth}
        \textbf{Input:~}{#2}\\[2pt]
        \textbf{Output:~}{#3}
      \end{minipage}};
    \node[problemtitle, right=8pt] at (box.north west) {\textsc{#1}};
  \end{tikzpicture}
  \end{center}}

\newenvironment{researchquestion}{%
  \par\smallskip\noindent
  \begin{minipage}{\linewidth}\hrule\vspace{0.5em}
  \noindent\textit{Research question.}\quad}{%
  \vspace{0.5em}\hrule\end{minipage}\par\smallskip}

\usepackage[colorinlistoftodos,prependcaption,textsize=scriptsize]{todonotes}

\newcommand{\cmark}{{\color{green}\Large\ding{51}}}%
\newcommand{\xmark}{{\color{red}\Large\ding{55}}}%

\title{On The Lexicographic Minimax Load Balancing\\
  for T-Adaptive Segment Routing}
\author{Kaoutar Bouaachra}
\date{}

\title{Lexicographic Minimax Load Balancing\\
       for T-Adaptive Segment Routing}

\author{Kaoutar Bouaachra\\[2pt]
  \small École polytechnique, Institut Polytechnique de Paris}

\date{Internship report}

\begin{document}
\maketitle

\begin{center}
\small
\begin{tabular}{@{}ll@{}}
  \textbf{Host organization}    & Orange Research, Data \& AI Division, MORE team \\
  \textbf{Industry supervisor}  & Yannick Carlinet \\
  \textbf{Context}              & ROADEF/EURO Challenge 2026, organized by  Orange \\
\end{tabular}
\end{center}

\begin{abstract}

We introduce the $T$-\srchallengeLong, a multi-period optimization problem that emerges in the traffic engineering of core IP/MPLS networks when scheduled maintenance operations are involved. The problem seeks a sequence of \textit{Segment Routing (SR)} configurations that can adapt to a multi-period scheduled maintenance, while allowing limited number of path reconfigurations between successive time steps.

Instead of only minimizing the classical \textit{Maximum Link Utilization (MLU)} criteria, we propose a more refined lexicographic objective that minimizes the sorted vector of link loads in its entirety, thereby enabling a more efficient resource utilization across all links.

To tackle this problem, we develop three exact formulations for solving the resulting lexicographic optimization problem. Among these, the \textsc{Stela} and \textsc{Carla} formulations provide the most favorable computational basis for trajectory-based optimization. We thus develop a trajectory column-generation scheme for these formulations, using pricing oracles that vary from heuristic to exact. Exact pricing gives an exact solution of the LP relaxation of the trajectory master, but it does not itself give a certificate of integer optimality: solving the integer master over the columns generated at the root node only certifies optimality within the generated column pool.

To overcome this gap, we incorporate trajectory column generation into a Branch-and-Price framework. The branching rule operates on aggregate variables representing segment usage, which correspond to the original compact routing variables. Thus, Branch-and-Price guarantees global integer optimality for every \textsc{Stela} and \textsc{Carla} rank, and when all ranks are solved to exactness, it guarantees lexicographic optimality of the final solution.

\end{abstract}

\input{sections/01-introduction}
\newpage
\input{sections/notation}
\newpage
\input{sections/02-background}

\input{sections/03-model}
\input{sections/04-methodsexact}

\input{sections/05-methodCG}
\input{sections/06-branchprice}

\input{sections/conclusionanddiscussion}
\printbibliography

\end{document}

%% file: preamble.tex
\usepackage[utf8]{inputenc}
\usepackage[english]{babel}

\usepackage[final]{microtype}
\usepackage{setspace}
\usepackage{indentfirst}
\usepackage[a4paper, margin=0.9in]{geometry}

\usepackage{xcolor}

\definecolor{burgundy}{RGB}{128,0,32}
\definecolor{darkblue}{RGB}{0,0,139}
\definecolor{normalblue}{RGB}{0,90,180}

\usepackage[
    colorlinks=true,
    linkcolor=burgundy,   
    citecolor=blue,   
    urlcolor=normalblue   
]{hyperref}

\hypersetup{
  pdftitle  = {\ArticleTitle},
  pdfauthor = {\ArticleAuthor},
  pdfcreator= {LaTeX}
}

\usepackage{amsmath, amssymb, amsthm, mathtools, bm}
\usepackage{xspace}

\usepackage{dsfont}
\usepackage{siunitx}
\newcommand{\lsr}{\langle}
\newcommand{\rsr}{\rangle}
      
\newcommand{\calA}{\mathcal{A}}

\newcommand{\Es}{E^{\cup}_d}

\newcommand{\minsr}{\textsc{Min Segment Routing}\xspace}
\newcommand{\srchallengeLong}{\textsc{Adaptive Segment Routing}}
\newcommand{\srchallenge}{\textsc{ASR}\xspace}

\newcommand{\segs}{\mathrm{segs}}

\newcommand{\mlu}{\mathrm{MLU}}

\newcommand{\maxSeg}{\texttt{maxSeg}}

\newcommand{\dist}{\operatorname{dist}}

\newcommand{\pool}{\mathrm{pool}}

\newcommand{\rc}{\bar{c}}
\newcommand{\AT}{\mathcal{A}_T}            

\newcommand{\bw}{\mathbf{w}}

\newcommand{\lex}{\leq_{\mathrm{lex}}}
\newcommand{\GG}{\mathcal{G}}

\DeclareMathOperator{\val}{val}

\theoremstyle{plain}
\newtheorem{theorem}{Theorem}
\newtheorem{proposition}{Proposition}
\newtheorem{lemma}{Lemma}
\newtheorem{corollary}{Corollary}

\theoremstyle{definition}
\newtheorem{definition}{Definition}

\theoremstyle{remark}
\newtheorem{remark}{Remark}

\usepackage{graphicx}
\graphicspath{{figures/}}
\usepackage{booktabs}
\usepackage{tabularx}
\usepackage{array}
\usepackage{multirow}
\usepackage{caption}
\usepackage{subcaption}

\usepackage{float}
\usepackage{algorithm}
\usepackage[noend]{algpseudocode}
\algnewcommand\algorithmicinput{\textbf{Input:}}
\algnewcommand\Input{\item[\algorithmicinput]}

\algnewcommand\algorithmicoutput{\textbf{Output:}}
\algnewcommand\Output{\item[\algorithmicoutput]}

\usepackage{listings}
\lstdefinestyle{report}{
  basicstyle=\ttfamily\small,
  numbers=left, numberstyle=\scriptsize, numbersep=8pt,
  frame=single, framerule=0.3pt, breaklines=true,
  showstringspaces=false, tabsize=2,
  columns=fullflexible, keepspaces=true
}
\usepackage{enumitem}
\setlist[itemize] {topsep=0.3em, itemsep=0.1em, parsep=0pt, leftmargin=1.6em}
\setlist[enumerate]{topsep=0.3em, itemsep=0.1em, parsep=0pt, leftmargin=1.8em}

\usepackage{csquotes}
\usepackage[
  backend=biber,
  style=numeric-comp,
  sorting=nyt,
  maxbibnames=99,
  giveninits=true,
  doi=true,
  url=false,
  isbn=false
]{biblatex}
\usepackage{tikz}

\usetikzlibrary{
  arrows.meta,
  decorations.pathreplacing,
  calc,
  positioning,
  shapes.geometric,
  fit,
  backgrounds
}

\definecolor{acc}{RGB}{25,55,110}
\usepackage{pifont}   

\usepackage[nameinlink, noabbrev]{cleveref}
\crefname{theorem}{theorem}{theorems}
\crefname{proposition}{proposition}{propositions}
\crefname{lemma}{lemma}{lemmas}
\crefname{corollary}{corollary}{corollaries}
\crefname{definition}{definition}{definitions}
\crefname{remark}{remark}{remarks}
\crefname{assumption}{assumption}{assumptions}
\crefname{algorithm}{algorithm}{algorithms}
\crefname{figure}{figure}{figures}
\crefname{table}{table}{tables}
\crefname{equation}{equation}{equations}

\usepackage{biblatex}
\usepackage{pgffor}
\usepackage{pgfmath}

%% file: sections/01-introduction.tex
\section{Introduction}
\label{chap:introduction}

\subsection{Context and motivation}

With the arrival of next-generation networks enabled by disruptive technologies such as Software-Defined Networking (SDN) and Network Virtualization, network managers have to reconsider classic \textit{Traffic Engineering} (TE) challenges and appropriate solutions through a different lens. These challenges include the adaptation of networks to accommodate the ever-growing volume of traffic and the multiplicity of user services and contents, all while maintaining high quality of service (QoS). Like for other telecommunications operators, TE is therefore a critical field at Orange as it allows to better tune network parameters so as to make a more efficient use of the existing infrastructure and resources. The purpose is mainly to avoid congestion and hence to improve the QoS experienced by users of the multiple services supported by the network. A standard design rule requires the network to remain 100\% functional whatever single link or node failure might occur. The term failure should be understood broadly, covering cases where a link becomes unavailable either due to unforeseen events or deliberate interventions by a network operator. Moreover, today’s geopolitical tensions and accelerating climate change increase the need for networks that are not only secure but also resilient to multiple router or link failures caused by natural disasters or acts of sabotage.

In this context, a key challenge is to design an optimization algorithm that, given an IP/MPLS network and a traffic matrix (that is, a set of projected values based on current measured volumes of traffic between given origin/destination pairs), computes a routing scheme that maintains acceptable network load, even in scenarios of network equipment unavailability.

The proposed routing scheme relies on the so-called Segment Routing (SR) protocol, a network technology proposed by IETF~\cite{rfc8402} which recently attracted much attention in both telecommunication networking and Operations Research communities. Indeed, the packets in a IP/MPLS network are traditionally routed along the shortest paths from the origin to the destination according to a set of arc weights set by the network administrators. While this approach is easy to implement in practice, it comes with several limitations such as the complexity of finding set of link weights that induce shortest-paths with the least possible network congestion while remaining robust against scenarios where multiple network components become unavailable.

Segment Routing was designed to palliate this issue by enabling the possibility to route packets over non-shortest paths without extensive modifications to the network. More precisely, each packet entering the network is assigned to a so-called \emph{segment path}, which is a sequence of routers referred to as \emph{node segments} or \emph{waypoints} that the packet must visit, one after the other in the network before reaching its final destination. Between two waypoints, traditional shortest path-based routing is used. By encoding routing instructions directly into the packet header, SR allows packets to deviate from the single shortest paths from origin to destination, and hence allowing greater flexibility with minimal additional implementation costs for network administrators.


Beyond its algorithmic difficulty, this problem is motivated by the practical
constraints faced by operational teams. In real IP/MPLS networks, routing
changes cannot be applied arbitrarily from one period to the next: they require
validation, deployment, supervision, and sometimes manual intervention.
Moreover, scheduled maintenance operations must be anticipated without
degrading service quality. The problem therefore consists not only in computing
good routing schemes for each network state, but in producing a sequence of
routing schemes that remains close enough to be operationally deployable. 

The choice of objective function is also important. A classical traffic
engineering objective is to minimize the Maximum Link Utilization (MLU), i.e.,
the load of the most congested link. While this criterion is natural and widely
used, it gives only a partial view of the network state. Two routing schemes may
have the same MLU while being very different operationally: in one solution,
only one link may be close to saturation, whereas in another solution many links
may operate near the same critical load. From an operator's perspective, the
second solution is less desirable, because it leaves less residual capacity to
absorb traffic forecast errors, short-term demand variations, or unexpected
failures.

For this reason, we consider a lexicographic objective over the sorted vector of
link loads across all time periods. This objective first minimizes the largest
load, exactly as MLU minimization does. However, among all solutions with the
same optimal MLU, it then minimizes the second-largest load, then the third one,
and so on. Therefore, the lexicographic criterion strictly refines MLU
minimization: any lexicographically optimal solution is MLU-optimal, but an
MLU-optimal solution is not necessarily lexicographically optimal. In practice,
this leads to routing schemes that distribute congestion more evenly over the
network and avoid hiding secondary bottlenecks behind the value of the single
most loaded link.

This objective is also preferable to simple aggregate criteria such as the
average load or a weighted sum of link utilizations. Average-based objectives
may hide local congestion, since a highly saturated link can be compensated by
many lightly loaded links. Weighted sums, on the other hand, require choosing
weights whose operational meaning may be unclear and may change from one
network to another. The lexicographic objective avoids this difficulty: it does
not require arbitrary tuning parameters and directly reflects the operational
priority of eliminating the worst bottlenecks first, then improving the
remaining critical links in a systematic way.

\subsection{Research question}

\begin{researchquestion}
How can multi-period Segment Routing under scheduled network interventions be modeled and solved efficiently while jointly accounting for congestion and operational stability?
\end{researchquestion}
More precisely, given a multi-period sequence of traffic matrices and network states induced by scheduled arc unavailabilities, how can one formulate a mathematical optimization model that computes a sequence of Segment Routing configurations, lexicographically minimizes the sorted vector of link utilizations across all arcs and time periods, and satisfies period-specific reconfiguration budgets between consecutive configurations? Furthermore, how can the resulting large-scale combinatorial optimization problem be solved efficiently on realistic network instances?

\subsection{Contributions}

The work carried out during my internship focuses on the optimization problem
proposed as the topic of the \emph{ROADEF Challenge 2026}, namely the
\emph{$T$-Adaptive Segment Routing} (T-ASR) problem. The problem definition,
as well as the benchmark suite used in the challenge, therefore constitutes
the starting point of this work. Building upon this framework, the main
contributions developed during my internship are the following:

\begin{itemize}

   \item \textbf{The ROADEF Challenge 2026 problem: a new multi-period
segment-routing problem with temporal coupling.}
The \emph{$T$-Adaptive Segment Routing} (T-ASR) problem is the optimization
problem proposed as the topic of the \emph{ROADEF Challenge 2026}. T-ASR is a
multi-period extension of segment routing in which routing decisions are
coupled across time through a limited reconfiguration budget. Unlike
classical segment-routing formulations that optimize each traffic period
independently, T-ASR explicitly captures the trade-off between congestion
reduction and routing stability. We show that this temporal coupling
fundamentally changes the combinatorial structure of the problem: in
particular, preprocessing and dominance arguments developed for the
single-period setting do not generally remain valid. We identify precisely
why the reconfiguration budget breaks these properties and provide
counterexamples establishing that T-ASR is not merely a straightforward
extension of existing formulations.

    \item \textbf{A principled leximin framework for congestion-aware routing.}
    We cast T-ASR as a \emph{lexicographic min-max}  optimization
    problem, in which the complete vector of link--time loads is minimized in
    decreasing lexicographic order. This objective goes beyond minimizing a
    single congestion metric: it first minimizes the worst congestion level,
    then the second-worst level, and so on. Since the ordering of the
    link--time loads is itself solution-dependent, the resulting objective is
    permutation-invariant and cannot be represented as a standard fixed-order
    lexicographic linear objective. We develop a sequential optimization
    framework that explicitly exploits the structure of the successive order
    statistics and provides a rigorous basis for solving this class of
    leximin problems.

    \item \textbf{Three exact MILP formulations for the successive leximin
    subproblems.}
    We develop and analyze three different formulations---\textsc{Alexa},
    \textsc{Stela}, and \textsc{Carla}---for the sequential leximin
    subproblems arising in T-ASR. The formulations provide different
    representations of the same underlying lexicographic structure, ranging
    from explicit ordering decisions to threshold-based and cumulative
    representations. Beyond proposing the formulations, we establish their
    correctness and show how they preserve the lexicographic structure of the
    original problem. This provides a formulation-level contribution that is
    relevant beyond T-ASR, since the same sequential construction can be
    adapted to other combinatorial optimization settings involving
    lexicographic min-max objectives.

    \item \textbf{A tailored column-generation framework for large-scale
    leximin optimization.}
    To overcome the size of the resulting MILPs, we develop a
    \emph{branch-and-price} framework based on column generation. In
    particular, we derive the corresponding pricing problems for the
    formulations used in the master problems and exploit the specific
    structure of segment routing and ECMP forwarding. This allows the
    method to avoid explicitly enumerating the complete set of feasible
    routing configurations, which becomes prohibitive as the network size
    increases. We further investigate two complementary implementations,
    based on \textsc{Stela} and \textsc{Carla}, and study how the choice of
    formulation affects the resulting column-generation process.

    \item \textbf{An exact branch-and-price algorithm for the complete
    lexicographic problem.}
    We integrate the sequential leximin procedure with column generation and
    branching to obtain an exact branch-and-price framework for T-ASR. The
    resulting algorithm does not optimize a surrogate congestion measure or
    approximate the lexicographic objective: it progressively fixes the
    optimal values of the highest-ranked congestion components and continues
    until the complete lexicographic vector is determined. 

    \item \textbf{A theoretical contribution to sequential leximin
    optimization.}
    Beyond the specific routing application, we formalize structural
    properties of the sequential leximin procedure, including the evolution
    of the optimal thresholds, the preservation of previously optimized
    ranks, and the interaction between the successive subproblems. These
    results clarify how a lexicographic min-max objective can be decomposed
    into a sequence of exact optimization problems. The resulting principles
    are not tied to segment routing and can potentially be transferred to
    other optimization problems in which a vector of competing
    criteria must be optimized lexicographically.

   \item \textbf{The benchmark suite introduced by the ROADEF Challenge 2026 and an extensive computational evaluation.}
We conduct an extensive computational evaluation on the benchmark suite
provided as part of the \emph{ROADEF Challenge 2026}, which is specifically
designed to capture the temporal and reconfiguration aspects of T-ASR. We
evaluate the proposed formulations and algorithms across multiple network
scales, comparing the exact formulations, the corresponding column-generation
approaches, and the resulting \textsc{Stela}- and \textsc{Carla}-driven
methods. The computational study highlights both the scalability of the
proposed approaches and the impact of the underlying formulation on their
ability to recover high-quality lexicographic solutions on larger instances.

\end{itemize}
Taken together, these contributions establish a general exact methodology
for optimization problems whose objective is to lexicographically minimize
the sorted values of a collection of competing criteria. While developed
in the context of multi-period segment routing, the sequential
reformulation, formulation principles, and algorithmic machinery developed
in this work are applicable to broader classes of
optimization problems with leximin objectives.

\subsection{Structure of the report}

The remainder of this report is organised as follows.
\Cref{chap:background} reviews segment routing with ECMP, traffic engineering
formulations, and lexicographic minimax optimisation.
\Cref{chap:model} states the T-ASR problem formally: data, constraints, and the
leximin objective on arc--time loads.
\Cref{sec:milp} develops three compact sequential exact algorithms: \textsc{Alexa}, \textsc{Stela} and \textsc{Carla}, and reports the campaigns
that discard \textsc{Alexa} while retaining the other two.
\Cref{sec:cg} applies trajectory-based column generation to both surviving
drivers, with the master problem and the pricing subproblem derived for each.
\Cref{bp-sec:main} extends this into a branch-and-price framework and discusses
the conditions under which it is exact.

%% file: sections/notation.tex
\section{Notation}
\markboth{Notation}{Notation}

Table~\ref{tab:notations} provides an overview of the general notations used throughout this report. Section-specific notations are introduced in the corresponding chapters as needed.

\begin{table}[ht!]
\centering
\footnotesize
\rowcolors{2}{gray!10}{white}
\setlength{\tabcolsep}{6pt}
\renewcommand{\arraystretch}{1.2}

\begin{tabularx}{\textwidth}{l X}
\toprule
\textbf{Notation} & \textbf{Description} \\
\midrule
\multicolumn{2}{l}{\textit{\textbf{Network \& Graph}}} \\
$G=(V,A)$ &
\begin{tabular}[t]{@{}l@{}}
Directed graph representing the network, with vertex set $V$ and arc set $A$.
\end{tabular} \\
$\omega(a)$ &
\begin{tabular}[t]{@{}l@{}}
Weight function for arc $a\in A$, used for shortest-path
computations.
\end{tabular} \\
$c(a)$ & Capacity (bandwidth) of arc $a\in A$. \\
$FG(u,v)$ & Forwarding Graph: arcs on shortest paths from $u$ to $v$. \\
$r(u,v,a,t)$ & Split coeff.: fraction of flow $u\rightarrow v$ on arc $a$ at time $t$. \\
$\lambda(a,t)$ & Load (link utilization) of arc $a$ at time $t$. \\
$\lambda^{\downarrow}(P)$ & The vector of sorted loadings
\\
 $\lambda_{i}^{\downarrow}(P)$ &  its $i$-th component of $\lambda'$
\\
$\mlu$ & Maximum Link Utilization. \\
\midrule
\multicolumn{2}{l}{\textit{\textbf{Time \& Interventions}}} \\
$T$ & Set of time periods $\{0,1,\ldots,h-1\}$. \\
$T^*$ &
\begin{tabular}[t]{@{}l@{}}
Set of time periods excluding the initial nominal state,
$T^*=T\setminus\{0\}$.
\end{tabular} \\
$q(t)$ & Intervention scenario: subset of arcs unavailable at time $t$. \\
$G_t$ & Subgraph of $G$ with available arcs $A\setminus q(t)$ at time $t$. \\
$\kappa(t)$ & Maximum budget for configuration changes at time $t$. \\
\midrule
\multicolumn{2}{l}{\textit{\textbf{Demands \& Segment Routing}}} \\
$D$ &
\begin{tabular}[t]{@{}l@{}}
Set of traffic demands, where each demand $d=(s,t)$\ has source node $s$ and target node $t$.
\end{tabular} \\
$\nu(d,t)$ & Traffic volume for demand $d$ at time $t$. \\
$\lsr s,w_1,\ldots,t\rsr$ & Notation for a segment path with waypoints $w_i$. \\
$\texttt{maxSeg}$ & Maximum number of segments allowed in a path. \\
$P$ & Routing scheme (set of segment paths for all demands). \\
$\dist(P,P')$ & Distance between two routing schemes (nb. of changes). \\
$\delta(p_d,ij)$ & Indicator equal to 1 if path $p_d$ uses segment $(i,j)$. \\
\midrule
\multicolumn{2}{l}{\textit{\textbf{Decision Variables}}} \\
$x^{dt}_{ij}$ &
\begin{tabular}[t]{@{}l@{}}
Binary variable: 1 if the path for demand $d$ uses 
segment $(i,j)$ at time $t$.
\end{tabular} \\
\bottomrule
\end{tabularx}

\vspace{0.1cm}
\captionsetup{justification=centering}
\caption{Table of Notations}
\label{tab:notations}
\end{table}

%% file: sections/02-background.tex
\section{Background, Related Work}
\label{chap:background}

\subsection{Operational setting}

In operational IP/MPLS networks, routing decisions cannot be considered independently
for each traffic situation. Network operators continuously monitor traffic evolution
and anticipate scheduled interventions such as maintenance operations or planned
equipment unavailability. These interventions modify the network state and may require
the computation of new Segment Routing configurations.

However, changing a routing configuration is not a cost-free operation. Each
reconfiguration requires validation, deployment, and coordination by operational
teams. Therefore, the routing scheme computed for a given period should not only
provide good network performance, but should also remain sufficiently close to the
configuration used in previous periods. This operational constraint is modeled through
a reconfiguration budget limiting the amount of changes between consecutive Segment
Routing solutions.

The operational workflow considered in this work is illustrated in
Figure~\ref{fig:workflow}. For each time period, traffic matrices and network
states affected by scheduled interventions are first used to generate feasible
Segment Routing configurations. Candidate configurations are then evaluated according
to their congestion level, while respecting the allowed reconfiguration budget.
Finally, the optimization model selects a sequence of routing configurations that
provides robust traffic engineering decisions over the whole planning horizon.

\begin{figure*}[ht]
\centering
\begin{tikzpicture}[
  font=\small,
  >={Stealth[length=2.2mm]},
  io/.style   ={draw=black, fill=gray!8,  rounded corners=1pt, align=center,
                text width=30mm, minimum height=9mm, inner sep=3pt},
  proc/.style ={draw=black, fill=gray!22, align=center,
                text width=30mm, minimum height=9mm, inner sep=3pt},
  sel/.style  ={draw=black, fill=gray!32, align=center, thick, inner sep=6pt},
  per/.style  ={draw=black, dashed, rounded corners=3pt, inner sep=6pt},
  flow/.style ={->, thick, draw=black},
  budget/.style={<->, thick, draw=black!75, dashed},
  perlbl/.style={font=\bfseries}
]
\def\colsep{5.4}

\foreach \i/\x in {0/0, 1/\colsep, 2/{2*\colsep} } {
  \node[io]   (tm\i)  at (\x,4.0) {Traffic matrix\\[1pt] $\nu(\cdot,\,\i)$};
  \node[io]   (ns\i)  at (\x,2.3) {Network state $G_{\i}$\\[1pt]
                                   \footnotesize interventions $q(\i)$};
  \node[proc] (gen\i) at (\x,0.6) {Generate feasible\\ SR configurations};

  \draw[flow] (tm\i) -- (gen\i);
  \draw[flow] (ns\i) -- (gen\i);

  \begin{scope}[on background layer]
    \node[per, fit=(tm\i)(ns\i)(gen\i)] (frame\i) {};
  \end{scope}
  \node[perlbl, above=2pt of frame\i.north] {$t=\i$};
}

\draw[budget] (frame0.east|-gen0) -- node[above,font=\footnotesize]{$\kappa(1)$}
              (frame1.west|-gen1);
\draw[budget] (frame1.east|-gen1) -- node[above,font=\footnotesize]{$\kappa(2)$}
              (frame2.west|-gen2);

\node[sel, text width=132mm] (opt) at (\colsep,-1.7)
   {\textbf{lexicographic minimax optimization over the whole horizon}\\[1pt]
    jointly evaluate the loads $\lambda(a,t)$ across all arcs and periods,\\
    and minimize the sorted load vector lexicographically,
    subject to the reconfiguration budget};

\foreach \i in {0,1,2}{ \draw[flow] (gen\i.south) -- (gen\i.south|-opt.north); }

\node[font=\Large] at ({2*\colsep+2.6},0.6) {$\cdots$};
\end{tikzpicture}
\caption{\small Operational workflow over the planning horizon. For each period~$t$,
the traffic matrix $\nu(\cdot,t)$ and the network state $G_t$ induced by the
scheduled interventions $q(t)$ feed the generation of feasible Segment Routing
configurations. The reconfiguration budget $\kappa(t)$ couples consecutive
periods. Congestion is not evaluated period by period: the loads
$\lambda(a,t)$ over all arcs and all periods are assembled into a single sorted
vector, which the lexicographic minimax model minimizes jointly subject to the budget.}
\label{fig:workflow}
\end{figure*}
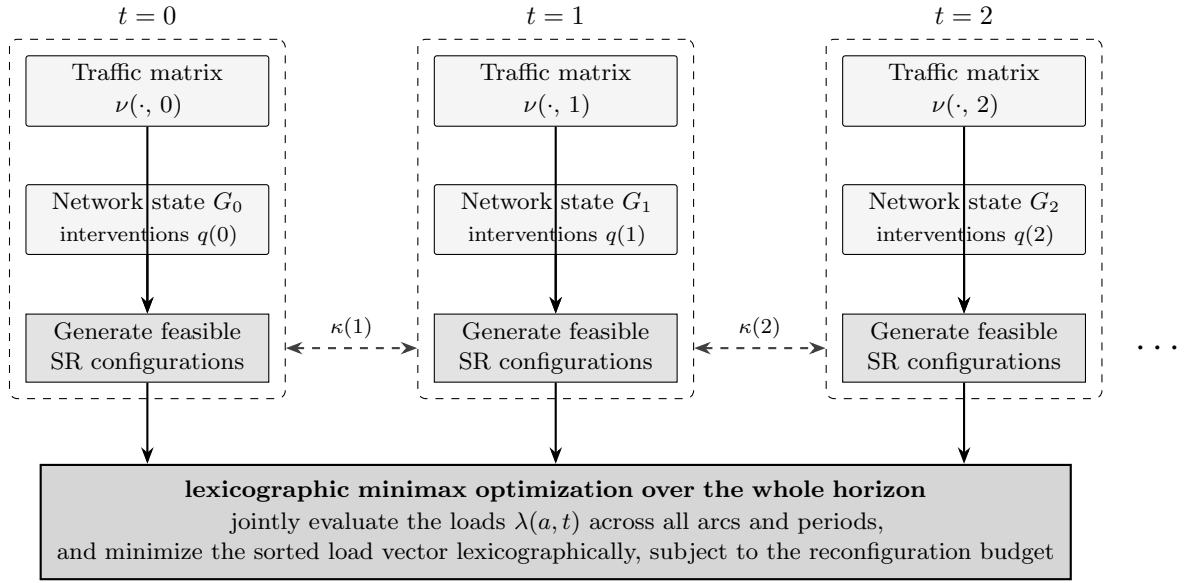

\subsection{State of the Art}
\label{s:sota}

\paragraph{Segment routing and traffic engineering:}
From the complexity point of view, the $T$-\srchallenge problem studied in this paper generalizes the \minsr problem~\cite{hartert2015solving}:
indeed, one can easily verify that \minsr is equivalent to~$\{0\}$-\srchallengeLong{}.
Consequently, the computational intractability results established for \minsr
carry over directly to $T$-\srchallenge{}.
In particular, this problem remains parameterized intractable and inapproximable
(under standard complexity assumptions), even on highly restricted network
topologies~\cite{camille25}.

Most existing optimization approaches focus on the static segment routing
problem, or on variants of it that do not jointly consider multi-period
planning, scheduled link interventions, and bounded waypoint reconfigurations. Bhatia et al.~\cite{bhatia} formulate SR with two
segments as a linear program using a path-based model.
Building on this, Hartert et al.~\cite{hartert2015solving} quantify the
benefit of SR over standard routing protocols, demonstrating congestion
reductions of 30\% to 50\% under known traffic, and introduce a
MILP-based heuristic to reduce computation times.
Schüller et al.~\cite{schuller} further extend this line of work with a
failure-resilient MILP that minimizes the Maximum Link Utilization (MLU),
though at the cost of prohibitive computational times on large instances.
A model supporting an arbitrary number of segments was proposed by ~\cite{cg4sr}, who combine column generation with a branch-and-bound
procedure to tackle the problem at scale. Pereira et al.~\cite{pereira} propose a heuristic for
three-segment SR that incorporates link failure handling and makes use of
adjacency segments.
Gay et al.~\cite{gay_srls} introduce SRLS, a local search heuristic aimed
at fast re-optimization in response to sudden traffic changes, addressing
the practical limitation of MILP-based approaches in time-sensitive
scenarios.
Parham et al.~\cite{parham} consider the joint optimization of SR paths
and shortest-path routing weights, proposing a heuristic for this richer
combined problem.

These works form the closest algorithmic background for our study.
However, they remain essentially static: the routing configuration is optimized
for a single traffic and network state.
They therefore do not address the temporal and reconfiguration-constrained
setting considered in this paper.

In contrast, the problem considered in this paper is inherently multi-period:
routing decisions must be planned over a horizon with scheduled interventions,
while limiting the number of waypoint changes between consecutive time steps. Moreover, rather than minimizing the MLU alone, we adopt a lexicographic
objective over the sorted vector of arc loads, as formally defined in
Section~\ref{problem definition}.
This objective strictly strengthens MLU minimization: any lex-optimal solution
is also MLU-optimal, but the converse does not necessarily hold. Beyond this algorithmic gap, we also propose a set of benchmark instances
specifically designed for this problem, including network topologies, traffic
matrices, intervention scenarios, and budget parameters, generated according to
the procedures described in~\Cref{sec:generation}.
These instances are made publicly available to the community, providing a
common basis for testing and comparing future approaches on this problem
setting.

Unlike the single-period SR optimization studied by Callebaut et al.~\cite{callebaut},
dominated segment paths; in particular paths containing loops; cannot be
removed from the search space in T-ASR without losing optimality. A path that is dominated in one isolated period may be necessary because it remains close to the previous routing scheme and therefore satisfies the budget (see~\Cref{fig:interventionsar3}). This shows that the budget constraint  can force an optimal solution to include segment paths containing loops, which the preprocessing of~\cite{callebaut} would incorrectly discard.

\begin{figure}[ht]
    \centering

    \includegraphics[
        width=0.4\linewidth,
        keepaspectratio
    ]{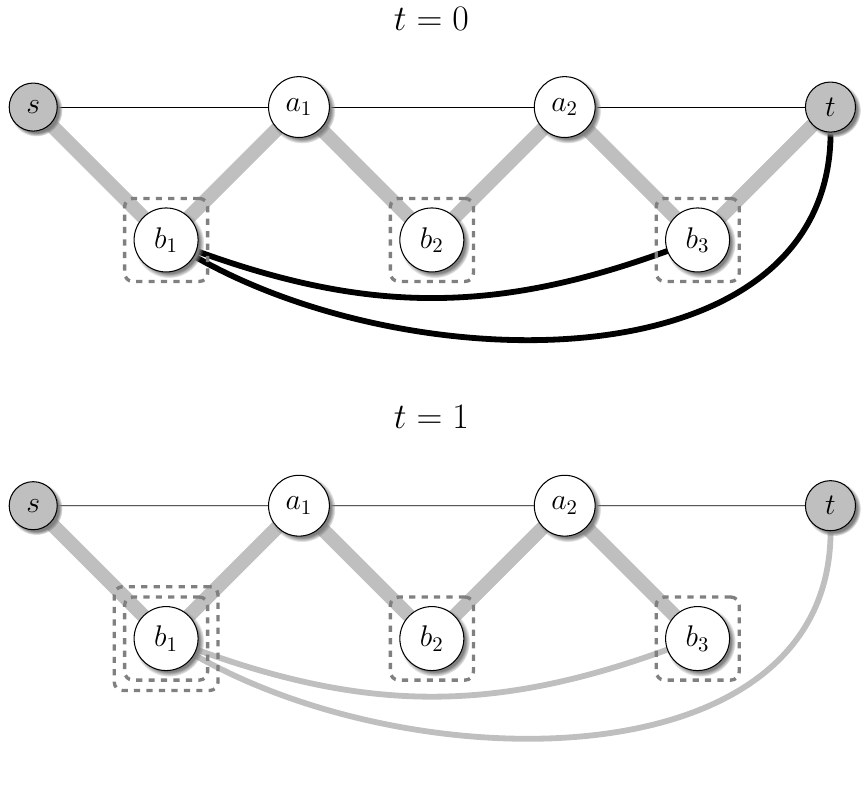}

    \captionsetup{
        font=small,
        skip=3pt,
        justification=justified,
        singlelinecheck=false
    }

    \caption{\small Example of an instance of $\{0,1\}$-\srchallenge{} where the
    only optimal solution includes a segment path containing a loop.
    The input graph is bidirected with each arc~$uv$ having the same
    capacity and weight as its reverse~$vu$. For this reason, we
    represent~$uv$ and~$vu$ as a single edge in the figure for clarity.
    The thickness of an edge represents the arc capacity value. Thin,
    medium, and thick edges correspond to capacities~$1$,~$2$, and~$3$,
    respectively. All arcs have unit weights except for the arcs
    $(b_1,t)$ and $(t,b_1)$, which have weight~$2$. The only demand to be
    routed has traffic volume~$1$. The budget is $\kappa(1)=3$ and the
    intervention scenario is
    $q(1)=\{(b_3,t),(t,b_3)\}$. The grayish links correspond to the arcs
    traversed by the flow. An optimal segment-path solution is
    $p_1=\langle b_1,b_2,b_3\rangle$ and
    $p_2=\langle b_1,b_2,b_3,b_1\rangle$ at times~$0$ and~$1$, with
    induced MLUs of $\frac{1}{3}$ and $\frac{1}{2}$, respectively.
    Notice that $p_1$ yields the best MLU at time~$0$, since it only uses
    arcs with maximum capacity and all three waypoints are necessary to
    achieve this. At time~$t=1$, the given budget allows only the addition
    or deletion of a single waypoint. Deleting any of $p_1$'s waypoints,
    or adding no waypoint at all, would divert the flow onto a thin arc,
    inducing an MLU of~$1$. The remaining possibilities without creating
    a loop consist of adding a waypoint on $a_1$ or $a_2$. However, each
    possibility would again divert the flow onto at least one thin arc.
    Hence, the only remaining option is to add $b_1$ as a new waypoint,
    obtaining an MLU of $\frac{1}{2}$.}

    \label{fig:interventionsar3}
\end{figure}

\paragraph{Lexicographic minimax optimization}
The objective of \mbox{T-ASR} is a \emph{lexicographic min--max} (lexicographic minimax)
criterion applied to the vector of arc loads over all arcs and all time
periods: one minimizes the largest load, then, among the resulting
solutions, the second largest, and so on. This refinement of the
classical min--max (MLU) criterion is the standard way of encoding
\emph{equity} between competing outcomes
\cite{kostreva1999,ogryczak1997location,luss1999}. Its core difficulty is
that the criterion is applied to the \emph{sorted} load vector, whose
underlying permutation is itself solution dependent: which arc--period
pair occupies lexicographic rank $k$ is not known a priori but is part of the decision.
This sets lexicographic minimax apart from lexicographic problems in which the order of
the coordinates is fixed in advance, and it is the source of the
non-linearity that any exact method must address.

%% file: sections/03-model.tex
\section{Mathematical Model}
\label{chap:model}

\subsection{Definitions}
\label{sec:notations}

In simple terms, the problem consists of determining the segment path for each demand in order to balance the network load.
Prior to providing the formal definition of the optimization problem, we introduce the necessary preliminary notations and definitions.

\paragraph*{Network}
A \emph{network} is a tuple $(G=(V,A), \omega, c)$ where $G$ is a directed
graph of~$n$ vertices with vertex set~$V = \{v_0, v_1, \ldots, v_{n-1}\}$
and arc set $A \subseteq \{a_{i,j} : (v_i, v_j) \in V \times V\}$,
$\omega\colon A \rightarrow \mathbb{Q}_{>0}$ is a positive weight function
used to compute the shortest paths in $G$, and
$c\colon A \rightarrow \mathbb{Q}_{>0}$ is a positive capacity function that
represents, for each arc $a \in A$, its bandwidth $c(a)$, that is, the maximum amount of traffic throughput that
it can accommodate. The terms ``arc'' and ``link'' are used interchangeably
throughout this document. Traffic volumes are given by a nonnegative
function $\nu\colon D \times T \rightarrow \mathbb{Q}_{\ge 0}$, where $D$ is
the set of demands and $T$ the set of time periods introduced below.
\paragraph*{Forwarding graph \& ECMP} \label{forwarding_graph}
Let $(G=(V,A), \omega, c)$ be a network. The \emph{forwarding graph} from a \emph{source} $u \in V$ to a \emph{target} $v \in V$ is the subgraph of $G$ containing all arcs that belong to any shortest path (according to weights $\omega$) from $u$ to $v$. It is denoted $FG(u,v)$. When $FG(u,v)$ is not a simple path, the \emph{Equal-Cost Multi-Path} (ECMP) mechanism is activated: the incoming flow in a vertex of $FG(u,v)$ is evenly divided between the outgoing arcs of this vertex in $FG(u,v)$ (see~\Cref{fig:definitions}). 

\paragraph*{Segment routing path} 
A \emph{segment routing path}, or \emph{segment path}, in a network $(G=(V,A), \omega, c)$ is a succession of forwarding graphs such that the target of the previous forwarding graph coincides with the source of the next one. A segment path is denoted by $\lsr s, w_1, \ldots, w_\ell, t \rsr$ where $s$ is the source of the segment path, $t$ the target, and $w_1, \ldots, w_\ell \in V$ are the \textit{waypoints} in that order. The source and target do not count as waypoints. The case where there is no waypoint (i.e. $l = 0$) is possible; it means that the demand is simply routed on the shortest paths from~$s$ to~$t$.

\paragraph*{Segment} 
In this context, we define a \textit{segment} as a pair of successive nodes in a segment path. These nodes are either the source, target or waypoints included in a segment path. For instance, ($s$, $w_{1}$), ($w_{1}$, $w_{2}$), \ldots, ($w_\ell$,~$t$) are the segments composing the segment path $\lsr s, w_1, \ldots, w_\ell, t \rsr$. Each segment~$(u,v)$ of a segment path is thus associated with the forwarding graph~$FG(u,v)$. A typical maximum number of segments in a segment path is between 4 and 10, depending on the underlying protocol and the router technology.

\paragraph*{Demand} 
A \emph{demand} on a network~$(G=(V,A), \omega, c)$ is a couple~$(s,t)$ where the \emph{terminals} $s,t \in V$ are respectively the \emph{source} and \emph{target} of the demand. Each demand is associated with a traffic volume, i.e. the size of the flow that needs to be routed from the source to the target. If the traffic volume is 1, it is referred to as a unit demand. The volume can vary in time, therefore the traffic volume is a function denoted $\nu:  D \times T \to \mathbb{Q}$ with $T$ the discrete set of time periods and $D$ the set of demands, referred to as the traffic matrix.

\paragraph*{Routing scheme} 
Given a network $(G=(V,A), \omega, c)$ and a set of $k$ demands $D = \lbrace (s_i, t_i) : i=1, \ldots, k \rbrace$ on $G$, a \emph{routing scheme} for $D$ is a set~$P$ of $k$ segment paths $\lbrace p_d : d \in D \rbrace$ such that $p_d$ is a segment path from~$s$ to~$t$ associated to the demand~$d=(s,t)$.

\paragraph*{Time horizon} 
The routing scheme has to be decided in advance and over a given time period (from a few days to a couple of weeks). We assume that the traffic demand values can be estimated through forecasting for each time period, whereas the network state (resources status, utilization and availability) depends on the scheduled maintenance interventions. We denote by $T$ = $\{0, 1, \ldots, h-1\}$ the set of time periods, and~${h \geq 1}$ the routing planning horizon. For example, $T$ = $\{0, 1, \ldots, 6\}$ for a daily planning over a week and $T$ = $\{0, 1, 2, 3\}$ for a weekly planning over a month. The time step $0$ is a special case because it represents the nominal situation, when the network runs without any failure. Therefore, we also denote $T^* = T \setminus \{0\}$.

\paragraph*{Interventions} 
An intervention in the network (for maintenance, or other reason) causes a link or a node to be turned down for a certain time. Functionally, it is the same as a failure, except it can be planned in advance. In the following, we consider only links turned down, because if a node is down, it is functionally the same as if all connected links to this node are down. 

\paragraph*{Intervention scenario}
An \emph{intervention scenario} is a function $q\colon T \to 2^A$ satisfying
$q(0) = \varnothing$: for each period $t \in T^*$, the set $q(t) \subseteq A$
contains the arcs of~$G$ made unavailable at period~$t$ by one or several
scheduled maintenance operations, called \textit{interventions}, while
period~$0$ represents the nominal state of the network. We denote by~$G_t$
the subgraph of~$G$ with arc set~$A \setminus q(t)$, that is, the network as
seen at period~$t$ once the intervention scenario has taken down some links.

\paragraph*{Budget}
For a given network and a set of demands, the routing scheme may require some changes or reconfiguration from one time period to the other, in order to take into account the links and nodes that are down. Each change in a routing scheme may require a human action that has a cost. In this context, it is desirable to limit the number of network reconfigurations (i.e. adding/removing waypoints). We denote by~$\dist(P, P') \in \mathbb{N}$ the number of changes between two routing schemes~$P$ and $P'$. More formally, 

$$\dist(P, P') = \sum_{d \in D, i,j \in V, i\neq j} |\delta(p_d,ij) - \delta(p_{d}',ij)|$$
where~$\delta(p_d,ij)$ is equal to 1 if the segment path~$p_d$ associated to the demand~$d$ contains the segment $(i, j)$,~$0$ otherwise (see~\Cref{tab:distances}).
\begin{table}[h!]
\centering
\setlength{\tabcolsep}{10pt}
\renewcommand{\arraystretch}{1}
\rowcolors{2}{gray!10}{white}
\begin{tabular}{lll}
\toprule
\textbf{Routing Scheme} $P$ & \textbf{Routing Scheme} $P'$ & $\dist(P, P')$ \\
\midrule
$\{\lsr s, w, t \rsr\}$ & $\{\lsr s, t \rsr\}$ & $3$ \\
$\{\lsr s, w_1, w_2, t \rsr\}$ & $\{\lsr s, w_1, t \rsr\}$ & $3$ \\
$\{\lsr s, w_1, w_2, t \rsr\}$ & $\{\lsr s, w_2, w_1, t \rsr\}$ & $6$ \\
$\{\lsr s, w_1, t \rsr\}$ & $\{\lsr s, w_2, t \rsr\}$ & $4$ \\
$\{\lsr s_1, w_1, t_1 \rsr,$ & $\{\lsr s_1, w_4, t_1 \rsr,$ & \\
$\quad\lsr s_2, w_2, w_3, t_2 \rsr\}$ & $\quad\lsr s_2, w_3, w_2, t_2 \rsr\}$ & $10~(= 4 + 6)$ \\
\bottomrule
\end{tabular}
\caption{\small Examples of distance values}
\label{tab:distances}
\end{table}
For each time step, a budget function is given, denoted $\kappa : T^* \to \mathbb{N}$, that represents the maximum number of changes allowed from one time step to next one.

\paragraph{Load} 
The load $\lambda(a, t) \in \mathbb{Q}^{+}$ of arc $a \in A$ at time step $t \in T$, is the ratio between the quantity of flow using an arc and its capacity. This ratio is also referred to as \textit{link utilization} and is often used as (part of) the optimization criterion in the literature related to Traffic Engineering. Note that it can be higher than 1 in case of congestion on the link. 
In order to compute the load, we introduce the concept of \textit{split coefficients}, denoted by $r(u, v, a, t)$. They represent the proportion of flow between source node $u$ and target node $v$ that passes through arc $a$ of $FG(u, v)$ under the condition of the network at time $t$ (i.e. in graph $G_{t}$). 
The ratios are given for all couples of nodes $(u, v)$ because the segment $(u, v)$ can be potentially used for routing a demand. 
Note that $r(u, v, a, t)$ is always greater or equal to $0$, and $r(u, v, a, t) = 0$ if arc $a$ is not in the forwarding graph $FG(u, v)$. The load of an arc $a \in A$ at time step $t \in T$ is then:

$$
\lambda(a, t) = \frac{\sum_{d \in D}\sum_{i,j \in V} r(i, j, a, t) \; \nu(d,t) \; \delta(p_d,ij)}{c(a)}
$$
See \Cref{fig:definitions} for an illustration of ECMP, split coefficients and of a segment path with one waypoint.

\begin{figure}[h!]
    \centering
    \includegraphics[width=0.6\linewidth]{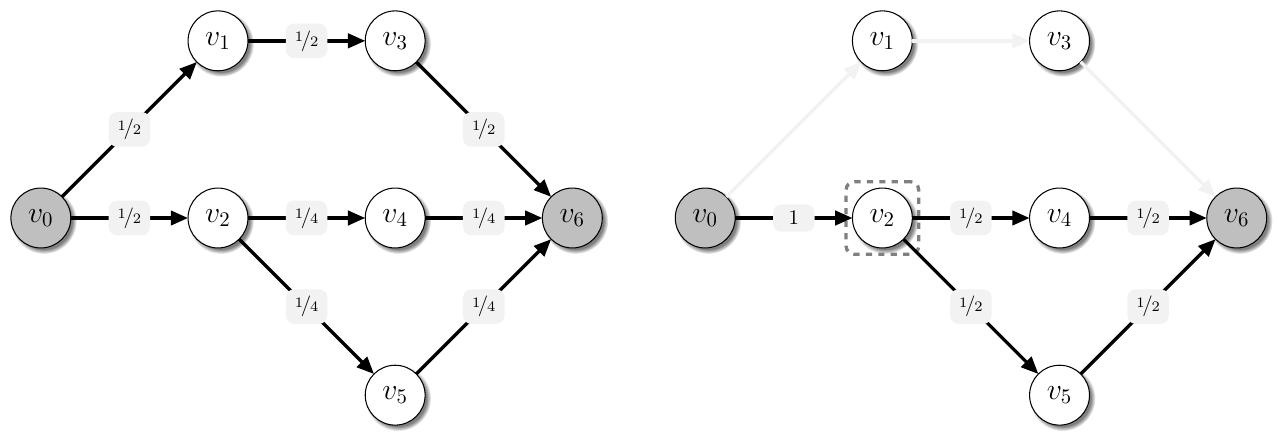}
    \caption{\small On the left, the demand~$(v_0, v_6)$ is routed on a network with unit weights through the segment path~$\lsr v_0, v_6 \rsr$. The associated forwarding graph, $FG(v_0, v_6)$, includes all arcs. The fraction on each arc indicates how the flow is split among the shortest paths (split coefficient) using the ECMP rule. On the right, a waypoint in~$v_2$ is introduced, resulting in the flow being routed through the segment path~$\lsr v_0, v_2, v_6 \rsr$ that is, through the bottom part of the network. The associated forwarding graphs, $FG(v_0, v_2)$ and $FG(v_2, v_6)$, are represented with bold arcs.}
    \label{fig:definitions}
\end{figure}

\paragraph{Computing the split coefficients}
\label{lem:ecmp}
The split coefficients $r(u,v,a,t)$ used in the load expression are not part of
the instance data: they are entirely determined by the topology $G_t$ and the
routing metric $\omega$ through the ECMP rule. Making this dependence explicit
serves two purposes. It provides a constructive way to obtain the values
$r(u,v,a,t)$, and it shows that, for a fixed instance, these values are
\emph{constants}, so that the load $\lambda(a,t)$ is a \emph{linear} function of
the routing variables, the property on which every method developed in this
report relies.

Once $G_t$ and $\omega$ are fixed, the split
coefficients $r(u,v,a,t)$ are obtained by a single topological pass over each
forwarding graph; they are therefore precomputed once and treated as input
parameters of the optimization model. Substituting them into the load definition,
\[
  \lambda(a,t)=\frac{1}{c(a)}\sum_{d\in D}\sum_{i,j\in V}
  r(i,j,a,t)\,\nu(d,t)\,\delta(p_d,ij),
\]
the load is a \emph{linear} function of the routing variables $x^{dt}_{ij}$
(through the indicators $\delta(p_d,ij)$). This linearity is the common foundation
of the exact methods of Chapter~\ref{sec:milp} and of the column
generation of~\Cref{sec:cg}: without it, the load would depend
nonlinearly on the routing, and neither the (MI)LP formulations nor the pricing
subproblem would be available.

\paragraph{Maximum Link Utilization (MLU)} 
Given a network $(G=(V,A), \omega, c)$, a set of $k$ demands $D = \lbrace (s_i, t_i) : i=1, \ldots, k \rbrace$ on $G$, a \emph{routing scheme}~$P$ for $D$, the maximum link utilization, denoted~$\mlu(P,D,G)~\in~\mathbb{Q}^{+}$ is the load of the most loaded link.

\subsection{Problem Formulation}

In this section, we introduce our new problem, which we called the
$T$-\srchallengeLong~($T$-\srchallenge{}). A formal expression of the decision variables, constraints and objective function of $T$-\srchallenge{} is provided.

\subsection{Decisions}

\paragraph{Routing}
Let us define, for every demand $d\in D$, time step $t\in T$ and ordered pair
of distinct nodes $(i,j)$ such that $j$ is reachable from $i$ in $G_t$, the
binary variable $x^{dt}_{ij}$ that takes value~$1$ if the segment path of
demand~$d$ contains the segment~$(i,j)$ at time step~$t$, and~$0$ otherwise.
No variable is introduced for $i=j$, which has no segment interpretation,
nor for a pair whose forwarding graph is empty in~$G_t$, which the model
could otherwise select at no load cost. Such pairs are discarded in a
preprocessing step, and all sums over $i,j\in V$ in the sequel are
understood to range over the remaining segments only.

\noindent Therefore, a solution of the $T$-\srchallenge{}  problem corresponds to a \textit{routing scheme} induced by the variables $x$ with non-zero values, that minimizes the objective described in~\Cref{ss:obj}, and satisfies the constraints given in~\Cref{ss:constraints}, for each period of time. 

\subsection{Constraints} \label{ss:constraints}

\paragraph{Flow-conservation constraints} 
The following set of equalities ensures that each traffic demand $d\in D$ is routed along one segment path which connects its origin and its destination, at each time step $t$, \textit{i.e.}, for all~$i \in V$, $d =(s,t) \in D$ and $t\in T$
\begin{align}
\label{ctn:flow} & \sum_{j \in V\setminus \{i\}} x^{dt}_{ij} - \sum_{j \in V\setminus \{i\}} x^{dt}_{ji} =  
\left \{ \begin{array}{ll}
1 & \mbox{if $i = s$,}\\
-1 & \mbox{if $i=t$,}\\
0 & \mbox{otherwise.}
\end{array} \right.
\end{align}

\paragraph{Number of segments} 
The following inequalities allow, for each demand $d$ and each time step $t$, to limit the number of segments, including the first hop from the source node of $d$ to the first waypoint visited, and the last waypoint visited to the destination node, used in the solution. Formally, for all~$d\in D$ and~$t\in T$,
\begin{align}
\label{ctn:segments} & \sum_{i,j \in V} x^{dt}_{ij} \leq \texttt{maxSeg}
\end{align}

\paragraph{Load constraints}

The total traffic of a link is computed by adding the traffic of all demands that are routed through this link.
Formally, for all~$t \in T$ and $a\in A \setminus q(t)$,
\begin{align}
\label{ctn:load} & \sum_{d \in D}\sum_{i,j \in V} r(i, j, a, t) \; \nu(d,t) \; x^{dt}_{ij}  \leq \lambda(a, t) \; c(a)
\end{align}

\paragraph{Budget constraints} 
Since each network configuration change incurs significant operational costs and may also risk deteriorating the Quality of Service, it is desirable to incorporate budget limitations in the model. The following set of inequalities (\ref{ctn:budget_3}) ensure that the value returned by~$\dist$ (defined in~\Cref{sec:notations}, paragraph {\it Budget}) is indeed bounded by $\kappa$ and allow to restrict the number of waypoint changes scheduled from one time step to the next.
Formally, for all~$t \in T^*$,
\begin{align}
\label{ctn:budget_3} & \sum_{d\in D}\sum_{i,j\in V,i\neq j}  | x^{dt}_{ij} - x^{d, t-1}_{ij} | \leq \kappa(t)
\end{align}

The reconfiguration cost used in this report is the number of segment-usage
changes, i.e.\ the cardinality of the symmetric difference between the segment
sets of two consecutive routing schemes. It should not be interpreted as the
number of waypoint insertions or deletions: one waypoint modification may
change several segments.
The absolute values in~\eqref{ctn:budget_3} admit a standard exact
linearization. Introducing continuous variables $z^{dt}_{ij}\ge 0$ for
$d\in D$, $t\in T^*$ and $i\neq j$, subject to
\begin{equation}\label{ctn:budget-lin}
z^{dt}_{ij}\ \ge\ x^{dt}_{ij}-x^{d,t-1}_{ij},
\qquad
z^{dt}_{ij}\ \ge\ x^{d,t-1}_{ij}-x^{dt}_{ij},
\end{equation}
the budget constraint becomes
\begin{equation}\label{ctn:budget-lin-sum}
\sum_{d\in D}\sum_{\substack{i,j\in V\\ i\neq j}} z^{dt}_{ij}\ \le\ \kappa(t),
\qquad t\in T^*.
\end{equation}
The linearization is exact even though $z$ does not appear in the objective:
\eqref{ctn:budget-lin} forces $z^{dt}_{ij}\ge|x^{dt}_{ij}-x^{d,t-1}_{ij}|$,
so any point feasible for~\eqref{ctn:budget-lin-sum} satisfies the original
constraint; conversely, setting $z^{dt}_{ij}=|x^{dt}_{ij}-x^{d,t-1}_{ij}|$
extends any $x$ feasible for~\eqref{ctn:budget_3} to a feasible point. The
two systems therefore have the same projection onto the $x$-variables, and
the compact formulation is a MILP.

Throughout this report, feasible segment paths are \emph{segment trails}: a
segment may appear at most once in a path, while a node may be visited more
than once. Thus loops created by revisiting a node are permitted, but
traversing the same ordered segment twice is not. The former generality is
necessary: \Cref{fig:interventionsar3} exhibits an instance whose unique
optimal solution uses a segment path revisiting a node. The latter
restriction is forced by the model itself, since $\delta(p_d,ij)$; and
hence $x^{dt}_{ij}$; is a binary indicator and cannot record
multiplicities.
\subsection{Objective}
\label{ss:obj}
For each period $t\in T$, let
$P^t=\{p_d^t:d\in D\}$ denote the routing scheme used at period $t$.
A multi-period routing plan is the tuple
$P=(P^t)_{t\in T}$. Equivalently, $P$ is described by the binary
variables $x^{dt}_{ij}$, and we write $P\in\mathcal{F}$ when $P$ satisfies the
constraints~\eqref{ctn:flow}--\eqref{ctn:budget_3}. Let
$\calA_T = \{(a,t): t\in T,\ a\in A\setminus q(t)\}$ be the set of available
arc--time pairs (with $q(0)=\varnothing$), and $N=|\calA_T|$. Each routing scheme $P$ induce a vector of loads $\lambda(P)$

We denote by
\begin{equation*}
  \label{eq:sorted-vector}
  \lambda^{\downarrow}(P)=\bigl(\lambda^{\downarrow}_{1}(P),\,\ldots,\,
  \lambda^{\downarrow}_{N}(P)\bigr),\qquad
  \lambda^{\downarrow}_{1}(P)\geq\cdots\geq\lambda^{\downarrow}_{N}(P),
\end{equation*}
the vector of these loads sorted in non-increasing order, so that
$\lambda^{\downarrow}_{k}(P)$ is the $k$-th largest load.
Formally, this sorting is realized by a bijection
$\sigma_P:\{1,\dots,N\}\to\calA_T$ satisfying
\begin{equation}
  \label{eq:sorting-perm}
  \lambda\bigl(\sigma_P(1);P\bigr)\geq\cdots\geq
  \lambda\bigl(\sigma_P(N);P\bigr),
  \qquad
  \lambda^{\downarrow}_{k}(P)=\lambda\bigl(\sigma_P(k);P\bigr)
  \ \ \text{for } k=1,\dots,N,
\end{equation}
so that $\sigma_P(k)$ is the arc--time pair carrying the $k$-th largest load.
The sorted vector $\lambda^{\downarrow}(P)$ is uniquely determined by $P$,
whereas $\sigma_P$ is unique only up to the reordering of tied loads.
Crucially, this order is \emph{not} fixed a priori: the sorting permutation
$\sigma_P$; and hence the identity of the pair realizing the $k$-th largest
load; is \emph{endogenous}, i.e.\ it depends on the routing scheme $P$ itself.
This is what sets \eqref{ctn:objlex} apart from classical lexicographic
optimization, where the coordinates are prioritized according to an order
known in advance: here the priority order is a function of the decision
variables, and two feasible schemes may be compared along entirely different
orderings of $\calA_T$.

The objective is to minimize this sorted vector lexicographically,
\begin{equation}
  \label{ctn:objlex}
  \operatorname*{lex\,min}_{P\in\mathcal{F}}\ \lambda^{\downarrow}(P),
\end{equation}
that is, $\lambda^{\downarrow}_{1}(P)$ is minimized first, then $\lambda^{\downarrow}_{2}(P)$ among all schemes achieving the
optimal first component, and so on.

\subsection{Problem definition} \label{problem definition}

We are now in position to give the formal definition of the $T$-\srchallenge{} problem

\noindent
\pboptdef{$T$-\srchallengeLong~($T$-\srchallenge{})}
{A network~$(G=(V,A), \omega, c)$, a set of demands~$D = \lbrace (s_i, t_i) : i=1, \ldots, k \rbrace$, a traffic volume $\nu: D \times T \to \mathbb{Q}$, set of time periods~$T = \{0, 1, \ldots, h-1\}$ with horizon~$h\geq1$, an intervention scenario~$q: T^* \to 2^A$, a budget~$\kappa : T^* \to \mathbb{N}$ and a maximum number of segments~$\maxSeg \in \mathbb{N}$ ($\maxSeg \geq 2$).}
{A routing scheme~$P\in\mathcal{F}$ (one segment path per demand and time step satisfying constraints~\eqref{ctn:flow}--\eqref{ctn:budget_3}) that minimizes the objective~\eqref{ctn:objlex}.}

\section{Benchmark instances and data generation procedure}
\label{sec:generation}

The benchmark instances were generated in order to provide a diverse and non-trivial test bed for the T-ASR problem~\cite{challenge2026}.
\paragraph{SetA:}The \texttt{setA} benchmark was produced by the challenge organizers and is
not generated in this work; the description below summarizes its design
rationale, as it informs the interpretation of our experiments, but is not
intended as a reproducible specification. The generation code and the
random seeds used to produce these instances are not part of the public
distribution~\cite{roadef2026_instances}, which releases the instance files
only; readers wishing to regenerate them should contact the challenge
organizers. This does not affect the reproducibility of our experiments: all
runs reported here use the unmodified released instance files, listed in
\Cref{tab:instances}, and our own code and driver scripts are available
at~\cite{stela_alexa_code}.

 Each instance is composed of four main elements: a network topology, a traffic matrix, a set of time-dependent intervention scenarios, and operational parameters such as the planning horizon, the maximum number of segments, and the reconfiguration budget.

The network topologies used in the benchmark setA were produced using three complementary generation procedures. The first one is the COLD algorithm, for Complex Optimization of Link Design, which was designed to synthesize realistic PoP-level network topologies~\cite{bowden2014cold}. The second generator is Phoenix, a topology generation procedure based on generative artificial intelligence and implemented for the benchmark production pipeline. The third family corresponds to a random topology generator with probabilities weighted by router sizes. This latter procedure introduces heterogeneous connectivity patterns while keeping the generation process flexible across different network sizes.

For each generated topology, a traffic matrix was produced by selecting a set of source--destination demands. The number of demands varies across instances and is expressed as a percentage of the total number of ordered node pairs, typically ranging from less than $1\%$ to $8\%$ of $|V|^2$. For each demand, a source node, a destination node, and a traffic volume are selected. Two mechanisms were used to generate demands: uniform random sampling, and centrality-based sampling where routers with lower centrality are more likely to be selected as demand endpoints. In the latter case, traffic volumes are drawn from an exponential distribution, which produces heterogeneous traffic matrices with a few large demands and many smaller ones. To avoid trivial instances, several traffic matrices were generated for each topology and then filtered according to a difficulty indicator. A lower bound was computed by solving a multi-commodity flow relaxation, corresponding to an idealized routing model without segment-routing restrictions. An upper bound was obtained from a greedy heuristic producing a feasible segment-routing solution. The gap between these two values was used as a proxy for instance difficulty. Traffic matrices with larger gaps were preferred, since a zero gap would indicate that the greedy heuristic is already optimal for the considered relaxation.

The remaining parameters were chosen to reflect the setting of the first benchmark set. All Set A instances contain two time periods: the nominal period $t=0$ and one intervention period $t=1$. The maximum number of segments per path is set to six, which is consistent with typical values used in commercial routers. The reconfiguration budget is chosen slightly below the number of changes required by the heuristic solution obtained when periods are optimized independently. This design choice prevents the problem from decomposing into two independent single-period routing problems and forces algorithms to explicitly handle the temporal coupling induced by the budget constraint.

Intervention scenarios were generated by selecting links that are relevant for the current routing solution. In particular, the removed link is chosen among the highly loaded links under the heuristic routing, while ensuring that its removal does not disconnect the graph. Links that are not used in the previous time period are avoided, since their removal would have little impact on the routing scheme. This produces intervention scenarios that are both feasible from a connectivity viewpoint and meaningful from a traffic-engineering perspective.

This procedure was used to obtain a benchmark set in which the difficulty does not only come from the size of the topology or the number of demands, but also from the interaction between traffic concentration, scheduled link removals, and limited reconfiguration budgets.
\begin{table}[H]
    \centering
  
    \begin{tabular}{c c r r r c c}
        \toprule
        Name & Generator & Nodes & Arcs & Demands & Demands/Nodes$^2$ & Timesteps \\
        \midrule
        01 & YG & 20  & 80   & 40   & 10\%  & 2 \\
        02 & YG & 30  & 150  & 45   & 5\%   & 2 \\
        03 & YG & 50  & 250  & 20   & 0.8\% & 2 \\
        04 & YG & 50  & 250  & 200  & 8\%   & 2 \\
        05 & CO & 100 & 396  & 100  & 1\%   & 2 \\
        06 & PH & 100 & 500  & 500  & 5\%   & 2 \\
        07 & YG & 100 & 500  & 800  & 8\%   & 2 \\
        08 & CO & 150 & 654  & 200  & 0.9\% & 2 \\
        09 & YG & 150 & 750  & 200  & 0.9\% & 2 \\
        10 & PH & 150 & 966  & 1000 & 4.5\% & 2 \\
        11 & YG & 200 & 1000 & 400  & 1\%   & 2 \\
        12 & CO & 200 & 898  & 400  & 1\%   & 2 \\
        13 & PH & 200 & 1000 & 2000 & 5\%   & 2 \\
        14 & CO & 250 & 1108 & 600  & 1\%   & 2 \\
        15 & YG & 250 & 1250 & 600  & 1\%   & 2 \\
        16 & PH & 250 & 1452 & 4800 & 8\%   & 2 \\
        17 & CO & 300 & 1270 & 2000 & 2.2\% & 2 \\
        18 & YG & 300 & 1500 & 2000 & 2.2\% & 2 \\
        19 & PH & 300 & 1998 & 6000 & 6.6\% & 2 \\
        20 & YG & 400 & 2000 & 6000 & 3.7\% & 2 \\
        \bottomrule
    \end{tabular}
\caption{\small Characteristics of the setA benchmark instances.}
  \label{tab:instances}
\end{table}

\paragraph{SetAP:}
The \texttt{setAP} benchmark~\cite{stela_alexa_code} is derived from
\texttt{setA} through a randomized instance-generation procedure designed to
explore a broader range of problem sizes and traffic configurations: the
process randomly samples and combines the main characteristics of
\texttt{setA} instances while systematically varying the network size and
the number of demands. All instances use $|T|=2$ time periods, and the
number of arcs is fixed to five times the number of nodes, $|A|=5|V|$. The
generation grid consists of $|V|\in\{10,20,30,40,50\}$ and
$|D|\in\{10,20,30,40,50\}$, giving $25$ combinations, each replicated with
$5$ independent random seeds, for a total of $125$ instances. The size class
$N=|\calA_T|=10|V|$ used in the experimental figures is therefore a
derived quantity rather than an independent parameter: it is in bijection
with the number of nodes, with
$N\in\{100,200,300,400,500\}$ corresponding to
$|V|\in\{10,20,30,40,50\}$ respectively.
\paragraph{SetB:} realistic instances with 12 time steps, where interventions affect $1/500$ of the links at each time step. These instances are derived from eight obfuscated real-world topologies (COLD, Phenix, YG, IGEN, RAEI, OTI, RBCI, and IGN), resulting in 15 instances ranging from 200 to 800 nodes and from 1,000 to 4,000 links.

\paragraph{SetC:} a large-scale realistic instance released as a representative case study for future research on multi-period segment routing optimization.

%% file: sections/04-methodsexact.tex
\section{Compact Sequential exact lexicographic minimax Formulations}
\label{sec:milp}

\subsection{Motivation and related work}
The previous section introduced the T-ASR problem and its lexicographic load-balancing objective. Unlike the classical min--max objective, which only minimizes the maximum link utilization, the objective considered here minimizes the entire vector of link loads sorted in non-increasing order. Consequently, after minimizing the largest load, the second-largest load is minimized without degrading the first one, followed by the third-largest load, and so forth. This produces a more balanced distribution of traffic across the network.

However, this objective cannot be represented directly by a standard scalar linear objective. An exact solution method must therefore explicitly model the ordering of the loads or, equivalently, identify their optimal values one rank at a time.

This section develops three exact approaches for the lexicographic
problem~\eqref{ctn:objlex}, which we then compare. All three share the same
feasible set $\mathcal{F}$ and the same sorted load vector
$\lambda^{\downarrow}(P)$, and all three are \emph{sequential}: they optimize one
rank at a time and freeze the optimal value obtained before moving to the next.
They differ in how each rank is encoded and, consequently, in what is frozen from
one level to the next; the purpose of the section is to contrast them from both a
modeling and a computational standpoint.

The first method, \emph{ALEXA} (\emph{Assignment-based LEXicographic Algorithm}),
encodes the sorting explicitly through binary assignment variables that associate
each arc--time load with one position in the sorted load vector, here the assignment takes the form of an
$N\times N$ permutation matrix coupled to the routing variables. Proceeding rank by
rank, it freezes at each level the \emph{continuous} load value assigned to that
rank, while leaving the full permutation matrix free.

The second method, \textsc{Stela} freezes the previously certified
threshold values and re-imposes their exceedance-cardinality constraints.
The binary exceedance variables themselves remain free and may be reassigned;
the formulation therefore does not freeze which arc--time pairs realize the
earlier ranks.
The third method, \emph{\textsc{Carla}} (\emph{Cumulative Aggregation of Ranked Loads
Algorithm}), is an adaptation of the classical cumulative formulation of
lexicographic min--max optimization due to~\cite{ogryczak2006}. It rests on the fact that the sum of the $k$
largest entries of a vector admits a compact linear representation; the lexicographic minimax
optimum is reached by minimizing these cumulative sums $\Theta_1,\Theta_2,\dots$ in
turn, each optimal value being frozen before the next rank is addressed.

\subsection{Overview and Common Framework}
\label{sec:exact-overview}

The three exact methods developed in this section rely on the same feasible
routing set and optimize the same sorted load vector; they differ only in the
mixed-integer formulation solved at each rank. We therefore introduce these common
elements once, together with the auxiliary notation used in the correctness proofs,
before presenting the three formulations.

\paragraph{Feasible set and loads}
Recall that $\mathcal{F}$ denotes the set of routing schemes satisfying
Constraints~\eqref{ctn:flow}--\eqref{ctn:budget_3}, and assume
$\mathcal{F}\neq\varnothing$ (feasible instance).
Each routing scheme $P \in \mathcal{F}$ consists of finitely many segment
paths, each a sequence of at most $\texttt{maxSeg}$ pairwise distinct
segments drawn from the finite set $\{(i,j): i,j \in V,\, i \neq j\}$; hence
$\mathcal{F}$ is finite (and non-empty by assumption). A routing scheme
$P \in \mathcal{F}$ induces, for every demand $d \in D$, time step
$t \in T$, and distinct nodes $i,j \in V$, a value $x^{dt}_{ij} \in \{0,1\}$
equal to $1$ when demand $d$ is routed through segment $(i,j)$ at time $t$
and $0$ otherwise; we write $x(P) \in \{0,1\}^m$ for the vector collecting
these values under a fixed indexing of the coordinates over all
$(d,t,i,j)$. The map $P \mapsto x(P)$ need not be injective when nodes may
repeat, since two orderings of the same segment set induce the same vector;
this is immaterial here, as every quantity appearing in the model; the
segment count~\eqref{ctn:segments}, the arc loads below, and the
reconfiguration distance~\eqref{ctn:budget_3}; is a function of $x(P)$
alone, so that routing schemes with the same image are feasible for exactly
the same instances and attain the same objective value. Making the
dependence on $P$ explicit through $x^{dt}_{ij}$, the load $\lambda(a,t)$ of
an available arc--time pair $(a,t)\in\mathcal{A}_T$ introduced in
Section~\ref{sec:notations} reads
\begin{equation}
\lambda(a,t;P)=\frac{1}{c(a)}
\sum_{d\in D}\sum_{i,j\in V}
r(i,j,a,t)\,\nu(d,t)\,x^{dt}_{ij}, \label{eq:load-P}
\end{equation}
and we use $\lambda(a,t)$ and $\lambda(a,t;P)$ interchangeably, keeping the
explicit form wherever the dependence on $P$ matters.

\paragraph{Sequential decomposition}
The lexicographic objective~\eqref{ctn:objlex} cannot be written as a single linear
function of $x$, since the ordering of the loads depends on $P$ itself
(see~\eqref{eq:sorting-perm}). We therefore decompose it into a sequence of
single-objective problems, each certifying one rank of the sorted load vector while
the previously certified ranks are held fixed. Setting $\mathcal{F}_0=\mathcal{F}$, assuming $\mathcal F \neq \varnothing$ (feasible instance).
we define for each $k\geq1$
\begin{align}
  \bar{\lambda}_k &= \min_{P\in\mathcal{F}_{k-1}} \lambda^{\downarrow}_k(P),
    \label{eq:lex-recursion1}\\
  \mathcal{F}_k &= \bigl\{P\in\mathcal{F}_{k-1} :
    \lambda^{\downarrow}_k(P)=\bar{\lambda}_k\bigr\},
    \label{eq:lex-recursion2}
\end{align}
which yields a nested chain
$\mathcal{F}=\mathcal{F}_0\supseteq\mathcal{F}_1\supseteq\cdots\supseteq\mathcal{F}_N$.

\begin{figure}[h!]
  \centering
  \begin{tikzpicture}[>=Stealth, line join=round, line cap=round]

    \def\ptdot{0.7pt}

    \draw[fill=black!5,  draw=black!55, thick]
      plot[smooth cycle,tension=0.5]
      coordinates {(-3.0,0.4)(-2.1,1.7)(-0.3,2.05)(1.7,1.6)(2.75,0.35)
                   (2.5,-1.15)(1.0,-1.95)(-1.0,-1.85)(-2.6,-1.0)};
    \draw[fill=black!12, draw=black!55, thick]
      plot[smooth cycle,tension=0.75]
      coordinates {(-2.15,0.35)(-1.45,1.35)(-0.1,1.55)(1.25,1.15)(1.95,0.2)
                   (1.7,-0.95)(0.6,-1.4)(-0.75,-1.3)(-1.85,-0.7)};
    \draw[fill=black!22, draw=black!55, thick]
      plot[smooth cycle,tension=0.75]
      coordinates {(-1.35,0.3)(-0.85,1.0)(0.15,1.05)(0.95,0.6)(1.15,-0.25)
                   (0.75,-0.95)(-0.15,-1.05)(-1.05,-0.55)};
    \draw[fill=black!34, draw=black!60, thick]
      plot[smooth cycle,tension=0.8]
      coordinates {(-0.65,0.2)(-0.25,0.6)(0.35,0.55)(0.6,0.05)(0.4,-0.5)(-0.25,-0.55)(-0.6,-0.2)};

    \foreach \p in {(-2.55,0.9),(-1.9,1.55),(1.9,1.15),(2.35,-0.4),(1.55,-1.55),
                    (-1.5,-1.5),(-2.7,-0.4),(0.7,1.75),(2.55,0.6)}
      \fill[black!75] \p circle (\ptdot);
    \foreach \p in {(-1.75,0.75),(1.35,0.75),(1.45,-0.85),(-1.0,-1.05),(0.1,1.3),(-1.55,-0.3)}
      \fill[black!75] \p circle (\ptdot);
    \foreach \p in {(-0.95,0.5),(0.7,0.3),(0.55,-0.7),(-0.6,-0.75),(0.25,0.85)}
      \fill[black!75] \p circle (\ptdot);
    \foreach \p in {(-0.35,0.25),(0.3,0.2),(0.05,-0.35)}
      \fill[black!75] \p circle (\ptdot);
    \node[star,star points=5,star point ratio=2.3,inner sep=0pt,minimum size=9pt,
          fill=black,draw=black] (opt) at (0.0,-0.02) {};

    \node[right] (l0) at (3.45,1.7) {$\mathcal{F}_0=\mathcal{F}$};
    \draw[black!50,thin] (l0.west) -- (2.35,1.35);
    \node[right] (l1) at (3.45,0.7) {$\mathcal{F}_1$};
    \draw[black!50,thin] (l1.west) -- (1.75,0.6);
    \node[right] (l2) at (3.45,-0.3) {$\mathcal{F}_2$};
    \draw[black!50,thin] (l2.west) -- (1.14,-0.22);
    \node[right] (lN) at (3.45,-1.3) {$\mathcal{F}_N$};
    \draw[black!50,thin] (lN.west) -- (0.58,-0.18);

    \node[below=3pt] (ps) at (-0.9,-1.95) {\small $P^\star$ lexicographic minimax-optimal };
    \draw[black!50,thin] (ps.north) -- (opt);

    \begin{scope}[xshift=6.5cm, yshift=-1.5cm]
      \def\bw{0.40}
      \draw[->,thick] (-0.45,0) -- (5.55,0) node[below right=-1pt]{\small rank $k$};
      \draw[->,thick] (-0.45,0) -- (-0.45,4.05) node[above]{\small $\lambda^{\downarrow}_k$};

      \foreach \cx/\h/\ff/\lab in
        {0.25/3.35/{black!12}/{$\mathcal{F}_1$},
         1.15/2.90/{black!22}/{$\mathcal{F}_2$},
         2.05/2.55/{black!30}/{$\mathcal{F}_3$}}
        {\draw[fill=\ff,draw=black!70,thick] (\cx-\bw,0) rectangle (\cx+\bw,\h);
         \node at (\cx,0.45) {\small \lab};}
      \foreach \cx/\h in {2.95/2.05, 3.85/1.65, 4.75/1.30}
        \draw[fill=white,draw=black!70,thick] (\cx-\bw,0) rectangle (\cx+\bw,\h);

      \node[above=1pt] at (0.25,3.35) {\small $\bar{\lambda}_1$};
      \node[above=1pt] at (1.15,2.90) {\small $\bar{\lambda}_2$};
      \node[above=1pt] at (2.05,2.55) {\small $\bar{\lambda}_3$};

      \draw[dashed,black!70] (2.50,-0.05) -- (2.50,3.75);
      \draw[->,thick] (2.65,3.45) .. controls (3.10,3.72) .. (3.38,3.42)
        node[right,align=left,inner sep=1.5pt]{\small solve\\[-1pt]\small rank $k{+}1$};

      \draw[decorate,decoration={brace,amplitude=6pt,mirror},black!70]
        (-0.20,-0.85) -- (2.48,-0.85)
        node[midway,below=7pt,align=center]{\small certified \& frozen\\[-1pt]\small
        $\lambda^{\downarrow}_j=\bar\lambda_j,\ j\le k$};
      \draw[decorate,decoration={brace,amplitude=6pt,mirror},black!70]
        (2.52,-0.85) -- (5.20,-0.85)
        node[midway,below=7pt,align=center]{\small still optimized};
    \end{scope}

  \end{tikzpicture}
  \caption{Sequential lexicographic minimax decomposition. \textbf{Left:} the nested chain
    $\mathcal{F}_0\supseteq\mathcal{F}_1\supseteq\cdots\supseteq\mathcal{F}_N$, contracting
    toward a lexicographic minimax optimum $P^\star$; regions are schematic ($\mathcal{F}$ is finite and,
    the sorting permutation being endogenous, non-polyhedral in load space).
    \textbf{Right:} the induced freezing of the sorted load vector. Certifying rank $k$
    moves the iterate from $\mathcal{F}_{k-1}$ to $\mathcal{F}_k$; each frozen bar is shaded
    with the colour of its set $\mathcal{F}_k$ and fixed at
    $\bar{\lambda}_k=\min_{\mathcal{F}_{k-1}}\lambda^{\downarrow}_k$, while later ranks (hollow,
    no set certified yet) are still optimized. The procedure ends at $\mathcal{F}_N$.}
  \label{fig:lex-decomposition}
\end{figure}
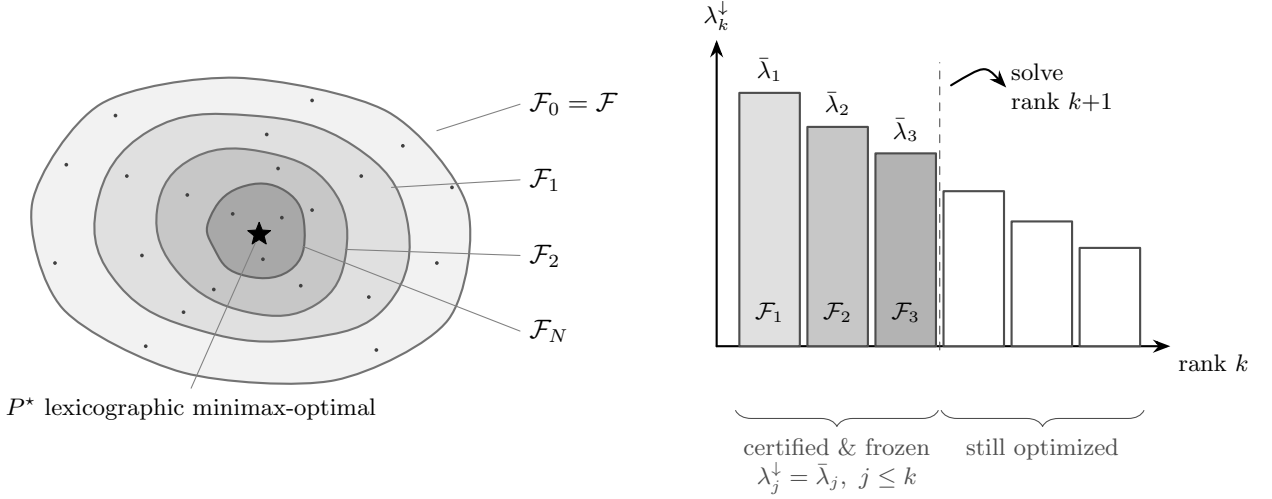
This kind of sequential scheme is standard in lexicographic
optimization. In particular,~\cite{abernethy2024} analyze the generic lexicographic \emph{maximization} of an
arbitrary set $X\subseteq\mathbb{R}^n$, where iteration $k$ maximizes the $k$-th
smallest component while fixing the previously optimized ones, and establish that
the procedure returns exactly the lexicographic optima of $X$. We use the
order-reversed (lexicographic minimax) counterpart; minimizing the $k$-th largest load rather
than maximizing the $k$-th smallest component: which is equivalent up to negation.
Since $\mathcal{F}\subseteq\{0,1\}^m$ is finite and non-empty, a lexicographic
minimizer exists by~\cite{abernethy2024}, and a scheme $P^\star\in\mathcal{F}$ is lexicographic minimax-optimal for T-ASR
if and only if $P^\star\in\mathcal{F}_N$.

\paragraph{Rank subproblems}

The values $(\bar{\lambda}_k)_{k\in\{1,\ldots,N\}}$ in
\eqref{eq:lex-recursion1} are defined, but the procedure for obtaining
them is not specified. We call $\mathcal{M}_k$ any mixed-integer linear
formulation whose resolution \emph{certifies rank $k$}, in the sense that
its optimal solution yields a scheme $P^\star_k\in\mathcal{F}_k$
together with the value $\bar{\lambda}_k$.

We write $\textsc{Solve}(\mathcal{M}_k)$ for a single call to a MILP
solver on $\mathcal{M}_k$, returning the pair
$(\bar{\lambda}_k,P^\star_k)$. Once $\mathcal{M}_k$ has been proved to solve the rank-$k$ subproblem
\eqref{eq:lex-recursion1}, the value $\bar{\lambda}_k$ may be recovered from
the returned optimal scheme by evaluating and sorting its loads according
to~\eqref{eq:load-P}. This evaluation recovers the attained rank value; its
optimality follows from the correctness theorem for $\mathcal{M}_k$.

Algorithm~\ref{alg:lex-template} drives the three methods (\textsc{Alexa}, \textsc{Stela}, \textsc{Carla}) through the common chain:
it solves $\mathcal{M}_1,\dots,\mathcal{M}_N$ in turn and returns a lexicographic minimax-optimal
scheme after at most $N$ solves~\cite{abernethy2024}. The three methods differ only
in the formulation $\mathcal{M}_k$: the assignment model
$\mathcal{M}_k^{\mathrm{\textsc{Alexa}}}$ (Section~\ref{sec:alexa}), the
threshold-exceedance model $\mathcal{M}_k^{\mathrm{\textsc{Stela}}}$
(Section~\ref{sec:stela}) and the cumulative model
$\mathcal{M}_k^{\mathrm{\textsc{Carla}}}$ (Section~\ref{sec:carla}). The validity of each; that solving $\mathcal{M}_k$ indeed
certifies $\bar{\lambda}_k$ over $\mathcal{F}_{k-1}$; is established by the lemmas of
the corresponding section.

The two methods \textsc{Alexa} and \textsc{Stela}, share more than the sequential framework: both instantiate the rank subproblem $\mathcal{M}_k$ by
\emph{directly} encoding the $k$-th largest load $\lambda^{\downarrow}_k$ as the
quantity minimized at level $k$. They differ only in \emph{how} they encode the
sort that singles out this rank; \textsc{Alexa} through an explicit $N\times N$
permutation matrix, \textsc{Stela} through threshold-exceedance counters; and,
consequently, in what each level freezes: a continuous rank-load with the
permutation left free for \textsc{Alexa}, a binary exceedance pattern for \textsc{Stela}. We
present these two direct encodings first.

We then turn to \emph{\textsc{Carla}}, an adaptation of the classical cumulative
formulation of~\cite{ogryczak2006}, which reaches the
\emph{same} nested chain of restrictions by a different route: instead of
$\lambda^{\downarrow}_k$, it minimizes the cumulative sum $\Theta_k$ of the $k$
largest loads, whose partial-sum structure admits a compact linear
representation. 

For $\mathsf{X}\in\{\textsc{Alexa},\textsc{Stela},\textsc{Carla}\}$, we denote by
$\mathcal{M}_k^\mathsf{X}$ the formulation used by method $\mathsf{X}$ to certify rank $k$. The three methods differ only in the formulation $\mathcal{M}_k^\mathsf{X}$:
the assignment model $\mathcal{M}_k^{\mathrm{\textsc{Alexa}}}$,
the threshold-exceedance model $\mathcal{M}_k^{\mathrm{\textsc{Stela}}}$
and the cumulative model $\mathcal{M}_k^{\mathrm{\textsc{Carla}}}$.

\begin{algorithm}[ht]
\caption{Generic sequential lexicographic minimization}
\label{alg:lex-template}
\begin{algorithmic}[1]

\Input Instance $(G,D,T,q,\kappa,\maxSeg)$; method
$\mathsf{X}\in\{\textsc{Alexa},\textsc{Stela},\textsc{Carla}\}$;
a family $\{\mathcal{M}_k^\mathsf{X}\}_{k=1}^{N}$

\Output lexicographic minimax-optimal routing $P^\star$ and
$(\bar{\lambda}_1,\dots,\bar{\lambda}_N)$

\For{$k \leftarrow 1$ \textbf{to} $N$}

\State $(\bar{\lambda}_k, P^\star_k) \leftarrow
\textsc{Solve}(\mathcal{M}_k^X)$
\Comment{certify rank $k$ over $\mathcal{F}_{k-1}$}

\If{$\bar{\lambda}_k = 0$}

\State \Return $P^\star_k,\,
(\bar{\lambda}_1,\dots,\bar{\lambda}_k,0,\dots,0)$
\Comment{$N-k$ trailing zeros}

\EndIf

\EndFor

\State \Return $P^\star_N,\,
(\bar{\lambda}_1,\dots,\bar{\lambda}_N)$

\end{algorithmic}
\end{algorithm}

\Cref{alg:lex-template} presents the generic sequential lexicographic
procedure that we propose for $T$-\srchallenge{}. At iteration $k$, given
a method $\mathsf{X}\in\{\textsc{Alexa},\textsc{Stela},\textsc{Carla}\}$ and its
corresponding MILP formulation $\mathcal{M}_k^X$, the procedure returns
the optimal value
\[
\bar{\lambda}_k
=
\min_{P\in\mathcal{F}_{k-1}}
\lambda^{\downarrow}_k(P),
\]
along with a corresponding minimizer. The procedure instantiates the
same sequential scheme for the three methods and returns a
lexicographically optimal routing scheme after at most $N$ MILP
solves~\cite{abernethy2024}. In what follows, we show that the validity
of each formulation $\mathcal{M}_k^\mathsf{X}$, i.e., that it indeed optimizes
$\lambda^{\downarrow}_k$ over $\mathcal{F}_{k-1}$, is established by the
lemmas in \Cref{sec:alexa,sec:stela,sec:carla}. In particular, we
instantiate the generic subproblem $\mathcal{M}_k$ through three
alternative formulations:
$\mathcal{M}_k^{\mathrm{\textsc{Alexa}}}$,
$\mathcal{M}_k^{\mathrm{\textsc{Stela}}}$, and
$\mathcal{M}_k^{\mathrm{\textsc{Carla}}}$, which are detailed in the
following sections.


\subsection{\textsc{ALEXA}: Assignment-based Lexicographic Algorithm}
\label{sec:alexa}

The difficulty in optimizing the sorted load vector is that the ordering it
refers to is not known before the routing is chosen: the permutation that sorts
the loads is endogenous (see~\eqref{eq:sorting-perm}). \textsc{Alexa} confronts this
difficulty head-on by \emph{materializing the sort as a decision variable}.
Rather than computing the sorted vector \emph{a posteriori} from a given routing,
it lets the model choose an assignment between arc--time pairs and rank positions,
and constrains this assignment to be exactly the sorting permutation of the
induced loads. 

Concretely, \textsc{Alexa} encodes the full ordering in a single MILP. Let
$\mathcal{K}=\{1,\ldots,N\}$ be the set of ranks. We introduce a
\emph{permutation variable} $\pi_{a,t,r}\in\{0,1\}$, equal to $1$ iff the load
$\lambda(a,t;P)$ is placed at rank $r$, and \emph{rank-load variables}
$L_r\in\mathbb{R}_{+}$ holding the $r$-th largest load. The pair $(\pi,L)$ thus
carries both the sort (through $\pi$) and its sorted values (through $L$), while
$P$ carries the routing. The model $\mathcal{M}_k^{\mathrm{\textsc{Alexa}}}$ is
\begin{subequations}
\label{mod:alexa}
\small
\begin{align}
\min_{P\in\mathcal{F},\,\pi,\,L}\ & L_k
  \label{alexa:obj}\\
\text{s.t.}\quad
& \sum_{r\in\mathcal{K}}\pi_{a,t,r}=1,
  \quad \forall (a,t)\in\mathcal{A}_T,
  \label{alexa:row}\\
& \sum_{(a,t)\in\mathcal{A}_T}\pi_{a,t,r}=1,
  \quad \forall r\in\mathcal{K},
  \label{alexa:col}\\
& L_r-\lambda(a,t;P)\le M\,(1-\pi_{a,t,r}),
  \ \ \forall (a,t),\,r,
  \label{alexa:link1}\\
& \lambda(a,t;P)-L_r\le M\,(1-\pi_{a,t,r}),
  \ \ \forall (a,t),\,r,
  \label{alexa:link2}\\
& L_{r-1}\ge L_{r},
  \quad r=2,\ldots,N,
  \label{alexa:sort}\\
& L_j=\bar{\lambda}_j,
  \quad j=1,\ldots,k-1,
  \label{alexa:freeze}\\
& \pi_{a,t,r}\in\{0,1\},\ L_r\ge 0,
  \quad \forall (a,t),\,r.
  \label{alexa:bin}
\end{align}
\end{subequations}

We read the model block by block.
Constraints~\eqref{alexa:row}--\eqref{alexa:col} are the two families of an
assignment problem: \eqref{alexa:row} states that each arc--time pair receives
exactly one rank, and \eqref{alexa:col} that each rank is occupied by exactly one
pair. Together with integrality~\eqref{alexa:bin}, they force $(\pi_{a,t,r})$ to
be a permutation matrix on $\mathcal{A}_T\times\mathcal{K}$; equivalently, $\pi$
encodes a bijection $\sigma:\mathcal{A}_T\to\mathcal{K}$.
The linking inequalities~\eqref{alexa:link1}--\eqref{alexa:link2} then tie the
chosen rank of each pair to its actual load. When $\pi_{a,t,r}=1$, both
right-hand sides vanish and the two inequalities collapse to the single equality
$L_r=\lambda(a,t;P)$; when $\pi_{a,t,r}=0$, both right-hand sides equal $M$ and,
provided $M$ is large enough (Proposition~\ref{prop:alexa-bigm}), leave the pair
$(a,t)$ and the rank $r$ mutually unconstrained. Hence only the diagonal
entries $r=\sigma(a,t)$ are active, and $(L_r)_r$ is a rearrangement of the loads
of $P$.

The role of the ordering constraint~\eqref{alexa:sort} deserves emphasis, as it
is what turns an \emph{arbitrary} assignment into \emph{the} sorting permutation.
Without it, any bijection $\sigma$ would be feasible and $(L_r)_r$ could be any
permutation of the loads; \eqref{alexa:sort} forces $L_1\ge L_2\ge\cdots\ge L_N$,
which singles out the unique non-increasing rearrangement, i.e.
$L_r=\lambda^{\downarrow}_{r}(P)$. Note that this holds \emph{regardless of how
ties are broken}: when several loads are equal, many matrices $\pi$ realize the
same sorted vector $(L_r)_r$, but the values $L_r$ are unaffected; the
non-uniqueness of $\sigma$ under ties does not affect the objective.
Finally, the freezing constraints~\eqref{alexa:freeze} fix the ranks already
certified at previous levels to the constants $\bar{\lambda}_j$, restricting the
search to $\mathcal{F}_{k-1}$.

A single family of constraints thus covers every $k\in\mathcal{K}$: only the
objective index in~\eqref{alexa:obj} and the freezing range
in~\eqref{alexa:freeze} change from one level to the next, while the assignment,
linking and ordering blocks stay identical. At $k=1$ the freezing range is empty,
the objective reduces to $L_1=\lambda^{\downarrow}_{1}(P)$, and no constant from a previous level is needed. In practice, this
means the whole lexicographic minimax computation can be driven on a single model, updated in
place across iterations by appending one equality~\eqref{alexa:freeze} and
shifting the objective to $L_{k+1}$, which lets the solver reuse the search tree
of the previous level.

\paragraph{A valid big-\texorpdfstring{$M$}{M} constant.}
The linking~\eqref{alexa:link1}--\eqref{alexa:link2} is only correct if the
inactive branch ($\pi_{a,t,r}=0$) genuinely imposes nothing, that is, if $M$
dominates every load that can arise. We use a data-driven bound derived from the traffic
volumes and capacities, valid in every regime.

\begin{proposition}
\label{prop:alexa-bigm} 
The constant
\[
  M=\max_{(a,t)\in\mathcal{A}_T}\frac{\maxSeg}{c(a)}
  \bigl(\max_{i\neq j} r(i,j,a,t)\bigr)\sum_{d\in D}\nu(d,t)
\]
satisfies $0\le\lambda(a,t;P)\le M$ for all $(a,t)\in\mathcal{A}_T$ and
$P\in\mathcal{F}$.
\end{proposition}
\begin{proof}
\[
\begin{array}{:c}
\begin{minipage}{0.97\textwidth}
Non-negativity is immediate from the load definition~\eqref{eq:load-P}, since all
its terms are non-negative. For the upper bound, fix $(a,t)$ and $P\in\mathcal{F}$.
The segment-count constraint~\eqref{ctn:segments} allows at most $\maxSeg$
non-zero variables $x^{dt}_{ij}$ per demand $d$ and time $t$, hence at most
$\maxSeg$ non-zero terms in the inner sum of~\eqref{eq:load-P} for that demand.
Each such term is at most $(\max_{i\neq j}r(i,j,a,t))\,\nu(d,t)$, because
$r(\cdot)$ is a split coefficient. Summing over $d\in D$ and dividing by $c(a)$
yields $\lambda(a,t;P)\le \frac{\maxSeg}{c(a)}(\max_{i\neq
j}r(i,j,a,t))\sum_{d\in D}\nu(d,t)$, and taking the maximum over
$(a,t)\in\mathcal{A}_T$ gives the stated $M$.
\end{minipage}
\end{array}
\]
\end{proof}

\paragraph{Correctness.}
We now establish that solving $\mathcal{M}_k^{\mathrm{\textsc{Alexa}}}$ certifies rank $k$.
The argument separates the sorting mechanism (Lemma~\ref{lem:alexa-sort}, proved
without freezing) from the effect of the freezing constraints
(Theorem~\ref{thm:alexa-correctness}).

\begin{lemma}
\label{lem:alexa-sort} 
Fix $P\in\mathcal{F}$ and consider $\mathcal{M}_k^{\mathrm{\textsc{Alexa}}}$ without the
freezing constraints~\eqref{alexa:freeze}. Let $x$ be the routing variables
induced by $P$.
\begin{enumerate}
  \item[(i)]  Every feasible solution $(x,\pi,L)$ satisfies
    $L_r=\lambda^{\downarrow}_{r}(P)$ for all $r\in\mathcal{K}$.
  \item[(ii)] Every $P\in\mathcal{F}$ extends to a
    feasible solution $(x,\pi,L)$.
\end{enumerate}
\end{lemma}

\begin{proof}
\[
\begin{array}{:c}
\begin{minipage}{0.97\textwidth}
(i) By~\eqref{alexa:row}--\eqref{alexa:col} and integrality~\eqref{alexa:bin},
$\pi$ is a permutation matrix, defining a bijection
$\sigma:\mathcal{A}_T\to\mathcal{K}$ with $\pi_{a,t,\sigma(a,t)}=1$. Consider a
pair $(a,t)$ and the rank $r=\sigma(a,t)$: the two linking
inequalities~\eqref{alexa:link1}--\eqref{alexa:link2} have zero right-hand side
and reduce to $L_{\sigma(a,t)}=\lambda(a,t;P)$. For every other rank
$r'\neq\sigma(a,t)$ we have $\pi_{a,t,r'}=0$, so~\eqref{alexa:link1}--\eqref{alexa:link2}
read $|L_{r'}-\lambda(a,t;P)|\le M$, which holds trivially since both quantities
lie in $[0,M]$ by Proposition~\ref{prop:alexa-bigm}; these constraints are thus
inactive. Consequently $(L_r)_r$ is a rearrangement of the loads of $P$, and the
ordering constraint~\eqref{alexa:sort} makes it non-increasing, so
$L_r=\lambda^{\downarrow}_{r}(P)$ for all $r$.

(ii) Given $P\in\mathcal{F}$, pick any sorting bijection $\sigma$ and set $\pi_{a,t,r}=\mathbf{1}[r=\sigma(a,t)]$ and
$L_r=\lambda^{\downarrow}_{r}(P)$. Then
\eqref{alexa:row}--\eqref{alexa:col} hold because $\sigma$ is a bijection, and
\eqref{alexa:sort} holds because $(L_r)_r$ is sorted by construction. The
linking~\eqref{alexa:link1}--\eqref{alexa:link2} holds with equality when
$\pi_{a,t,r}=1$ and with slack at most $M$ otherwise, since
$L_r,\lambda(a,t;P)\in[0,M]$ by Proposition~\ref{prop:alexa-bigm}. Hence
$(x(P),\pi,L)$ is feasible.
\end{minipage}
\end{array}
\]
\end{proof}

\begin{theorem}
\label{thm:alexa-correctness}
For every $k\in\mathcal{K}$, $\mathcal{M}_k^{\mathrm{\textsc{Alexa}}}$
satisfies~\eqref{eq:lex-recursion1}: its optimal value is
$\bar{\lambda}_k=\min_{P\in\mathcal{F}_{k-1}}\lambda^{\downarrow}_{k}(P)$, and
the scheme $P^\star_k$ read from any optimal solution belongs to $\mathcal{F}_k$.
\end{theorem}
\begin{proof}
\[
\begin{array}{:c}
\begin{minipage}{0.97\textwidth}
By Lemma~\ref{lem:alexa-sort}, the feasible set of $\mathcal{M}_k^{\mathrm{\textsc{Alexa}}}$
without freezing projects onto
$\{(P,\lambda^{\downarrow}(P)):P\in\mathcal{F}\}$: every feasible triple has
$L_r=\lambda^{\downarrow}_{r}(P)$, and every $P\in\mathcal{F}$ is attained. In
particular $L_j=\lambda^{\downarrow}_{j}(P)$ for all $j$, so the freezing
constraints~\eqref{alexa:freeze} hold if and only if
$\lambda^{\downarrow}_{j}(P)=\bar{\lambda}_j$ for all $j<k$; by the definition of
the chain~\eqref{eq:lex-recursion1}--\eqref{eq:lex-recursion2}, this is exactly $P\in\mathcal{F}_{k-1}$. On
that restricted set the objective~\eqref{alexa:obj} satisfies
$L_k=\lambda^{\downarrow}_{k}(P)$, hence
$\min L_k=\min_{P\in\mathcal{F}_{k-1}}\lambda^{\downarrow}_{k}(P)=\bar{\lambda}_k$.
At an optimal solution we therefore have $P^\star_k\in\mathcal{F}_{k-1}$ and
$\lambda^{\downarrow}_{k}(P^\star_k)=L_k=\bar{\lambda}_k$, which is precisely
$P^\star_k\in\mathcal{F}_k$.
\end{minipage}
\end{array}
\]
\end{proof}

With correctness established, \textsc{Alexa} needs no dedicated driver: it is exactly the
generic procedure of \Cref{alg:lex-template} run on the family
$\{\mathcal{M}_k^{\mathrm{\textsc{Alexa}}}\}_{k=1}^{N}$. At level $k$, the model
\eqref{mod:alexa} is built, $\textsc{Solve}(\mathcal{M}_k^{\mathrm{\textsc{Alexa}}})$
returns $(\bar{\lambda}_k,P^\star_k)$, the freezing equality
$L_k=\bar{\lambda}_k$ is appended, and the objective is shifted to $L_{k+1}$;
at $k=1$ the freezing range is empty and the objective reduces to
$L_1=\lambda^{\downarrow}_{1}(P)$. By \Cref{thm:alexa-correctness} each solve
certifies its rank over $\mathcal{F}_{k-1}$, so the loop returns a
lexicographic minimax-optimal scheme after at most $N$ solves~\cite{abernethy2024}.

\subsection{\textsc{STELA}: A Sequential Threshold-Exceedance Algorithm}
\label{sec:stela}

\textsc{Alexa} pays a high price for materializing the sort: it maintains an $N\times N$
permutation matrix that must be re-optimized in full at every level. \textsc{Stela} avoids this by observing that certifying a
single rank does not require knowing the entire ordering. it only requires
knowing \emph{how many} loads exceed a candidate value. This shift from
\texttt{which pair sits at rank $k$} to \texttt{how many pairs exceed a threshold} replaces
the global assignment by a simple counting mechanism, and is the source of
\textsc{Stela}'s lighter per-level encoding.

Concretely, at iteration $k$ \textsc{Stela} computes the smallest threshold exceeded by at
most $k-1$ loads; this threshold equals $\lambda^{\downarrow}_{k}(P)$, while the
previously optimized ranks are preserved. For $P\in\mathcal{F}$ and $z\ge 0$, let
$E_P(z)=\{(a,t)\in\mathcal{A}_T:\lambda(a,t;P)>z\}$ denote the set of arc--time
pairs whose load strictly exceeds $z$. The following elementary observation is
the backbone of the method.

\begin{proposition}
\label{prop:stela-threshold}
For every $P\in\mathcal{F}$ and $k\in\mathcal{K}$,
$\lambda^{\downarrow}_{k}(P)=\min\{z\ge 0:\lvert E_P(z)\rvert\le k-1\}$.
\end{proposition}
\begin{proof}
\[
\begin{array}{:c}
\begin{minipage}{0.97\textwidth}
For $z\ge\lambda^{\downarrow}_{k}(P)$, only the ranks $1,\ldots,k-1$ can exceed
$z$, so $\lvert E_P(z)\rvert\le k-1$; for $z<\lambda^{\downarrow}_{k}(P)$, the $k$
largest loads all exceed $z$, so $\lvert E_P(z)\rvert\ge k$. Hence
$\lambda^{\downarrow}_{k}(P)$ is the smallest threshold whose exceedance set has
at most $k-1$ elements.
\end{minipage}
\end{array}
\]
\end{proof}

Proposition~\ref{prop:stela-threshold} rewrite the $k$-th largest load as the
optimal value of a small optimization problem over $z$: push the threshold down
as far as possible, subject to no more than $k-1$ loads exceeding it. It remains
to express the counting constraint $\lvert E_P(z)\rvert\le k-1$ in mixed-integer
form.

\paragraph{Exceedance variables.}
To turn the threshold characterization of
Proposition~\ref{prop:stela-threshold} into a MILP, we introduce, for each
level $j\in\{2,\ldots,N\}$ and each arc--time pair $(a,t)\in\mathcal{A}_T$, a
binary variable $y^{(j)}_{a,t}\in\{0,1\}$ acting as an \emph{exceedance
allowance}. For a threshold value $z$, it is linked to the load by
\begin{equation}
\label{bilinear constraint}
  \lambda(a,t;P) \;\leq\; z + \bigl(\bar{\lambda}_1 - z\bigr)\,y_{a,t}^{(j)},
\end{equation}
so that $y^{(j)}_{a,t}=0$ forces $\lambda(a,t;P)\le z$, while a load strictly
above $z$ forces $y^{(j)}_{a,t}=1$. Here $\bar\lambda_1=\lambda^{\downarrow}_1$
is the optimal value of level~$1$; being the largest load of the level-$1$
optimum, it upper-bounds every load of any solution feasible at
levels~$2,\ldots,N$, which makes~\eqref{bilinear constraint} valid; it cuts
off no admissible routing. The coefficient $\bar\lambda_1-z$ is a
big-$M$ compatible with this validity: when $y^{(j)}_{a,t}=1$ the right-hand
side is exactly $\bar\lambda_1$, so the constraint is vacuous but no looser
than necessary; when $y^{(j)}_{a,t}=0$ it reduces to $\lambda(a,t;P)\le z$.
Any larger coefficient would only weaken the linear relaxation.

Bounding the number of exceedances by the cardinality constraint
\begin{equation}
\label{eq:stela-card}
  \sum_{(a,t)\in\mathcal{A}_T} y_{a,t}^{(j)} \;\leq\; j-1
\end{equation}
then allows at most $j-1$ loads to exceed the threshold, which by
Proposition~\ref{prop:stela-threshold} pins it to
$\lambda^{\downarrow}_{j}(P)$.

The linking is deliberately one-sided: a load below the threshold may still
carry $y^{(j)}_{a,t}=1$, so the variables do not characterize exceedance
exactly. This is immaterial. The variable $y^{(j)}_{a,t}$ occurs only
in~\eqref{bilinear constraint} and in the cardinality
constraint~\eqref{eq:stela-card}; lowering a spurious $y^{(j)}_{a,t}=1$ to
$0$ tightens only its own constraint~\eqref{bilinear constraint}; which is
already satisfied, the load being below $z$; while relaxing the cardinality
budget. Any feasible point thus remains feasible after the reset, with
unchanged loads, and the projection onto the routing variables is the intended
one.

Constraint~\eqref{bilinear constraint} is linear when $z$ is a constant, a
\emph{frozen} level $z=\bar{\lambda}_j$, but becomes bilinear through the
product $z\,y_{a,t}^{(j)}$ when $z$ is itself an optimization variable, as at
the current level. The two regimes are handled separately below, and the
current-level bilinearity is removed in
Remark~\ref{rem:stela-linearization}.
\paragraph{Level 1.}
The first subproblem $\mathcal{M}_1^{\mathrm{\textsc{Stela}}}$ is writen as:
\begin{equation}
  \bar{\lambda}_1=\min_{P\in\mathcal{F},\,z\ge0}
  \{z:\lambda(a,t;P)\le z,\ \forall(a,t)\in\mathcal{A}_T\},
  \label{eq:stela-l1}
\end{equation}
which needs no exceedance variables: forcing $z$ above \emph{every} load makes it
an upper bound on the maximum, and minimizing $z$ pins it to
$\lambda^{\downarrow}_{1}(P)$. Its value $\bar{\lambda}_1$ upper-bounds every load
at all later levels, and this is precisely the constant that will parametrize the
threshold constraints from level $2$ onward.

\paragraph{Level $k\ge 2$.}
Given the frozen values $\bar{\lambda}_1,\ldots,\bar{\lambda}_{k-1}$ certified at
the earlier levels, the subproblem $\mathcal{M}_k^{\mathrm{\textsc{Stela}}}$ is
\begin{subequations}
\small
\label{mod:stela}
\begin{align}
\min_{P\in\mathcal{F},\,z\ge 0,\,y}\ & z
  \label{stela:obj}\\
\text{s.t.}\quad
& \lambda(a,t;P)\le \bar{\lambda}_1,\quad \forall (a,t),
  \label{stela:gb}\\
& \lambda(a,t;P)\le \bar{\lambda}_j+(\bar{\lambda}_1-\bar{\lambda}_j)\,y^{(j)}_{a,t},
  \ \forall (a,t),\ 2\le j<k,
  \label{stela:freeze-t}\\
& \sum_{(a,t)} y^{(j)}_{a,t}\le j-1,\quad 2\le j<k,
  \label{stela:freeze-c}\\
& \lambda(a,t;P)\le z+\bar{\lambda}_1\,y^{(k)}_{a,t},\quad \forall (a,t),
  \label{stela:cur-t}\\
& \sum_{(a,t)} y^{(k)}_{a,t}\le k-1,
  \label{stela:cur-c}\\
& 0\le z\le \bar{\lambda}_{k-1},
  \label{stela:mon}\\
& y^{(j)}_{a,t}\in\{0,1\},\quad \forall (a,t),\ 2\le j\le k.
  \label{stela:bin}
\end{align}
\end{subequations}

We read the model block by block.
The global bound~\eqref{stela:gb} imposes $\lambda(a,t;P)\le\bar{\lambda}_1$ on
every pair; it is the level-$1$ certificate carried forward, and it is what lets
every later threshold constraint use $\bar{\lambda}_1$ as its ``off'' value.
The \emph{freezing block}~\eqref{stela:freeze-t}--\eqref{stela:freeze-c}
reproduces, for each already-certified level $j<k$, the threshold constraint at
$z=\bar{\lambda}_j$ together with its cardinality budget $j-1$: by
Proposition~\ref{prop:stela-threshold}, requiring at most $j-1$ loads to exceed
$\bar{\lambda}_j$ is equivalent to $\lambda^{\downarrow}_{j}(P)=\bar{\lambda}_j$,
so this block re-imposes membership in $\mathcal{F}_{j}$ without fixing which
pairs realize those ranks.
The \emph{current-level block}~\eqref{stela:cur-t}--\eqref{stela:cur-c} is the
same construction with the free threshold $z$ in place of a constant: it counts
how many loads exceed $z$ and caps that count at $k-1$, so minimizing $z$
in~\eqref{stela:obj} drives it down to $\lambda^{\downarrow}_{k}(P)$.
Finally,~\eqref{stela:mon} enforces $z\le\bar{\lambda}_{k-1}$, keeping the
sequence $(\bar{\lambda}_k)_k$ non-increasing, consistent with $z$ being a
lower-ranked (hence no larger) load than the previous one.

\begin{remark}
\label{rem:stela-linearization}
At the current level, the coefficient $\bar{\lambda}_1-z$
in~\eqref{bilinear constraint} depends on the variable $z$, making the constraint
bilinear. \textsc{Stela} avoids this by replacing~\eqref{bilinear constraint} with the pair
\eqref{stela:gb}--\eqref{stela:cur-t},
\begin{equation*}
  \lambda(a,t;P) \leq \bar{\lambda}_1
  \qquad\text{and}\qquad
  \lambda(a,t;P) \leq z + \bar{\lambda}_1\,y_{a,t}^{(k)},
\end{equation*}
in which the coefficient of $y^{(k)}_{a,t}$ is the \emph{constant}
$\bar{\lambda}_1$. These two constraints together are equivalent to the bilinear
form:
\begin{itemize}
  \item if $y_{a,t}^{(k)}=0$, both reduce to $\lambda(a,t;P)\leq z$;
  \item if $y_{a,t}^{(k)}=1$, the first gives $\lambda(a,t;P)\leq\bar{\lambda}_1$,
    which coincides with the bilinear form~\eqref{bilinear constraint} evaluated
    at $y=1$ (right-hand side $z+(\bar{\lambda}_1-z)=\bar{\lambda}_1$).
\end{itemize}
The substitution is therefore exact, not a relaxation. Since $\bar{\lambda}_1$ is
a constant fixed at the first iteration, every subproblem
$\mathcal{M}_k^{\mathrm{\textsc{Stela}}}$ with $k>1$ is a standard mixed-integer linear
program.
\end{remark}

\paragraph{Correctness.}
We now prove that solving $\mathcal{M}_k^{\mathrm{\textsc{Stela}}}$ certifies rank $k$. In
contrast with \textsc{Alexa}, where a sorting lemma is established once and the freezing
constraints are then switched on, \textsc{Stela} intertwines the two: the current level
uses a free threshold while the earlier levels are frozen, so both the value and
the membership $P^\star_k\in\mathcal{F}_{k-1}$ are established simultaneously by a
two-sided bounding argument, itself carried by an induction over the levels.

\begin{theorem}
\label{thm:stela-correctness}
For every $k\in\mathcal{K}$, $\mathcal{M}_k^{\mathrm{\textsc{Stela}}}$, namely
\eqref{eq:stela-l1} for $k=1$ and~\eqref{mod:stela} for $k\ge2$,
satisfies~\eqref{eq:lex-recursion1}: its optimal value is
$\bar{\lambda}_k=\min_{P\in\mathcal{F}_{k-1}}\lambda^{\downarrow}_{k}(P)$, and the
scheme $P^\star_k$ read from any optimal solution belongs to $\mathcal{F}_k$.
\end{theorem}
\begin{proof}
\[
\begin{array}{:c}
\begin{minipage}{0.97\textwidth}
By induction on $k$. For $k=1$, the constraints of~\eqref{eq:stela-l1} force
$z\ge\max_{(a,t)}\lambda(a,t;P)=\lambda^{\downarrow}_{1}(P)$,
so~\eqref{eq:stela-l1} minimizes $\lambda^{\downarrow}_{1}$ over
$\mathcal{F}_0=\mathcal{F}$: its value is $\bar{\lambda}_1$ and any optimal
$P^\star_1$ attains it, i.e.\ $P^\star_1\in\mathcal{F}_1$. Let now $k\ge2$ and
assume the claim for all $j<k$; write $v_k$ for the optimal value of
$\mathcal{M}_k^{\mathrm{\textsc{Stela}}}$. The proof has two halves: a feasible point
built from a genuine minimizer shows $v_k\le\bar{\lambda}_k$, and any optimal
solution is shown to lie in $\mathcal{F}_{k-1}$ and to satisfy
$\lambda^{\downarrow}_{k}\le z^\star$, giving $v_k\ge\bar{\lambda}_k$.

\emph{Upper bound: $v_k\le\bar{\lambda}_k$.} Fix a scheme
$P\in\mathcal{F}_{k-1}$ attaining $\bar{\lambda}_k$, and build a candidate
solution by setting $z=\bar{\lambda}_k$,
$y^{(j)}_{a,t}=\mathbf{1}\bigl[\lambda(a,t;P)>\bar{\lambda}_j\bigr]$
$(2\le j<k)$ and $y^{(k)}_{a,t}=\mathbf{1}\bigl[\lambda(a,t;P)>z\bigr]$; that is,
we let the exceedance flags take the values they \emph{must} have for $P$.
Since $P\in\mathcal{F}_{k-1}$, we have $\lambda^{\downarrow}_{1}(P)=\bar{\lambda}_1$,
so every load is at most $\bar{\lambda}_1$ and the global bound~\eqref{stela:gb}
holds. For each $j$ with $2\le j<k$, again $\lambda^{\downarrow}_{j}(P)=\bar{\lambda}_j$,
so by Proposition~\ref{prop:stela-threshold} at most $j-1$ loads exceed
$\bar{\lambda}_j$; with the above choice of $y^{(j)}$ this gives both the
threshold constraint~\eqref{stela:freeze-t} and the cardinality
constraint~\eqref{stela:freeze-c}. Likewise
$z=\bar{\lambda}_k=\lambda^{\downarrow}_{k}(P)$ is exceeded by at most $k-1$
loads, yielding~\eqref{stela:cur-t}--\eqref{stela:cur-c}; and
$z=\bar{\lambda}_k\le\bar{\lambda}_{k-1}$ gives~\eqref{stela:mon}. This feasible
point has objective value $\bar{\lambda}_k$, hence $v_k\le\bar{\lambda}_k$.

\emph{Lower bound: $v_k\ge\bar{\lambda}_k$.} Let $(P_k^\star,z^\star,y)$ be an
optimal solution. We first show $P_k^\star\in\mathcal{F}_{k-1}$ by an inner
induction on $j<k$. For $j=1$, the global bound~\eqref{stela:gb} gives
$\lambda^{\downarrow}_{1}(P_k^\star)\le\bar{\lambda}_1$, and since
$P_k^\star\in\mathcal{F}$ the minimality of $\bar{\lambda}_1$ forces
$\lambda^{\downarrow}_{1}(P_k^\star)=\bar{\lambda}_1$. For $2\le j<k$, the
threshold constraints~\eqref{stela:freeze-t} and the cardinality
constraints~\eqref{stela:freeze-c} imply that at most $j-1$ loads exceed
$\bar{\lambda}_j$, so $\lambda^{\downarrow}_{j}(P_k^\star)\le\bar{\lambda}_j$ by
Proposition~\ref{prop:stela-threshold}; combined with
$P_k^\star\in\mathcal{F}_{j-1}$ (inner hypothesis) and the minimality of
$\bar{\lambda}_j$, this gives $\lambda^{\downarrow}_{j}(P_k^\star)=\bar{\lambda}_j$.
Hence $P_k^\star\in\mathcal{F}_{k-1}$. Finally, the current-level
constraints~\eqref{stela:cur-t}--\eqref{stela:cur-c} and
Proposition~\ref{prop:stela-threshold} give
$\lambda^{\downarrow}_{k}(P_k^\star)\le z^\star$, so
\[
  \bar{\lambda}_k
  =\min_{P\in\mathcal{F}_{k-1}}\lambda^{\downarrow}_{k}(P)
  \le\lambda^{\downarrow}_{k}(P_k^\star)\le z^\star=v_k.
\]

Combining both bounds, $v_k=\bar{\lambda}_k$; the last chain is then tight, so
$\lambda^{\downarrow}_{k}(P_k^\star)=\bar{\lambda}_k$ and
$P_k^\star\in\mathcal{F}_k$.
\end{minipage}
\end{array}
\]
\end{proof}

By Theorems~\ref{thm:alexa-correctness} and~\ref{thm:stela-correctness}, both
formulations meet~\eqref{eq:lex-recursion1}, hence both return a lexicographically
optimal routing scheme when driven by Algorithm~\ref{alg:lex-template}.

\subsection{\textsc{CARLA}: Cumulative Aggregation of Ranked Loads Algorithm.}
\label{sec:carla}

\textsc{Carla} reaches the same nested chain
$\mathcal{F}_0 \supseteq \mathcal{F}_1 \supseteq \cdots \supseteq \mathcal{F}_N$ as \textsc{Alexa} and \textsc{Stela}, but
replaces the \emph{direct} encoding of the $k$-th largest load by its
\emph{cumulative} counterpart. The obstacle to a direct encoding is that, for an intermediate rank
$1<k<N$, the order statistic $\lambda^{\downarrow}_k$ is neither convex nor
concave as a function of the load vector, and cannot serve as a MILP
objective without introducing auxiliary binaries. \textsc{Carla} instead
minimises $\Theta_k$, the sum of the $k$ largest loads: a convex quantity
that admits the compact linear representation
of~\cite{ogryczak2003ktamir} (\Cref{lem:ot}) and whose rank-by-rank
minimisation yields \emph{exactly} the same sequence of optimal sets
(\Cref{prop:same-sequence}). The payoff is structural: the only integer
variables of $\mathcal{M}_k^{\mathrm{\textsc{Carla}}}$ are the routing
variables themselves, every auxiliary being continuous.

\paragraph{Cumulative ordered outcomes.}
\label{sec:cumulative}
We now introduce the cumulative ordered outcomes $\Theta_k$, which serve as
a convenient tool for analyzing and comparing load vectors
\begin{definition}\label{def:theta}
For $\mu \in \mathbb{R}^{N}$ and $k \in \mathcal{K}$, the $k$-th \emph{cumulative ordered
outcome} of $\mu$ is
\begin{equation}\label{eq:theta-def}
  \Theta_k(\mu) \;=\; \sum_{r \leq k} \mu^{\downarrow}_r,
  \qquad \Theta_0(\mu) = 0,
\end{equation}
the sum of the $k$ largest entries of $\mu$. For a routing scheme $P \in \mathcal{F}$ we
abbreviate $\Theta_k(P) = \Theta_k(\lambda(P))$.
\end{definition}

Two facts make these quantities useful here. They carry exactly the same
lexicographic information as the sorted vector itself
(\Cref{lem:cum-order}), and, unlike the individual sorted entries, each of them is
the optimal value of a small linear program (\Cref{lem:ot}).

\paragraph{Cumulative sums carry the lexicographic order.}The next lemma establishes the equivalence between the lexicographic order of ordered vectors and that of their cumulative ordered outcomes.

\begin{lemma}\label{lem:cum-order}
For all $\mu, \nu \in \mathbb{R}^{N}$,
\[
  \mu^{\downarrow} \lex \nu^{\downarrow}
  \qquad \Longleftrightarrow \qquad
  \bigl(\Theta_k(\mu)\bigr)_{k \in \mathcal{K}} \lex \bigl(\Theta_k(\nu)\bigr)_{k \in \mathcal{K}}.
\]
\end{lemma}

\begin{proof}
\[
\begin{array}{:c}
\begin{minipage}{0.97\textwidth}
If $\mu^{\downarrow} = \nu^{\downarrow}$, both sides hold with equality throughout. Otherwise
let $k$ be the first index with $\mu^{\downarrow}_k \neq \nu^{\downarrow}_k$. The prefixes of
length $k-1$ coincide, hence $\Theta_j(\mu) = \Theta_j(\nu)$ for every $j < k$,
while
\[
  \Theta_k(\mu) - \Theta_k(\nu)
  \;=\; \mu^{\downarrow}_k - \nu^{\downarrow}_k \;\neq\; 0 .
\]
Thus $k$ is also the first index at which the two cumulative vectors differ, and
they differ with the same sign.
\end{minipage}
\end{array}
\]
\end{proof}

\Cref{lem:cum-order} uses only algebra: it assumes nothing about convexity,
continuity, or the set to which the vectors belong. It therefore applies without
change to the discrete set $\mathcal{F} \subseteq \{0,1\}^{m}$.

To see the structure explicitly, fix a sorted vector
$\mu^{\downarrow} = (\mu^{\downarrow}_1, \dots, \mu^{\downarrow}_N)^{\top}$. By
definition, $\Theta_k(\mu) = \sum_{r \leq k} \mu^{\downarrow}_r$ is a partial sum,
so the map $\mu^{\downarrow} \mapsto \bigl(\Theta_1(\mu), \dots, \Theta_N(\mu)\bigr)$
is linear and given by a single constant matrix:
\begin{equation}\label{eq:theta-matrix}
  \begin{pmatrix} \Theta_1 \\ \Theta_2 \\ \Theta_3 \\ \vdots \\ \Theta_N \end{pmatrix}
  =
  \underbrace{\begin{pmatrix}
    1 & 0 & 0 & \cdots & 0 \\
    1 & 1 & 0 & \cdots & 0 \\
    1 & 1 & 1 & \cdots & 0 \\
    \vdots & & & \ddots & \vdots \\
    1 & 1 & 1 & \cdots & 1
  \end{pmatrix}}_{Q}
  \begin{pmatrix}
    \mu^{\downarrow}_1 \\ \mu^{\downarrow}_2 \\ \mu^{\downarrow}_3 \\ \vdots \\ \mu^{\downarrow}_N
  \end{pmatrix},
  \qquad
  Q_{k,r} =
  \begin{cases} 1 & r \leq k, \\ 0 & r > k. \end{cases}
\end{equation}
The matrix $Q$ is lower triangular, since $\Theta_k$ involves only the ranks
$r \leq k$, and its diagonal entries all equal $1$, since $\mu^{\downarrow}_k$
appears exactly once in $\Theta_k$. Hence $\det Q = 1$, so $Q$ is invertible and
the map is a bijection; its inverse recovers each sorted load from two consecutive
cumulative sums,
\[
  \mu^{\downarrow}_k = \Theta_k(\mu) - \Theta_{k-1}(\mu),
  \qquad \Theta_0(\mu) = 0.
\]
Being a bijection, $Q$ maps the lexicographic order on sorted vectors onto the
lexicographic order on cumulative vectors, which is exactly what
\Cref{lem:cum-order} states.

\paragraph{Linear representation.}We now recall a classical linear optimization representation of cumulative ordered outcomes due to \cite{ogryczak2003ktamir}.

\begin{lemma} \label{lem:ot} 
For every $\mu \in \mathbb{R}_+^{N}$ and $k \in \mathcal{K}$,
\begin{equation}\label{eq:ot}
  \Theta_k(\mu)
  \;=\;
  \min_{z \geq 0,\ u \geq 0}
  \Bigl\{\, k z + \sum_{i=1}^{N} u_i
     \ :\ u_i \geq \mu_i - z \quad \forall i \,\Bigr\},
\end{equation}
and the minimum is attained at $z = \mu^{\downarrow}_k$.
\end{lemma}

\begin{proof}
\[
\begin{array}{:c}
\begin{minipage}{0.97\textwidth}
For a fixed $z$, the inner minimisation is separable and gives
$u_i = (\mu_i - z)_+$. The right-hand side of \eqref{eq:ot} therefore equals
$\min_{z \geq 0} f(z)$ with
\[
  f(z) \;=\; kz + \sum_{i=1}^{N} (\mu_i - z)_+ ,
\]
a convex piecewise-linear function of the single variable $z$, whose one-sided
derivatives are
\[
  f'(z^{+}) = k - \bigl|\{ i : \mu_i > z \}\bigr|,
  \qquad
  f'(z^{-}) = k - \bigl|\{ i : \mu_i \geq z \}\bigr| .
\]
At $z = \mu^{\downarrow}_k$ we have $|\{ i : \mu_i > \mu^{\downarrow}_k \}| \leq k-1$ and
$|\{ i : \mu_i \geq \mu^{\downarrow}_k \}| \geq k$, hence
$f'(z^{+}) \geq 1 > 0 \geq f'(z^{-})$; by convexity, $\mu^{\downarrow}_k$ is a
minimiser. Its value is
\[
  k\,\mu^{\downarrow}_k + \sum_{i} \bigl(\mu_i - \mu^{\downarrow}_k\bigr)_+
  \;=\; k\,\mu^{\downarrow}_k + \sum_{r \leq k} \bigl(\mu^{\downarrow}_r - \mu^{\downarrow}_k\bigr)
  \;=\; \Theta_k(\mu),
\]
where the middle equality holds because $\mu^{\downarrow}_r \leq \mu^{\downarrow}_k$ for
$r > k$. Finally $\mu \geq 0$ implies $\mu^{\downarrow}_k \geq 0$, so the restriction
$z \geq 0$ is inactive.
\end{minipage}
\end{array}
\]
\end{proof}

\paragraph{The two sequences of optimal sets coincide.}
\label{sec:coincide}

Replacing the objective $\lambda^{\downarrow}_k$ by $\Theta_k$ in the recursion
\eqref{eq:lex-recursion1} does not merely preserve the final optimum: it produces the
same nested family of optimal sets at every rank.

\begin{proposition}\label{prop:same-sequence}
Let $\GG_0 = \mathcal{F}$ and, for $k \geq 1$,
\[
  \bar{\Theta}_k \;=\; \min_{P \in \GG_{k-1}} \Theta_k(P),
  \qquad
  \GG_k \;=\; \bigl\{ P \in \GG_{k-1} \,:\, \Theta_k(P) = \bar{\Theta}_k \bigr\}.
\]
Then, for every $k \in \mathcal{K}$,
\[
  \GG_k \;=\; \mathcal{F}_k
  \qquad\text{and}\qquad
  \bar{\Theta}_k \;=\; \sum_{r \leq k} \bar{\lambda}_r ,
\]
where $\mathcal{F}_k$ and $\bar{\lambda}_k$ are as in \eqref{eq:lex-recursion1}.
\end{proposition}

\begin{proof}
\[
\begin{array}{:c}
\begin{minipage}{0.97\textwidth}
By induction on $k$. The case $k = 0$ holds by definition, with
$\bar{\Theta}_0 = 0$. Assume $\GG_{k-1} = \mathcal{F}_{k-1}$ and
$\bar{\Theta}_{k-1} = \sum_{r < k} \bar{\lambda}_r$. Every $P \in \mathcal{F}_{k-1}$
satisfies $\lambda^{\downarrow}_r(P) = \bar{\lambda}_r$ for all $r < k$, hence
\begin{equation}\label{eq:shift}
  \Theta_k(P)
  \;=\; \sum_{r<k} \lambda^{\downarrow}_r(P) + \lambda^{\downarrow}_k(P)
  \;=\; \bar{\Theta}_{k-1} + \lambda^{\downarrow}_k(P) .
\end{equation}
On $\mathcal{F}_{k-1}$ the two objectives $\Theta_k$ and $\lambda^{\downarrow}_k$ therefore
differ by the constant $\bar{\Theta}_{k-1}$. They consequently share the same set
of minimisers, which gives $\GG_k = \mathcal{F}_k$, and their optimal values satisfy
$\bar{\Theta}_k = \bar{\Theta}_{k-1} + \bar{\lambda}_k$.
\end{minipage}
\end{array}
\]
\end{proof}

\Cref{prop:same-sequence} is the structural result that legitimises the whole
approach. It states that \textit{rank $k$ has been solved} means the same thing whether
one optimises the individual sorted loads or their cumulative sums: in both cases
the feasible set has been reduced to $\mathcal{F}_k$, and the certified prefix
$(\bar{\lambda}_1, \dots, \bar{\lambda}_k)$ is the same. 

\begin{corollary}\label{cor:lexopt}
$P^{\star} \in \GG_N$ if and only if $P^{\star}$ is lexicographically optimal
for~\eqref{ctn:objlex}.
\end{corollary}

\begin{proof}
\[
\begin{array}{:c}
\begin{minipage}{0.97\textwidth}
Immediate from \Cref{prop:same-sequence} with $k = N$ and the characterisation
$P^{\star} \in \mathcal{F}_N$ of the lexicographic optima.
\end{minipage}
\end{array}
\]
\end{proof}

\paragraph{The rank subproblem formulation:}\label{sec:mip}

Fix $k \in \mathcal{K}$ and assume the values $\bar{\lambda}_1, \dots, \bar{\lambda}_{k-1}$
of the preceding ranks known. Define the constants
\begin{equation}\label{eq:cj}
  c_j \;=\; \bar{\Theta}_j - j\,\bar{\lambda}_j
      \;=\; \sum_{r < j} \bigl( \bar{\lambda}_r - \bar{\lambda}_j \bigr),
  \qquad 1 \leq j < k .
\end{equation}
The formulation $\mathcal{M}_k^{\mathrm{\textsc{Carla}}}$ is
\begin{subequations}\label{eq:carla}
\begin{align}
  \min_{P,\, z_k,\, u} \quad
    & k\,z_k + \sum_{(a,t) \in \mathcal{A}_T} u^{(k)}_{a,t}
    \label{eq:carla-obj}\\[2pt]
  \text{s.t.}\quad
    & P \in \mathcal{F},
    \label{eq:carla-routing}\\
    & u^{(j)}_{a,t} \;\geq\; \lambda(a,t;P) - \bar{\lambda}_j,
      && \forall (a,t) \in \mathcal{A}_T,\ 1 \leq j < k,
    \label{eq:carla-frozen-lin}\\
    & \sum_{(a,t) \in \mathcal{A}_T} u^{(j)}_{a,t} \;\leq\; c_j,
      && 1 \leq j < k,
    \label{eq:carla-frozen-card}\\
    & u^{(k)}_{a,t} \;\geq\; \lambda(a,t;P) - z_k,
      && \forall (a,t) \in \mathcal{A}_T,
    \label{eq:carla-current}\\
    & 0 \;\leq\; u^{(j)}_{a,t} \;\leq\; \bar{\lambda}_1 - \bar{\lambda}_j,
      && \forall (a,t) \in \mathcal{A}_T,\ 1 \leq j < k,
    \label{eq:carla-box}\\
    & 0 \;\leq\; z_k \;\leq\; \bar{\lambda}_{k-1},
      \qquad u^{(k)}_{a,t} \;\geq\; 0
      && \forall (a,t) \in \mathcal{A}_T .
    \label{eq:carla-zbox}
\end{align}
\end{subequations}

Constraints \eqref{eq:carla-frozen-lin}--\eqref{eq:carla-frozen-card} preserve
the optimal value of each preceding rank; \eqref{eq:carla-current} together with
the objective \eqref{eq:carla-obj} encodes the current rank through
\Cref{lem:ot}; and \eqref{eq:carla-box}--\eqref{eq:carla-zbox} are valid bounds
tightening the linear relaxation. At $k = 1$ the frozen blocks are empty, the box
\eqref{eq:carla-zbox} reads $z_1 \geq 0$ by the convention
$\bar{\lambda}_0 = +\infty$, and \eqref{eq:carla} reduces to the minimisation of
the maximum link utilisation.

\begin{remark}\label{rem:j1}
For $j = 1$ we have $c_1 = 0$ by \eqref{eq:cj}. Combined with $u^{(1)} \geq 0$,
constraint \eqref{eq:carla-frozen-card} forces $u^{(1)} = 0$, and
\eqref{eq:carla-frozen-lin} then reduces to
\[
  \lambda(a,t;P) \;\leq\; \bar{\lambda}_1
  \qquad \forall (a,t) \in \mathcal{A}_T,
\]
that is, to the global bound stating that no load exceeds the optimal maximum link
utilisation. The box \eqref{eq:carla-box} is consistent with this, since it reads
$u^{(1)}_{a,t} \leq 0$.
\end{remark}

 The next result states that $\mathcal{M}_k^{\mathrm{\textsc{Carla}}}$ correctly freezes
the ranks already certified and optimises the next one: its feasible set projects
onto exactly $\mathcal{F}_{k-1}$, so minimising $\Theta_k$ returns $\bar{\Theta}_k$ and
yields a scheme in $\mathcal{F}_k$.
\begin{theorem}\label{thm:carla}
For every $k \in \mathcal{K}$, the optimal value of $\mathcal{M}_k^{\mathrm{\textsc{Carla}}}$ is $\bar{\Theta}_k$, and the
routing scheme $P^{\star}_k$ read from any optimal solution satisfies
$P^{\star}_k \in \mathcal{F}_k$.
\end{theorem}

\begin{proof}
\[
\begin{array}{:c}
\begin{minipage}{0.97\textwidth}
We first show that constraints
\eqref{eq:carla-frozen-lin}--\eqref{eq:carla-frozen-card}, together with
\eqref{eq:carla-box}, cut out exactly $\mathcal{F}_{k-1}$ in the space of routing schemes.

\medskip
\noindent\emph{Soundness.} Let $(P, z_k, u)$ be feasible for \eqref{eq:carla} and
let $j < k$. By \eqref{eq:carla-frozen-lin}, the pair
$\bigl( \bar{\lambda}_j, u^{(j)} \bigr)$ is feasible for the minimisation
\eqref{eq:ot} applied to $\mu = \lambda(P)$; hence by \Cref{lem:ot} and
\eqref{eq:cj},
\[
  \Theta_j(P)
  \;\leq\; j\,\bar{\lambda}_j + \sum_{(a,t) \in \mathcal{A}_T} u^{(j)}_{a,t}
  \;\leq\; j\,\bar{\lambda}_j + c_j
  \;=\; \bar{\Theta}_j .
\]
We now argue by induction on $j$. Suppose $P \in \mathcal{F}_{j-1}$, which holds for
$j = 1$ since $P \in \mathcal{F} = \mathcal{F}_0$. By \Cref{prop:same-sequence},
$\bar{\Theta}_j$ is the minimum of $\Theta_j$ over $\mathcal{F}_{j-1}$, so the inequality
$\Theta_j(P) \leq \bar{\Theta}_j$ forces $\Theta_j(P) = \bar{\Theta}_j$, that is,
$P \in \GG_j = \mathcal{F}_j$. Iterating up to $j = k-1$ gives $P \in \mathcal{F}_{k-1}$.

\medskip
\noindent\emph{Completeness.} Conversely, let $P \in \mathcal{F}_{k-1}$. For $j < k$ set
\[
  u^{(j)}_{a,t} \;=\; \bigl( \lambda(a,t;P) - \bar{\lambda}_j \bigr)_+ .
\]
Since $\lambda^{\downarrow}_j(P) = \bar{\lambda}_j$, \Cref{lem:ot} shows that
$z = \bar{\lambda}_j$ attains the minimum in \eqref{eq:ot}, whence
\[
  j\,\bar{\lambda}_j + \sum_{(a,t) \in \mathcal{A}_T} u^{(j)}_{a,t}
  \;=\; \Theta_j(P) \;=\; \bar{\Theta}_j ,
\]
so that \eqref{eq:carla-frozen-card} holds with equality and
\eqref{eq:carla-frozen-lin} by construction. Moreover
$\lambda^{\downarrow}_1(P) = \bar{\lambda}_1$, so every load is at most
$\bar{\lambda}_1$ and \eqref{eq:carla-box} holds. Finally, taking
\[
  z_k = \lambda^{\downarrow}_k(P) \;\leq\; \lambda^{\downarrow}_{k-1}(P) = \bar{\lambda}_{k-1},
  \qquad
  u^{(k)}_{a,t} = \bigl( \lambda(a,t;P) - z_k \bigr)_+ ,
\]
satisfies \eqref{eq:carla-current}--\eqref{eq:carla-zbox}.

\medskip
\noindent\emph{Objective.} By the two previous steps, the projection of the
feasible set of \eqref{eq:carla} onto the space of routing schemes is exactly
$\mathcal{F}_{k-1}$. For a fixed $P \in \mathcal{F}_{k-1}$, minimising \eqref{eq:carla-obj} over
$(z_k, u^{(k)})$ yields $\Theta_k(P)$ by \Cref{lem:ot}, and the completeness step
exhibits a minimiser $z_k$ lying inside the box \eqref{eq:carla-zbox}, so the
restriction is not binding. Hence
\[
  \min \mathcal{M}_k^{\mathrm{CARLA}} \;=\; \min_{P \in \mathcal{F}_{k-1}} \Theta_k(P) \;=\; \bar{\Theta}_k ,
\]
and any optimal $P^{\star}_k$ attains this minimum, i.e.\
$P^{\star}_k \in \GG_k = \mathcal{F}_k$ by \Cref{prop:same-sequence}.
\end{minipage}
\end{array}
\]
\end{proof}

\paragraph{The algorithm.}The complete \textsc{Carla} framework is summarized in Algorithm~\ref{alg:carla}.
\label{sec:algorithm-carla}
\begin{algorithm}[H]
  \caption{\textsc{Carla} framework}
  \label{alg:carla}
  \begin{algorithmic}[1]
    \Input Instance $(G,D,T,q,\kappa,\maxSeg)$;
    \Output lexicographic minimax-optimal routing $P^\star$ and
      $(\bar{\lambda}_1,\dots,\bar{\lambda}_N)$
    \For{$k \leftarrow 1$ \textbf{to} $N$}
     \State Compute $c_j$ for $j < k$ \Comment From \eqref{eq:cj}\;
     \State Build $\mathcal{M}_k^{\mathrm{\textsc{Carla}}}$ \Comment As in \eqref{eq:carla}\;
     \State $(\bar{\Theta}_k, P^{\star}_k) \leftarrow \textsc{Solve}$ $(\mathcal{M}_k^{\mathrm{\textsc{Carla}}})$\;
     \If{k=1}
        \State $\bar{\lambda}_1\leftarrow \bar{\Theta}_1$ 
     
     \Else
      \State $\bar{\lambda}_k \leftarrow \bar{\Theta}_k -\bar{\Theta}_{k-1}$ 
      \EndIf
  \If{$\bar{\lambda}_k = 0$}{
    \Return $P^{\star}_k$, $(\bar{\lambda}_1, \dots, \bar{\lambda}_k, 0, \dots, 0)$\;\Comment{Early stop}
  }   
      \EndIf
    \EndFor
    \State \Return $P^\star_N,\ (\bar{\lambda}_1,\dots,\bar{\lambda}_N)$
  \end{algorithmic}
\end{algorithm}

\Cref{alg:carla} performs at most $N$ calls to a mixed-integer linear programming
solver. Its correctness follows from \Cref{thm:carla} and
\Cref{cor:lexopt}: after iteration $k$ the incumbent belongs to $\mathcal{F}_k$, and after
the last iteration it belongs to $\mathcal{F}_N$, hence is lexicographically optimal. The
early termination test is justified by the fact that
$\bar{\lambda}_k = 0$ implies $\bar{\lambda}_r = 0$ for all $r > k$, by monotonicity
of the sorted vector.

\subsection{Comparison of the three Methods}
\label{sec:exact-comparison}
\subsubsection{Experimental dataset}
We developed benchmark instances for the $T$-ASR problem. Each instance
combines a network topology, a time-dependent traffic matrix, intervention
scenarios, and operational parameters. Traffic matrices are filtered using a
difficulty indicator based on the gap between a multi-commodity flow lower
bound and a greedy segment-routing upper bound. The reconfiguration budget is
chosen below the cost of independent period-wise optimization to preserve the
temporal coupling. The benchmark instance files are publicly available from the Challenge repository~\cite{roadef2026_instances}. The generation procedure was explained in ~\Cref{sec:generation}. 
In the computational experiments, a rank is called \emph{numerically resolved}
when its MILP terminates under the stated MIP-gap tolerance. This terminology
is deliberately distinct from the exact certification proved for the idealized
formulations: because the experiments use a nonzero MIP gap and a finite
freezing tolerance, the computed prefix is tolerance-dependent and need not
coincide exactly with the mathematical lexicographic optimum.
\subsubsection{Experimental setup}
\label{subsec:setup}

We compare \textsc{Alexa}, \textsc{Stela} and \textsc{Carla} on the \texttt{setAP}~\cite{stela_alexa_code}, one of the benchmark families developed for the
$T$-ASR problem. It contains $125$ instances grouped by
$N\in\{100,200,300,400,500\}$. All three algorithms follow the same sequential
template (\Cref{alg:lex-template}) and differ only in the formulation
$\mathcal{M}_k$ solved at each step, so the comparison isolates the effect of the
encoding.

The evaluation proceeds in two stages. We first run the three algorithms on all
$125$ instances under a time limit of $10$ minutes. This first campaign shows that
\textsc{Alexa} is dominated by both \textsc{Stela} and \textsc{Carla}, and we therefore discard it from the
rest of the study. We then re-run \textsc{Stela} and \textsc{Carla} alone on the same instances
under a longer time limit of $30$ minutes, in order to compare the two remaining
methods on a more demanding budget and to observe how their relative behaviour
evolves as more ranks are certified.

In both campaigns, each run is stopped when its time limit is reached, and we
record the largest rank $k$ reached together with the certified values
$\bar{\lambda}_1,\ldots,\bar{\lambda}_k$. Due to the sequential nature of the
lexicographic procedure, reaching rank $k$ means that the first $k$ ranks have
been solved to optimality. Since each algorithm solves one lexicographic level
through one MILP subproblem, $k$ measures the depth reached within the time
limit. All experiments were implemented in \texttt{Python~3.10.11} using
\texttt{networktools}~\cite{networktools} and \texttt{Gurobi~13.0.2} under
deterministic settings on an \texttt{Intel Xeon~2.80\,GHz} machine with 32\,GB RAM
and \texttt{Ubuntu~24.04}. The complete implementation of the three algorithms is
publicly available in the following Git repository~\cite{stage3A}.

\subsubsection{Numerical tolerances}
\label{alg:num-tolerance}
\Cref{alg:lex-template} assumes exact MILP optimization and exact equality between
successive levels. In the computational experiments both assumptions are relaxed:
each subproblem is solved with a relative MIP gap $\gamma=10^{-8}$, and a certified
level $\bar{\lambda}_j$ is carried into the subsequent subproblems by requiring the
corresponding loads to remain within a band $\bar{\lambda}_j\pm\varepsilon$, with
$\varepsilon=10^{-9}$. The same setting is used for the three algorithms. Enforcing
a two-sided band makes the certified prefix sensitive to the value returned at each
rank: should a level be certified slightly above its true optimum, the lower end of
the band excludes the schemes attaining that optimum. The gap $\gamma$ is chosen
small enough for this to remain marginal, but it is the mechanism behind the
sensitivity reported next.

\paragraph{Observed sensitivity to numerical tolerances.}
For these parameter values, the three algorithms produce consistent solutions on
almost all instances. On a few difficult instances, however, changing $\varepsilon$
was enough to modify the certified lexicographic load vector. The discrepancies we
observed never affected the first ranks: they appeared only from relatively
advanced ranks onward, where many arc-periods carry nearly identical loads and the
lexicographic optimization becomes increasingly degenerate. Since the three
algorithms certify ranks sequentially, a perturbation introduced at one rank is
inherited by all subsequent ones, which is why an isolated tolerance effect
manifests itself as a corrupted tail rather than a single wrong value.
This rank-dependent robustness is qualitatively consistent with the analysis of
Abernethy~et~al.~\cite{abernethy2024}, who establish that the leading lexicographic
components are substantially better conditioned than the later ones. 
\subsubsection{Results}

\begin{figure}[H]
    \centering
    \includegraphics[width=0.4\linewidth]{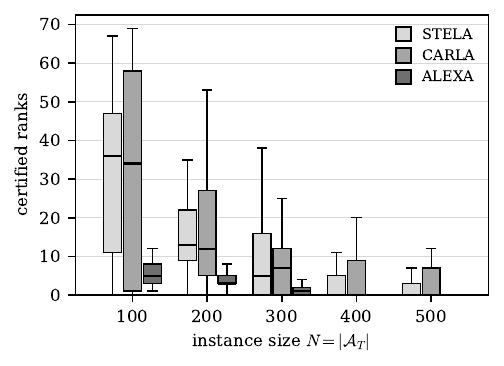}
    \caption{\small Number of certified ranks of the sorted load vector within the
    $10$-minute time limit, for \textsc{Alexa}, \textsc{Stela} and \textsc{Carla}, grouped by instance size
    class $N=|\mathcal{A}_T|$. For each class, the box spans the interquartile
    range (first to third quartile) of the instances in that class, the inner line
    is the median, and the whiskers extend to the most extreme value within
    $1.5\times\mathrm{IQR}$; outliers are not shown.}
    \label{fig:box-three}
\end{figure}

\begin{figure}[H]
    \centering
    \includegraphics[width=0.4\linewidth]{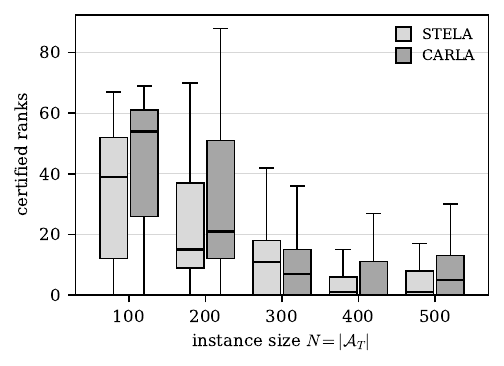}
    \caption{\small Number of certified ranks of the sorted load vector within the
    $30$-minute time limit, for \textsc{Stela} and \textsc{Carla} only, grouped by instance size
    class $N=|\mathcal{A}_T|$. Box, median and whiskers are read as in
    \Cref{fig:box-three}.}
    \label{fig:box-two}
\end{figure}
\Cref{fig:box-three} reports the first campaign, under the $10$-minute limit.
\textsc{Alexa} certifies markedly fewer ranks than \textsc{Stela} and \textsc{Carla} in every size class, and
the gap widens with $N$; this confirms that materialising the full $N\times N$
permutation matrix at each level does not scale, and \textsc{Alexa} is therefore left out of
the second campaign. \textsc{Stela} and \textsc{Carla} reach comparable depths, with a slight
advantage to \textsc{Carla} on the median.

\Cref{fig:box-two} compares \textsc{Stela} and \textsc{Carla} alone under the longer $30$-minute
limit. Both certify more ranks than under $10$ minutes, and the number of certified
ranks decreases as $N$ grows, the larger MILP subproblems being solved more slowly.
Across the classes, \textsc{Carla} remains at or above \textsc{Stela} on the median.

\subsection{Conclusion} \label{conclusion-exact-algo}

The two campaigns support one clear conclusion and one deliberate
non-conclusion.

The clear conclusion concerns \textsc{Alexa}. It certifies markedly fewer ranks than
the other two drivers in every size class, and the gap widens with $N$. This is
directly explained by the encoding: materialising an $N \times N$ permutation
matrix introduces $O(N^2)$ binaries per lexicographic level, against $N$
exceedance binaries for \textsc{Stela} and no additional binary at all for
\textsc{Carla}, whose Ogryczak--Tamir cumulative encoding only requires continuous
auxiliaries. The separation is systematic rather than instance-dependent, so
\textsc{Alexa} is discarded from the remainder of this work.

The non-conclusion concerns \textsc{Stela} and \textsc{Carla}. Their distributions of
certified ranks overlap substantially in every size class, under both the
$10$-minute and the $30$-minute limits; the differences in median are small
relative to the interquartile spread. We therefore do not claim a winner between
the two. 

Both campaigns also show that the number of certified ranks decreases sharply
with $N$: on the largest classes only a short prefix of the sorted load vector is
proved optimal within $30$ minutes, the returned routing scheme being
lexicographic minimax-optimal up to that prefix and merely feasible beyond it. Increasing the
time limit from $10$ to $30$ minutes buys additional ranks without changing the
shape of the curve, which points to the difficulty of the individual rank-$k$
MILPs rather than to the number of levels. This is what motivates the next~\Cref{sec:cg}, where column generation is applied to \emph{both} surviving drivers,
\textsc{Stela} and \textsc{Carla}.

%% file: sections/05-methodCG.tex
\section{TCG: Trajectory-based column generation of STELA and CARLA}
\label{sec:cg}

\subsection{Overview}

\subsubsection{Column generation: master, pool, and pricing}
\label{sec:cg-overview}

The formulations of this Section instantiate the classical
column-generation scheme \cite{dantzig-wolfe-1960,lubbecke-desrosiers-2005};
we recall the terminology in the form used throughout.

A \emph{master problem} is a formulation whose variables select, for each demand $d$, one
trajectory from the set $\Pi_d$ introduced in Definition~\ref{def:cg-traj}. Each trajectory
$\pi\in\Pi_d$ is a \emph{column} of the master. Since $|\Pi_d|$ grows like a product over
periods of the number of admissible paths, the family $\Pi_d$ is astronomically large and
the master cannot be written explicitly.

Column generation solves it without enumeration. For each demand one keeps a \emph{pool}
$\Pi_d^{\mathrm{pool}}\subseteq\Pi_d$, a finite subset of the trajectories generated so far,
and solves the \emph{restricted master problem} (RMP): the master with each $\Pi_d$ replaced
by its pool $\Pi_d^{\mathrm{pool}}$. The LP relaxation of the RMP returns dual multipliers
for its constraints. The \emph{pricing} step then asks, for each demand, whether $\Pi_d$
contains a trajectory of negative \emph{reduced cost}, its master cost corrected by those
duals; such a column can improve the LP and is appended to the pool. Master and pricing
alternate until no demand admits a negative-reduced-cost trajectory, at which point the LP
relaxation of the \emph{full} master is solved to optimality.

The scheme is effective here because, although each $\Pi_d$ is huge, the master has few
rows: one selection row per demand, one load row per arc--time pair, and one budget row per
period. An optimal LP basis therefore activates only a handful of columns per demand, and
pricing discovers exactly those without listing the rest.

\subsubsection{A common trajectory column-generation framework for \textsc{STELA} and \textsc{CARLA}}
\label{sec:cg-motivation}

The theoretical and numerical comparison of Section~\ref{sec:exact-comparison}
identifies \textsc{Stela} and \textsc{Carla} as the two most attractive exact
formulations for the $T$-\srchallenge{} problem. Both formulations avoid the
$O(N^2)$ assignment variables and the associated permutation symmetry of
\textsc{Alexa}, while handling the sorted load vector through a sequential
rank-by-rank optimization scheme. 

And all these exact methods of~\Cref{sec:milp} solve $T$-\srchallenge{} on
the compact arc-based formulation, whose decision variables are the segment
indicators $x^{dt}_{ij}$. This formulation is faithful but does not scale: it
carries one binary per demand, segment and period, the budget
constraint~\eqref{ctn:budget_3} is expressed through absolute values. On the instances of
interest, the number of segments a demand could conceivably use is far larger than
the number it will actually use, so enumerating them implicitly through column
generation is the natural remedy. We therefore use \textsc{Stela} and
\textsc{Carla} as the two compact formulations underlying our
trajectory column-generation framework.

Although their master problems differ, the two formulations share a
fundamental structural property: \emph{they use the same routing decision unit}.
For each demand, the routing decisions over the entire time horizon can be
represented by a single trajectory, i.e., a sequence of segment paths linked
across consecutive periods. Consequently, the temporal coupling induced by
the reconfiguration budget can be absorbed directly into the trajectory.
This observation allows us to develop a common column-generation machinery
for both formulations.

More precisely, the trajectory definition, the restricted master structure,
and the pricing problem associated with a given set of dual variables are
identical for \textsc{Stela} and \textsc{Carla}. The only difference lies in
the \emph{rank-specific master problem}: at each lexicographic rank $k$, the
two formulations impose different rank-dependent constraints and therefore
give rise to different restricted master problems, denoted by
$\mathrm{RMP}^{\textsc{Stela}}_k$ and
$\mathrm{RMP}^{\textsc{Carla}}_k$. The corresponding dual solutions induce
the coefficients of the same underlying trajectory-pricing problem.

This common structure suggests separating the development of the
column-generation method into two layers. We first study the first
lexicographic rank, $k=1$, for which both formulations reduce to a
\emph{minimize-the-maximum-load} problem. At this stage, we introduce the
trajectory masters $\mathrm{RMP}^{\textsc{Stela}}_1$ and $\mathrm{RMP}^{\textsc{Carla}}_1$ , derive their linear relaxation, formulate its dual, and
derive the associated trajectory-pricing problem. Since the decision unit
and the pricing structure are common to both formulations, these elements
are developed once rather than duplicated for \textsc{Stela} and
\textsc{Carla}.

The second layer incorporates the sequential lexicographic structure. Once
the common MLU column-generation machinery has been established, we explain
how the rank-$k$ restricted master is modified when moving from one
lexicographic rank to the next, and how the corresponding dual information
is carried into the pricing problem. This leads to two instances of the
same trajectory column-generation framework, one driven by
$\mathrm{RMP}^{\textsc{Stela}}_k$ and the other by
$\mathrm{RMP}^{\textsc{Carla}}_k$.
\subsubsection{Related work.}
Single-period column generation for segment routing was pioneered by
CG4SR~\cite{cg4sr}, with a resource-constrained shortest-path pricing; it
does not treat temporal coupling. Lexicographic min--max optimization is
classical~\cite{ogryczak1997location,ogryczak2006}. Our setting differs structurally:
their objective is MLU whereas
ours is a lexicographic minimax over the \emph{sorted} load vector, whose coordinates are
induced by the sort; and our columns are coupled across periods by the
reconfiguration budget, a feature absent from CG4SR. It is precisely this coupling that the
trajectory column absorbs.
\subsubsection{Trajectories as decision units}

\begin{definition}
\label{def:cg-traj}
For a horizon $T=\{0,1,\dots,h-1\}$ with $h\ge1$, a \emph{trajectory} for demand $d\in D$
is a tuple $\pi_d=(p^{(0)}_d,\dots,p^{(h-1)}_d)$, where each $p^{(t)}_d$ is an $s_d$--$t_d$
segment path in $G_t$ with at most $\texttt{maxSeg}$ segments. The set of all trajectories
of $d$ is $\Pi_d$, and $\Pi_d^{t}$ denotes the set of feasible $s_d$--$t_d$ segment paths
in $G_t$. The whole treatment is stated for a general horizon $h$; the only coupling
between periods is the reconfiguration chain between consecutive periods.
\end{definition}

A trajectory is meant to be a \emph{self-contained decision unit}: everything the
$T$-\srchallenge{} model needs from demand $d$ over the horizon, the load its routing
places on every arc--time pair, and the reconfiguration cost it incurs between consecutive
periods, is determined by $\pi_d$ \emph{alone}, with no dependence on the trajectories
chosen for the other demands. This locality is the structural point of the formulation and
the reason a per-trajectory column is the right decision unit.

This change of decision unit has two important consequences. First, the
temporal coupling is represented directly in the columns, so that the
reconfiguration budget appears as linear resource constraints in the
master problem. Second, the reduced cost of a trajectory can be evaluated
through a single pricing problem whose structure is common to
\textsc{Stela} and \textsc{Carla}. The two formulations therefore differ
not in how trajectories are generated, but in the master-side information
used to price them at each lexicographic rank.

The remainder of this section develops this common machinery, starting with
the first rank $k=1$ and the associated MLU problem. We then return to the
rank-by-rank lexicographic structure and show how the common trajectory
column-generation scheme is embedded into
$\mathrm{RMP}^{\textsc{Stela}}_k$ and
$\mathrm{RMP}^{\textsc{Carla}}_k$.

\subsection{Restricted master problem (RMP) formulation for rank k=1}

\subsubsection{The trajectory master}
\label{sec:cg-master}

Recall the compact $T$-\srchallenge{} model of \Cref{chap:model}. Its variables are the
binary segments $x^{dt}_{ij}$, its feasible set is
$\mathcal F=\{x\in\{0,1\}:\eqref{ctn:flow},\eqref{ctn:segments},\eqref{ctn:load},
\eqref{ctn:budget_3}\}$, and it minimizes the sorted load vector lexicographically,
\begin{equation}\label{eq:cg-compact}
(\mathrm P)\qquad
\operatorname*{lex\,min}_{x\in\mathcal F}\ \lambda^{\downarrow}(x).
\end{equation}
We build the master from $(\mathrm P)$ in two moves: first we target the top rank of the lexicographic minimax objective, which yields the MLU master carrying all the structure; then we apply a
Dantzig--Wolfe decomposition to $\mathcal F$.

\paragraph{Targeting the top rank (MLU).}
The first lexicographic minimax component is the maximum load
$\lambda^{\downarrow}_1(x)=\max_{(a,t)\in\calA_T}\lambda(a,t)$, whose minimization is the
maximum-link-utilization (MLU) problem. An epigraph variable $U$ linearizes it,
\begin{equation}\label{eq:cg-mlu-linear}
\min\ U
\qquad\text{s.t.}\qquad
\lambda(a,t)\le U,\quad\forall (a,t)\in\calA_T,
\end{equation}
The full lexicographic minimax objective is recovered in \Cref{sec:cg-lex} by stacking, on top
of this master, the rank-$k$ threshold rows of \textsc{Stela}; it therefore suffices to build
the MLU master.

\paragraph{The compact problem and its two families of constraints}

The compact model has binary variables $x^{dt}_{ij}$ and four families of
constraints: flow conservation \eqref{ctn:flow}, the segment count
\eqref{ctn:segments}, the load definition \eqref{ctn:load} and the
reconfiguration budget \eqref{ctn:budget_3}. We classify them by \emph{which
variables they couple}:

\begin{center}
\renewcommand{\arraystretch}{1.25}
\begin{tabular}{@{}lll@{}}
\toprule
Family & Couples & Rôle in the decomposition\\
\midrule
flow conservation, per $(d,t)$ & the variables of one demand only & \emph{separable} $\to$ absorbed in the columns\\
segment count $\sum_{ij}x^{dt}_{ij}\le\maxSeg$ & the variables of one demand only & \emph{separable} $\to$ absorbed in the columns\\
load $\lambda(a,t)=\frac1{c(a)}\sum_{d,ij}r\,\nu\,x^{dt}_{ij}$ & all demands sharing an arc & \emph{coupling} $\to$ stays in the master\\
budget $\sum_{d,ij}|x^{dt}_{ij}-x^{d,t-1}_{ij}|\le\kappa(t)$ & all demands, one period link & \emph{coupling} $\to$ stays in the master\\
\bottomrule
\end{tabular}
\end{center}
 The two coupling families are the only ones that mix demands.

\paragraph{Columns as integer points of the separable block.}

Fix a demand $d$ and consider the separable system
\begin{equation}
\label{eq:sepblock}
S_d:=\Bigl\{x^{d\bullet}\in\{0,1\}^{|T|\times|V|^2}\ :\
\text{flow conservation at every }t,\quad
\textstyle\sum_{i,j}x^{dt}_{ij}\le\maxSeg\ \ \forall t\Bigr\}.
\end{equation}
Because $S_d$ involves no other demand, a feasible point of $S_d$ is a complete
description of ``what demand $d$ does over the whole horizon''. Dantzig--Wolfe
\emph{discretisation} replaces the variables $x^{d\bullet}$ by a selection among
the elements of $S_d$: introduce one weight $\xi_{d,\pi}$ per element $\pi$, and
write
\begin{equation}
\label{eq:cg-recover}
x^{dt}_{ij}=\sum_{\pi}\delta\bigl(p^{(t)}_\pi,ij\bigr)\xi_{d,\pi},
\qquad
\sum_{\pi}\xi_{d,\pi}=1\quad(\mathrm S),\qquad \xi_{d,\pi}\ge0 .
\end{equation}

The integer points of the separable system $S_d$ are the tuples of one feasible
$s_d$--$t_d$ segment path per period with at most $\texttt{maxSeg}$ segments: that is, the
trajectories $\Pi_d$ of \Cref{def:cg-traj}. Dantzig--Wolfe convexification replaces the
$x$-variables of $d$ by a selection over $\Pi_d$ through a binary weight
$\xi_{d,\pi}\in\{0,1\}$, with recovery identity and convexity row
$(\mathrm S)$.

\paragraph{Rewriting the linking constraints.}
Substituting~\eqref{eq:cg-recover} into the two linking families turns them into rows over
the $\xi$-variables, and produces the per-column coefficients.

\emph{Load.} Dividing~\eqref{ctn:load} by $c(a)$ and inserting~\eqref{eq:cg-recover},
\[
\lambda(a,t)\ \ge\
\sum_{d\in D}\sum_{i,j\in V}\frac{r(i,j,a,t)\,\nu(d,t)}{c(a)}\,x^{dt}_{ij}
=\sum_{d\in D}\sum_{\pi\in\Pi_d}
\Bigl[\tfrac{\nu(d,t)}{c(a)}\!\sum_{i,j\in V}\! r(i,j,a,t)\,\delta(p^{(t)}_\pi,ij)\Bigr]\,\xi_{d,\pi}.
\]
The bracketed quantity depends on $\pi$ alone; we name it the \emph{load footprint} of $\pi$
on the pair $(a,t)$,
\begin{equation}
\label{eq:cg-Gamma}
\Gamma_{d,a,t}(\pi)
\;:=\;
\frac{\nu(d,t)}{c(a)}\sum_{i,j\in V} r(i,j,a,t)\,\delta\!\bigl(p^{(t)}_\pi,ij\bigr),
\qquad (a,t)\in\calA_T,
\end{equation}
well defined as a constant precisely because the split coefficients $r(i,j,a,t)$ are
precomputable input data (\Cref{lem:ecmp} of \Cref{chap:model}). Taking this at equality
defines $\lambda(a,t)$ as the row $(\mathrm L)$, and~\eqref{eq:cg-mlu-linear} supplies the
threshold rows $(\mathrm M)$.

\emph{Budget.} For a fixed trajectory $\pi$, the inner sum of~\eqref{ctn:budget_3} across the
chain link $t-1\to t$ depends on $\pi$ alone; we name it the \emph{reconfiguration cost}
\begin{equation}
\label{eq:cg-beta}
\beta_{d,t}(\pi)
\;:=\;
\sum_{i,j\in V}\bigl|\delta(p^{(t)}_\pi,ij)-\delta(p^{(t-1)}_\pi,ij)\bigr|
=\dist\!\bigl(p^{(t-1)}_\pi,p^{(t)}_\pi\bigr),
\qquad t\in T^*,
\end{equation}
a \emph{constant} once $\pi$ is fixed. Hence, using $\sum_\pi\xi_{d,\pi}=1$, the budget
constraint becomes the single linear row
\[
(\mathrm B)\qquad
\sum_{d\in D}\sum_{\pi\in\Pi_d}\beta_{d,t}(\pi)\,\xi_{d,\pi}\le\kappa(t),
\qquad t\in T^*,
\]
with \emph{no absolute values and no auxiliary variables}: the modulus
of~\eqref{ctn:budget_3} is evaluated once per column, at generation time, not carried into
the master.

\paragraph{Absorption of \texttt{maxSeg}.}
Because~\eqref{ctn:flow} and~\eqref{ctn:segments} were used to \emph{define} $\Pi_d$, neither
appears as a master row: both are absorbed into the column set. Only trajectories already
respecting flow conservation and the segment budget are eligible, and this feasibility is
re-imposed structurally inside the pricing subproblem $(\mathrm{SP}_d)$
(\Cref{sec:cg-pricing}), never as an explicit master constraint, which is what keeps the
master rows few and free of the combinatorial segment-count logic. Collecting $(\mathrm S)$,
$(\mathrm L)$, $(\mathrm M)$, $(\mathrm B)$ and the objective $\min U$ gives the master stated
next.

\subsubsection{\textsc{STELA} driven restricted master problem $\mathrm{RMP}^{\textsc{Stela}}_1$ formulation}
\label{sec:cg-rmp}

The master derived above ranges over the full trajectory sets, are far too
large to enumerate. For each demand $d$ we therefore keep a finite \emph{pool}
$\Pi_d^{\mathrm{pool}}\subseteq\Pi_d$ of trajectories and solve the \emph{restricted master
problem} $\mathrm{RMP}^{\textsc{Stela}_1}$: the master with each $\Pi_d$ replaced by its pool. With the per-column
constants $\Gamma_{d,a,t}(\pi)$ of~\eqref{eq:cg-Gamma} and $\beta_{d,t}(\pi)$
of~\eqref{eq:cg-beta}, the MLU restricted master driven by \textsc{Stela} framewoek reads

\begin{subequations}\label{eq:cg-rmp-mlu}
\begin{align}
\mathrm{RMP}^{\textsc{Stela}}_1:\qquad
\min\quad & U \label{eq:cg-rmp-obj}\\
\text{s.t.}\quad
& \sum_{\pi\in\Pi_d^{\pool}} \xi_{d,\pi}=1,
&& \forall d\in D,
&& (\mathrm S)\ \ [\alpha^{\textsc{Stela}}_d\ \text{free}]
\label{eq:cg-S}\\
& \lambda(a,t)-\sum_{d\in D}\sum_{\pi\in\Pi_d^{\pool}}
  \Gamma_{d,a,t}(\pi)\,\xi_{d,\pi}=0,
&& \forall (a,t)\in\AT,
&& (\mathrm L)\ \ [\eta^{\textsc{Stela}}_{a,t}\ \text{free}]
\label{eq:cg-L}\\
& \lambda(a,t)-U\le 0,
&& \forall (a,t)\in\AT,
&& (\mathrm M)\ \ [\mu^{\textsc{Stela}}_{a,t}\le 0]
\label{eq:cg-M}\\
& \sum_{d\in D}\sum_{\pi\in\Pi_d^{\pool}}
  \beta_{d,t}(\pi)\,\xi_{d,\pi}\le\kappa(t),
&& \forall t\in T^*,
&& (\mathrm B)\ \ [\psi^{\textsc{Stela}}_t\le 0]
\label{eq:cg-B}\\
& \xi_{d,\pi}\ge 0,\quad
  \lambda(a,t)\ \text{free},\quad U\ \text{free}.
&& && \notag
\end{align}
\end{subequations}


Constraint $(\mathrm S)$ selects exactly one trajectory per demand, which (since a
trajectory contains one path per period) fixes the whole routing of $d$ over the horizon.
Constraint $(\mathrm L)$ assembles the arc--time loads from the load footprints
$\Gamma_{d,a,t}(\pi)$, constraint $(\mathrm M)$ encodes the MLU objective, and $(\mathrm B)$
is a single linear inequality per period, as established in \Cref{eq:cg-beta}. The relaxation $\xi_{d,\pi}\in\{0,1\}\rightsquigarrow\xi_{d,\pi}\ge0$ loses
      nothing on the upper side: $(\mathrm S)$ together with $\xi\ge0$ already
      implies $\xi_{d,\pi}\le1$.

\begin{proposition}
\label{prop:cg-valid}
Every feasible point of \eqref{eq:cg-rmp-mlu} induces a $T$-\srchallenge{} routing scheme
with the same arc--time loads and the same reconfiguration costs. Conversely, if
$\bigcup_d\Pi_d^{\mathrm{pool}}$ contains an optimal trajectory for every demand, then
\eqref{eq:cg-rmp-mlu} attains the $T$-\srchallenge{} optimum for the MLU objective.
\end{proposition}

\begin{proof} \[ \begin{array}{:c} \begin{minipage}{0.97\textwidth}
Let $(\xi,\lambda,U)$ be feasible. By $(\mathrm S)$, for each $d$ exactly one trajectory
$\pi_d$ has $\xi_{d,\pi_d}=1$. Routing $d$ along $p^{(t)}_d$ at each period $t$ defines a
routing scheme $P$. Fix $(a,t)\in\calA_T$. Since exactly one $\xi_{d,\pi}$ is one per demand,
$(\mathrm L)$ collapses to a sum over the selected trajectories, and substituting the
footprint~\eqref{eq:cg-Gamma} gives
\[
\lambda(a,t)=\sum_{d\in D}\Gamma_{d,a,t}(\pi_d)
=\sum_{d\in D}\frac{\nu(d,t)}{c(a)}\sum_{i,j\in V} r(i,j,a,t)\,\delta\!\bigl(p^{(t)}_d,ij\bigr),
\]
exactly the load $\lambda(a,t;P)$ of the model (\Cref{chap:model}). Likewise, for each
$t\in T^*$, substituting~\eqref{eq:cg-beta} into $(\mathrm B)$ reduces it to
$\sum_{d}\dist(p^{(t-1)}_d,p^{(t)}_d)\le\kappa(t)$, i.e.\ the budget
constraint~\eqref{ctn:budget_3}. Hence $P\in\mathcal{F}$ and its loads coincide with the
master variables, so $U\ge\mlu(P)$; minimizing $U$ minimizes the MLU.

Conversely, let $P^\star$ be MLU-optimal with trajectory $\pi_d^\star$ for each demand. If
$\pi_d^\star\in\Pi_d^{\mathrm{pool}}$ for all $d$, setting $\xi_{d,\pi_d^\star}=1$ and all
other selections to zero reproduces, via the same substitutions~\eqref{eq:cg-Gamma}
and~\eqref{eq:cg-beta}, the same loads and budget consumption, hence the same MLU. Since the
RMP optimizes over a subset of trajectories, its value cannot beat the global optimum, so it
attains it.
\end{minipage} \end{array} \] \end{proof}

\subsubsection{\textsc{\textsc{CARLA}}-driven restricted master problem $\mathrm{RMP}^{\textsc{Carla}}_1$ formulation}
\label{sec:cg-carla}

The trajectory master of \Cref{sec:cg-master} rests on three
families of rows that are independent of how the lexicographic minimax objective is encoded: the
selection rows $(\mathrm S)$, the load-assembly rows $(\mathrm L)$ and the budget
rows $(\mathrm B)$ of~\eqref{eq:cg-rmp-mlu}, together with the per-column constants
$\Gamma_{d,a,t}(\pi)$ of~\eqref{eq:cg-Gamma} and $\beta_{d,t}(\pi)$
of~\eqref{eq:cg-beta} attached to the trajectories $\pi\in\Pi_d$ of
\Cref{def:cg-traj}. \textsc{Carla} keeps all of these unchanged and differs from
\textsc{Stela} only in the rows that bound the assembled loads $\lambda(a,t)$. We
therefore build the \textsc{Carla} master the way \Cref{sec:cg-master} built the
\textsc{Stela} one;  from its first rank, the MLU;  prove and price it there,
and only then extend it to the full lexicographic order, where alone the two
encodings diverge. No new column, oracle or pool is introduced.

\paragraph{The cumulative MLU master}
\label{sec:cg-carla-mlu}

Where \textsc{Stela} linearises the maximum load by the epigraph $\min U$,
$\lambda(a,t)\le U$, \textsc{Carla} uses the cumulative (Ogryczak--Tamir)
representation of the sum of the $k$ largest loads,
\begin{equation}
\label{eq:cg-carla-OT}
\Theta_k(\lambda)=\sum_{r\le k}\lambda^{\downarrow}_r
=\min_{z\in\mathbb R}\Big\{k\,z+\sum_{(a,t)\in\calA_T}\bigl(\lambda(a,t)-z\bigr)^+\Big\},
\end{equation}
and minimises the ranks $\Theta_1,\Theta_2,\dots$ in turn. At the first rank
$\Theta_1=\lambda^{\downarrow}_1$ is the MLU itself. Introducing continuous excess
variables $u_{a,t}\ge0$ for the positive parts $(\lambda(a,t)-z_1)^+$, the
\textsc{Carla} MLU restricted master is $(\mathrm S),(\mathrm L),(\mathrm B)$
of~\eqref{eq:cg-rmp-mlu} with the objective $\min U$ and the threshold row
$(\mathrm M)$ replaced by the cumulative block

\begin{subequations}\label{eq:cg-carla-mlu}
\begin{equation}\label{eq:cg-carla-mlu-cur}
\mathrm{RMP}^{\textsc{Carla}}_1:\qquad
\begin{array}[t]{llll}
\min & \displaystyle z_1+\sum_{(a,t)\in\AT}u_{a,t} &&\\[10pt]
\text{s.t.}
&\displaystyle\sum_{\pi\in\Pi_d^{\pool}}\xi_{d,\pi}=1,
& \forall d\in D,
& [\alpha^{\textsc{Carla}}_d\ \text{free}]\quad(\mathrm S)\\[10pt]
&\displaystyle \lambda(a,t)-\sum_{d\in D}\sum_{\pi\in\Pi_d^{\pool}}
   \Gamma_{d,a,t}(\pi)\,\xi_{d,\pi}=0,
& \forall(a,t)\in\AT,
& [\eta^{\textsc{Carla}}_{a,t}\ \text{free}]\quad(\mathrm L)\\[10pt]
&\displaystyle\sum_{d\in D}\sum_{\pi\in\Pi_d^{\pool}}
   \beta_{d,t}(\pi)\,\xi_{d,\pi}\le\kappa(t),
& \forall t\in T^*,
& [\psi^{\textsc{Carla}}_t\le0]\quad(\mathrm B)\\[10pt]
&\lambda(a,t)-z_1-u_{a,t}\le0,
& \forall(a,t)\in\AT,
& [\mu^{\textsc{Carla}}_{a,t}\le0]\quad(\mathrm C)\\[6pt]
&\xi_{d,\pi}\ge0,\quad\lambda(a,t)\ \text{free},\quad
 u_{a,t}\ge0,\quad z_1\ \text{free}. &&
\end{array}
\end{equation}
\end{subequations}

The row $(\mathrm C)$ is the cumulative row~\eqref{eq:cg-carla-mlu} written in the
canonical ``$\le$'' form: $u_{a,t}\ge\lambda(a,t)-z_1$ is
$\lambda(a,t)-z_1-u_{a,t}\le0$. Writing it this way \emph{before} taking the
dual is not cosmetic: the coefficients of $\lambda(a,t)$, $z_1$ and $u_{a,t}$
in the row are read off directly, and they are $+1$, $-1$ and $-1$
respectively.

\begin{proposition}
\label{prop:cg-carla-mlu}
For integer $\xi$, \eqref{eq:cg-carla-mlu} induces a $T$-\srchallenge{} routing
scheme $P$ (\Cref{prop:cg-valid}), and its optimal value equals $\mlu(P)$. Hence the
minimum of~\eqref{eq:cg-carla-mlu} over integer $\xi$ equals $\min_{P}\mlu(P)$, the
same value as the \textsc{Stela} MLU master~\eqref{eq:cg-rmp-mlu}; the two LP
relaxations share this value as well.
\end{proposition}

\begin{proof} 
\[ \begin{array}{:c} \begin{minipage}{0.97\textwidth}
Fix an integer $\xi$. By \Cref{prop:cg-valid} it selects one trajectory per demand,
hence a routing scheme $P$ with $\lambda(a,t)=\lambda(a,t;P)$ and $(\mathrm B)$
reducing to the budget constraint. With $\lambda$ fixed, the inner minimisation over
$u\ge0$ is separable and solved by $u_{a,t}^{\star}=(\lambda(a,t)-z_1)^+$, leaving
\[
\phi(z_1):=z_1+\sum_{(a,t)\in\calA_T}\bigl(\lambda(a,t)-z_1\bigr)^+ .
\]
$\phi$ is convex and piecewise linear, with right derivative
$1-\#\{(a,t):\lambda(a,t)>z_1\}$, which is nonnegative exactly when at most one load
exceeds $z_1$. Thus $\phi$ attains its minimum for $z_1\in[\lambda^{\downarrow}_2,
\lambda^{\downarrow}_1]$, where only the largest load exceeds $z_1$ and
$\phi(z_1)=z_1+(\lambda^{\downarrow}_1-z_1)=\lambda^{\downarrow}_1=\mlu(P)$. So the
master objective at its optimal $(z_1,u)$ equals $\mlu(P)$, and minimising over
integer $\xi$ minimises the MLU. The same pointwise identity
$\min_{z_1,u}\{\cdot\}=\max_{(a,t)}\lambda(a,t)$ holds for \emph{every} fractional
load vector produced by $(\mathrm L)$, so~\eqref{eq:cg-carla-mlu}
and~\eqref{eq:cg-rmp-mlu} minimise the same function of $\xi$ and have equal LP
optima.
\end{minipage} \end{array} \]\end{proof}

\subsection{Dual formulations and the factorized pricing problem for k=1}
\label{sec:cg-dual}

The two restricted masters of the previous subsections are \emph{not} the same
linear program. They share every row that involves a trajectory, but they encode
the maximum link utilization differently. Consequently
\[
\mathrm{RMP}^{\textsc{Stela}}_1\ \neq\ \mathrm{RMP}^{\textsc{Carla}}_1
\]
as formulations: they carry different variables and different rows, and their
duals are linear programs over different variable sets. The purpose of this
subsection is to show that this difference is confined to the master side: the
reduced cost of a trajectory has, in both formulations, the same
trajectory-dependent structure, so that a single pricing oracle serves both and
only the dual coefficients supplied to it change. Writing
$\mathsf X\in\{\textsc{Stela},\textsc{Carla}\}$ for the formulation index, the
narrative is
\[
\boxed{\ \text{different }\mathrm{RMP}^{\mathsf X}_1
\ \Longrightarrow\ \text{different dual prices }
(\alpha^{\mathsf X},\eta^{\mathsf X},\gamma^{\mathsf X})
\ \Longrightarrow\ \text{same pricing structure}\ }
\]
and we establish each step below. The generic reduced-cost construction used
throughout is detailed in \cite{lubbecke-desrosiers-2005}; we now specialise it
to the two rank-1 restricted masters.

\subsubsection{The two rank-1 restricted masters and their common core}
\label{sec:cg-dual-core}

Throughout, $\mathsf X\in\{\textsc{Stela},\textsc{Carla}\}$ indexes the
formulation, the pools $\Pi_d^{\mathrm{pool}}\subseteq\Pi_d$ of
\Cref{def:cg-traj} are fixed, and the per-column constants are the load footprint
$\Gamma_{d,a,t}(\pi)$ of~\eqref{eq:cg-Gamma} and the reconfiguration cost
$\beta_{d,t}(\pi)$ of~\eqref{eq:cg-beta}. We work with the LP relaxations,
obtained by replacing $\xi_{d,\pi}\in\{0,1\}$ by $\xi_{d,\pi}\ge0$; the selection
rows then force $\xi_{d,\pi}\le1$ automatically.

Both masters are built on the same \emph{core}, namely the rows of
\eqref{eq:cg-rmp-mlu} that contain a selection variable, together with the free
load variables they define:
\begin{subequations}\label{eq:cg-core}
\begin{align}
(\mathcal C)\qquad
&\sum_{\pi\in\Pi_d^{\mathrm{pool}}}\xi_{d,\pi}=1,
&&\forall d\in D,
&&(\mathrm S)\ \ [\alpha^{\mathsf X}_d]
\label{eq:cg-core-S}\\
&\lambda(a,t)-\sum_{d\in D}\sum_{\pi\in\Pi_d^{\mathrm{pool}}}
   \Gamma_{d,a,t}(\pi)\,\xi_{d,\pi}=0,
&&\forall(a,t)\in\calA_T,
&&(\mathrm L)\ \ [\eta^{\mathsf X}_{a,t}]
\label{eq:cg-core-L}\\
&\sum_{d\in D}\sum_{\pi\in\Pi_d^{\mathrm{pool}}}\beta_{d,t}(\pi)\,\xi_{d,\pi}
   \le\kappa(t),
&&\forall t\in T^*,
&&(\mathrm B)\ \ [\psi^{\mathsf X}_t]
\label{eq:cg-core-B}\\
&\xi_{d,\pi}\ge0,\qquad \lambda(a,t)\ \text{free},
&&&&
\label{eq:cg-core-box}
\end{align}
\end{subequations}
the bracketed symbols being the associated dual variables. The two formulations
differ only by the block that constrains the assembled loads and by the objective
it carries:
\begin{align}
(\mathcal M^{\textsc{Stela}})\quad
&\min\ U
&&\text{s.t.}\ \ \lambda(a,t)\le U,
&&\forall(a,t),\ [\mu^{\textsc{Stela}}_{a,t}],
\label{eq:cg-block-stela}\\[2pt]
(\mathcal M^{\textsc{Carla}})\quad
&\min\ z_1+\!\!\sum_{(a,t)\in\calA_T}\!\! u_{a,t}
&&\text{s.t.}\ \ \lambda(a,t)\le z_1+u_{a,t},
&&\forall(a,t),\ [\mu^{\textsc{Carla}}_{a,t}],
\label{eq:cg-block-carla}
\end{align}
where $U$ is free in~\eqref{eq:cg-block-stela}, and $u_{a,t}\ge0$ with $z_1$ free
in~\eqref{eq:cg-block-carla}. Thus
$\mathrm{RMP}^{\textsc{Stela}}_1=(\mathcal C)+(\mathcal M^{\textsc{Stela}})$ is
\eqref{eq:cg-rmp-mlu} and
$\mathrm{RMP}^{\textsc{Carla}}_1=(\mathcal C)+(\mathcal M^{\textsc{Carla}})$ is
\eqref{eq:cg-carla-mlu}. Three observations organise what follows. First, the
selection variables occur \emph{only} in the core: no $\xi_{d,\pi}$ appears in
either MLU block, which reaches them solely through the free load variables
$\lambda(a,t)$. Second, the columns are identical objects in the two masters,
since $\Gamma$ and $\beta$ are attributes of a trajectory and of nothing else.
Third, the objective is carried by the block alone, so $\xi_{d,\pi}$ has zero
direct objective cost in both, as already noted in \Cref{sec:cg-pricing}.

We use throughout the following standing convention for a minimisation primal:
equality rows receive free duals, $\le$ rows receive nonpositive duals, a free
primal variable yields a dual \emph{equality} and a nonnegative primal variable a
dual \emph{inequality} (dual row activity $\le$ objective coefficient), and the
dual objective is the maximisation of the right-hand-side inner product. We also
assume, as is the case once the pool contains the nominal trajectories, that both
restricted masters are feasible; both are bounded below by~$0$, so strong duality
applies to each.

\subsubsection{Dual of $\mathrm{RMP}^{\textsc{Stela}}_1$ and its pricing problem}
\label{sec:cg-dual-stela}

\paragraph{Identification of the primal variables}
The dual of $\mathrm{RMP}^{\textsc{Stela}}_1$ has one relation per primal
variable, determined by the variable's domain, its objective coefficient, and
the rows in which it occurs. Identifying these three items is thus the whole
of the derivation, and we record them in \Cref{tab:cg-incidence-stela} before
stating the dual.
\begin{table}[H]
\centering
\small
\renewcommand{\arraystretch}{1.15}
\begin{tabular}{@{}l c c >{\raggedright\arraybackslash}p{0.52\linewidth}@{}}
\toprule
Variable & Domain & Obj. & Rows and coefficients \\
\midrule
$\xi_{d,\pi}$
  & $\ge 0$ & $0$
  & $(\mathrm S)_d$: $+1$;\ \
    $(\mathrm L)_{a,t}$: $-\Gamma_{d,a,t}(\pi)$;\ \
    $(\mathrm B)_t$: $+\beta_{d,t}(\pi)$ \\[2pt]
$\lambda(a,t)$
  & free & $0$
  & $(\mathrm L)_{a,t}$: $+1$;\ \ $(\mathrm M)_{a,t}$: $+1$ \\[2pt]
$U$
  & free & $1$
  & $(\mathrm M)_{a,t}$: $-1$, for every $(a,t)\in\AT$ \\
\bottomrule
\end{tabular}
\caption{\small Incidence of the primal variables of
$\mathrm{RMP}^{\textsc{Stela}}_1$~\eqref{eq:cg-rmp-mlu}: domain, objective
coefficient, and the rows in which each variable occurs. }
\label{tab:cg-incidence-stela}
\end{table}

\paragraph{The dual of $\mathrm{RMP}^{\textsc{Stela}}_1$}

We state the dual first and prove it afterwards, variable by variable. The
derivation uses nothing beyond \Cref{tab:cg-incidence-stela}: the dual has one
relation per primal variable, and that relation is read off the variable's
domain, its objective coefficient, and the rows in which it occurs.

\begin{proposition}
\label{prop:cg-dual-stela}
The LP dual of $\mathrm{RMP}^{\textsc{Stela}}_1$~\eqref{eq:cg-rmp-mlu} is
\begin{subequations}\label{eq:cg-dual-stela}
\begin{align}
\max\;&\;\sum_{d\in D}\alpha^{\textsc{Stela}}_d
        +\sum_{t\in T^*}\kappa(t)\,\psi^{\textsc{Stela}}_t
\label{eq:cg-dual-stela-obj}\\
\text{s.t.}\;
&\;\alpha^{\textsc{Stela}}_d
   -\!\!\sum_{(a,t)\in\AT}\!\!
     \eta^{\textsc{Stela}}_{a,t}\,\Gamma_{d,a,t}(\pi)
   +\!\sum_{t\in T^*}\!\psi^{\textsc{Stela}}_t\,\beta_{d,t}(\pi)\le0,
&&\forall d\in D,\ \pi\in\Pi_d^{\pool},
&& [\xi_{d,\pi}]
\label{eq:cg-dual-stela-xi}\\
&\;\eta^{\textsc{Stela}}_{a,t}+\mu^{\textsc{Stela}}_{a,t}=0,
&&\forall(a,t)\in\AT,
&& [\lambda(a,t)]
\label{eq:cg-dual-stela-lambda}\\
&\;\sum_{(a,t)\in\AT}\mu^{\textsc{Stela}}_{a,t}=-1,
&&
&& [U]
\label{eq:cg-dual-stela-U}\\
&\;\mu^{\textsc{Stela}}_{a,t}\le0,\quad \psi^{\textsc{Stela}}_t\le0,\quad
\alpha^{\textsc{Stela}}_d,\ \eta^{\textsc{Stela}}_{a,t}\ \text{free}.
&& &&
\label{eq:cg-dual-stela-signs}
\end{align}
\end{subequations}
\end{proposition}

\begin{proof}
\[
\begin{array}{:c}
\begin{minipage}{0.97\textwidth}
We derive the objective first, then one relation per primal variable.

\medskip
\noindent\emph{Dual objective.} The dual objective is the inner product of the
duals with the right-hand sides of the primal rows, which are $1$ for each
$(\mathrm S)_d$, $0$ for each $(\mathrm L)_{a,t}$, $0$ for each
$(\mathrm M)_{a,t}$ and $\kappa(t)$ for each $(\mathrm B)_t$. Hence
\[
\sum_{d\in D}\alpha^{\textsc{Stela}}_d\cdot 1
+\sum_{(a,t)\in\AT}\eta^{\textsc{Stela}}_{a,t}\cdot 0
+\sum_{(a,t)\in\AT}\mu^{\textsc{Stela}}_{a,t}\cdot 0
+\sum_{t\in T^*}\psi^{\textsc{Stela}}_t\,\kappa(t)
\;=\;
\sum_{d\in D}\alpha^{\textsc{Stela}}_d
+\sum_{t\in T^*}\kappa(t)\,\psi^{\textsc{Stela}}_t,
\]
which is~\eqref{eq:cg-dual-stela-obj}. Only the selection and budget rows carry
a nonzero right-hand side, so only $\alpha^{\textsc{Stela}}$ and
$\psi^{\textsc{Stela}}$ appear.

\medskip
\noindent\emph{Variable $U$.} It occurs only in the family $(\mathrm M)$,
written $\lambda(a,t)-U\le0$, hence with coefficient $-1$, once per pair
$(a,t)\in\AT$. Its objective coefficient is $1$ and it is free, so its dual
relation is an \emph{equality}:
\[
\sum_{(a,t)\in\AT}\mu^{\textsc{Stela}}_{a,t}\cdot(-1)=1,
\qquad\text{i.e.}\qquad
\sum_{(a,t)\in\AT}\mu^{\textsc{Stela}}_{a,t}=-1,
\]
which is~\eqref{eq:cg-dual-stela-U}. Each ingredient has a distinct origin: the
coefficient $-1$ comes from the way $(\mathrm M)$ is written, the right-hand
side $1$ is the objective coefficient of $U$, and the equality rather than an
inequality comes from $U$ being free.

\medskip
\noindent\emph{Variable $\lambda(a,t)$.} Fix a pair $(a,t)$. It occurs in
exactly two rows, both of the \emph{same} pair: the load-assembly row
$(\mathrm L)_{a,t}$ with coefficient $+1$, and the threshold row
$(\mathrm M)_{a,t}$ with coefficient $+1$. It occurs in no row of another pair,
in no selection row and in no budget row. Its objective coefficient is $0$ and
it is free, so
\[
\eta^{\textsc{Stela}}_{a,t}\cdot 1+\mu^{\textsc{Stela}}_{a,t}\cdot 1=0,
\qquad\text{i.e.}\qquad
\eta^{\textsc{Stela}}_{a,t}+\mu^{\textsc{Stela}}_{a,t}=0,
\]
which is~\eqref{eq:cg-dual-stela-lambda}.

\medskip
\noindent\emph{Variable $\xi_{d,\pi}$.} It occurs in $(\mathrm S)_d$ with
coefficient $+1$, in each $(\mathrm L)_{a,t}$ with coefficient
$-\Gamma_{d,a,t}(\pi)$, and in each $(\mathrm B)_t$ with coefficient
$\beta_{d,t}(\pi)$; it occurs in no other row, in particular in no threshold
row. Being sign-constrained, its dual relation is an \emph{inequality}, dual
activity at most the objective coefficient, the latter being $0$:
\[
\underbrace{\alpha^{\textsc{Stela}}_d\cdot 1}_{(\mathrm S)_d}
+\underbrace{\sum_{(a,t)\in\AT}\eta^{\textsc{Stela}}_{a,t}
  \bigl(-\Gamma_{d,a,t}(\pi)\bigr)}_{(\mathrm L)}
+\underbrace{\sum_{t\in T^*}\psi^{\textsc{Stela}}_t\,
  \beta_{d,t}(\pi)}_{(\mathrm B)}
\;\le\;0,
\]
which is~\eqref{eq:cg-dual-stela-xi}.

Finally the declared signs
\eqref{eq:cg-dual-stela-signs} are those of the rows: equality rows
$(\mathrm S)$ and $(\mathrm L)$ give free duals
$\alpha^{\textsc{Stela}}_d,\eta^{\textsc{Stela}}_{a,t}$, and the inequality
rows $(\mathrm M)$ and $(\mathrm B)$ of a minimisation give
$\mu^{\textsc{Stela}}_{a,t}\le0$ and $\psi^{\textsc{Stela}}_t\le0$.

\end{minipage}
\end{array}
\]
\end{proof}

\begin{proposition}
\label{prop:cg-simplex-stela}
At any dual feasible point of~\eqref{eq:cg-dual-stela},
\[
\eta^{\textsc{Stela}}_{a,t}=-\mu^{\textsc{Stela}}_{a,t}\ \ge\ 0
\quad\forall(a,t)\in\AT,
\qquad
\sum_{(a,t)\in\AT}\eta^{\textsc{Stela}}_{a,t}=1,
\]
that is, $\eta^{\textsc{Stela}}\in\Delta_{\calA_T}
=\{\eta\in\mathbb R^{\calA_T}_+:\sum_{(a,t)}\eta_{a,t}=1\}$.
\end{proposition}

\begin{proof}
\[
\begin{array}{:c}
\begin{minipage}{0.97\textwidth}
\emph{Nonnegativity.} Stationarity of $\lambda(a,t)$,
\eqref{eq:cg-dual-stela-lambda}, gives
$\eta^{\textsc{Stela}}_{a,t}=-\mu^{\textsc{Stela}}_{a,t}$, and
$\mu^{\textsc{Stela}}_{a,t}\le0$ by~\eqref{eq:cg-dual-stela-signs}.
 Summing the previous identity over $\AT$ and
using~\eqref{eq:cg-dual-stela-U},
$\sum_{(a,t)}\eta^{\textsc{Stela}}_{a,t}
=-\sum_{(a,t)}\mu^{\textsc{Stela}}_{a,t}=1$.
\end{minipage}
\end{array}
\]

\end{proof}

\begin{remark}
$\eta^{\textsc{Stela}}$ is a probability distribution over the arc--time pairs.
Complementary slackness on $(\mathrm M)$ forces
$\eta^{\textsc{Stela}}_{a,t}=0$ whenever $\lambda(a,t)<U$, so the mass sits
entirely on the pairs achieving the maximum load. 
\end{remark}
 We first derive the reduced cost associated with a trajectory at rank 1.
\begin{proposition}
\label{prop:cg-rc-stela}
Setting $\gamma^{\textsc{Stela}}_t:=-\psi^{\textsc{Stela}}_t\ge0$, the reduced
cost of a trajectory $\pi\in\Pi_d$, pooled or not, is
\begin{equation}
\label{eq:cg-rc-stela}
\rc^{\textsc{Stela}}_{d}(\pi)
=-\alpha^{\textsc{Stela}}_d
+\sum_{(a,t)\in\AT}\eta^{\textsc{Stela}}_{a,t}\,\Gamma_{d,a,t}(\pi)
+\sum_{t\in T^*}\gamma^{\textsc{Stela}}_t\,\beta_{d,t}(\pi).
\end{equation}
\end{proposition}

\begin{proof}
The reduced cost of a column is its direct objective cost minus its dual
activity. The direct cost is zero: the objective of $\mathrm{RMP}^{\textsc{Stela}}_1$
is $\min U$, and $\xi_{d,\pi}$ does not appear in it; the column reaches the
objective only indirectly, through the free variables $\lambda(a,t)$. The dual
activity was computed in the proof of \Cref{prop:cg-dual-stela}, so
\[
\rc^{\textsc{Stela}}_d(\pi)
=0-\Bigl[\alpha^{\textsc{Stela}}_d
-\sum_{(a,t)\in\AT}\eta^{\textsc{Stela}}_{a,t}\Gamma_{d,a,t}(\pi)
+\sum_{t\in T^*}\psi^{\textsc{Stela}}_t\beta_{d,t}(\pi)\Bigr],
\]
and substituting $\gamma^{\textsc{Stela}}_t=-\psi^{\textsc{Stela}}_t$
gives~\eqref{eq:cg-rc-stela}. The three signs are worth reading. The term
$-\alpha^{\textsc{Stela}}_d$ is the price already paid for serving demand $d$,
which a new column must beat. The load term is a \emph{charge}: the two minus
signs, one from writing $(\mathrm L)$ as $\lambda-\sum\Gamma\xi=0$, one from
the subtraction, cancel, so a column loading the network where
$\eta^{\textsc{Stela}}$ is concentrated is expensive. The budget term is a
charge as well, since $\gamma^{\textsc{Stela}}_t\ge0$ penalises columns that
reconfigure heavily, and a tight budget row makes it large.

Finally, $\Gamma_{d,a,t}(\pi)$ and $\beta_{d,t}(\pi)$ are defined
by~\eqref{eq:cg-Gamma}--\eqref{eq:cg-beta} for \emph{every} trajectory, pooled
or not. The dual constraint~\eqref{eq:cg-dual-stela-xi} is stated for
$\pi\in\Pi_d^{\pool}$, but the right-hand side of~\eqref{eq:cg-rc-stela}
defines a function on all of $\Pi_d$; this extension is a definition, not a
constraint of the restricted dual, and it is what makes pricing over the whole
of $\Pi_d$ meaningful.
\end{proof}

Using the reduced-cost expression above, we now derive the corresponding pricing problem.
\begin{proposition}
\label{prop:cg-cert}
Fix an optimal dual solution of the LP relaxation of
$\mathrm{RMP}^{\textsc{Stela}}_1$ and a demand $d$, and let
\begin{equation}
\label{eq:cg-SP1}
\val(\mathrm{SP}^{\textsc{Stela}}_d):=\min_{\pi\in\Pi_d}
\Bigl\{\sum_{(a,t)\in\AT}\eta^{\textsc{Stela}}_{a,t}\Gamma_{d,a,t}(\pi)
+\sum_{t\in T^*}\gamma^{\textsc{Stela}}_t\beta_{d,t}(\pi)\Bigr\}.
\end{equation}
Then $\min_{\pi\in\Pi_d}\rc^{\textsc{Stela}}_d(\pi)
=-\alpha^{\textsc{Stela}}_d+\val(\mathrm{SP}^{\textsc{Stela}}_d)$, and a
negative-reduced-cost column exists if and only if
$\val(\mathrm{SP}^{\textsc{Stela}}_d)<\alpha^{\textsc{Stela}}_d$.
\end{proposition}

\begin{proof}
\[
\begin{array}{:c}
\begin{minipage}{0.97\textwidth}
In~\eqref{eq:cg-rc-stela} the term $-\alpha^{\textsc{Stela}}_d$ does not depend
on $\pi$: $\alpha^{\textsc{Stela}}_d$ is the dual of the convexity row of
demand $d$, in which every column of $d$ has coefficient $1$. It is therefore a
constant of the minimisation, which leaves exactly~\eqref{eq:cg-SP1}. The
minimum is attained since $\Pi_d$ is finite, and the criterion follows.
\end{minipage}
\end{array}
\]
\end{proof}
\begin{proposition}
\label{prop:cg-theta}
Define, for $t\in T$ and $i\neq j$,
\begin{equation}
\label{eq:cg-theta1}
\theta^{\textsc{Stela}}_d(i,j,t):=
\sum_{a\,:\,(a,t)\in\AT}\eta^{\textsc{Stela}}_{a,t}\,
\frac{\nu(d,t)}{c(a)}\,r(i,j,a,t)\ \ge\ 0 .
\end{equation}
Then, writing $\theta^{\textsc{Stela}}_d(p,t)
:=\sum_{(i,j)\in\segs(p)}\theta^{\textsc{Stela}}_d(i,j,t)$,
\begin{equation}
\label{eq:cg-SP1-path}
\val(\mathrm{SP}^{\textsc{Stela}}_d)=\min_{\pi\in\Pi_d}
\Bigl\{\sum_{t\in T}\theta^{\textsc{Stela}}_d\bigl(p^{(t)},t\bigr)
+\sum_{t\in T^*}\gamma^{\textsc{Stela}}_t\,
  \dist\bigl(p^{(t-1)},p^{(t)}\bigr)\Bigr\}.
\end{equation}
\end{proposition}

\begin{proof}
\[
\begin{array}{:c}
\begin{minipage}{0.97\textwidth}
Insert~\eqref{eq:cg-Gamma} into the first sum of~\eqref{eq:cg-SP1} and exchange
the order of summation, legitimate because all sums are finite:
$$
\begin{aligned}
\sum_{(a,t)\in\AT}\eta^{\textsc{Stela}}_{a,t}\Gamma_{d,a,t}(\pi)
&=\sum_{t\in T}\ \sum_{a:(a,t)\in\AT}\eta^{\textsc{Stela}}_{a,t}\,
  \frac{\nu(d,t)}{c(a)}\sum_{i,j\in V}r(i,j,a,t)\,\delta\bigl(p^{(t)},ij\bigr)\\
&=\sum_{t\in T}\ \sum_{i,j\in V}\delta\bigl(p^{(t)},ij\bigr)\,
  \theta^{\textsc{Stela}}_d(i,j,t)\\
&=\sum_{t\in T}\sum_{(i,j)\in\segs(p^{(t)})}
  \theta^{\textsc{Stela}}_d(i,j,t),
\end{aligned}
$$
the last equality using $\delta(p,ij)\in\{0,1\}$ with
$\delta(p,ij)=1\Leftrightarrow(i,j)\in\segs(p)$, i.e.\ the fact that a segment
path uses each segment at most once. The second sum
of~\eqref{eq:cg-SP1} is $\sum_{t\in T^*}\gamma^{\textsc{Stela}}_t
\dist(p^{(t-1)},p^{(t)})$ by the definition~\eqref{eq:cg-beta} of $\beta$.
Nonnegativity of $\theta^{\textsc{Stela}}_d$ follows from
$\eta^{\textsc{Stela}}\ge0$ (\Cref{prop:cg-simplex-stela}), $\nu(d,t)\ge0$,
$c(a)>0$ and $r(i,j,a,t)\ge0$.
\end{minipage}
\end{array}
\]
\end{proof}

\subsubsection{Dual of $\mathrm{RMP}^{\textsc{Carla}}_1$}
\label{sec:cg-dual-carla}

\paragraph{Identification of the primal variables.}
The dual has one relation per primal variable, determined by the variable's
domain, its objective coefficient, and the rows in which it occurs. We record
these three items before deriving anything.

\begin{table}[ht]
\centering
\small
\renewcommand{\arraystretch}{1.15}
\begin{tabular}{@{}l c c >{\raggedright\arraybackslash}p{0.52\linewidth}@{}}
\toprule
Variable & Domain & Obj. & Rows and coefficients \\
\midrule
$\xi_{d,\pi}$
  & $\ge 0$ & $0$
  & $(\mathrm S)_d$: $+1$;\ \
    $(\mathrm L)_{a,t}$: $-\Gamma_{d,a,t}(\pi)$;\ \
    $(\mathrm B)_t$: $+\beta_{d,t}(\pi)$ \\[2pt]
$\lambda(a,t)$
  & free & $0$
  & $(\mathrm L)_{a,t}$: $+1$;\ \ $(\mathrm C)_{a,t}$: $+1$ \\[2pt]
$z_1$
  & free & $1$
  & $(\mathrm C)_{a,t}$: $-1$, for every $(a,t)\in\calA_T$ \\[2pt]
$u_{a,t}$
  & $\ge 0$ & $1$
  & $(\mathrm C)_{a,t}$: $-1$ \\
\bottomrule
\end{tabular}
\caption{\small Incidence of the primal variables of
$\mathrm{RMP}^{\textsc{Carla}}_1$~\eqref{eq:cg-carla-mlu}: domain, objective
coefficient, and the rows in which each variable occurs. Each variable listed
occurs in no row other than those shown.}
\label{tab:cg-incidence-carla}
\end{table}

The first line of \Cref{tab:cg-incidence-carla} is again the decisive one:
$\xi_{d,\pi}$ occurs in the core rows $(\mathrm S)$, $(\mathrm L)$,
$(\mathrm B)$ and \emph{not} in $(\mathrm C)$. The cumulative block reaches the
columns only through the free load variables $\lambda(a,t)$, exactly as the
threshold block did in \textsc{Stela}, which is why the two reduced costs will
share one functional form.

\begin{proposition}
\label{prop:cg-dual-carla}
The LP dual of $\mathrm{RMP}^{\textsc{Carla}}_1$~\eqref{eq:cg-carla-mlu} is
\begin{subequations}\label{eq:cg-dual-carla}
\begin{align}
\max\;&\;\sum_{d\in D}\alpha^{\textsc{Carla}}_d
        +\sum_{t\in T^*}\kappa(t)\,\psi^{\textsc{Carla}}_t
\label{eq:cg-dual-carla-obj}\\
\text{s.t.}\;
&\;\alpha^{\textsc{Carla}}_d
   -\!\!\sum_{(a,t)\in\calA_T}\!\!
     \eta^{\textsc{Carla}}_{a,t}\,\Gamma_{d,a,t}(\pi)
   +\!\sum_{t\in T^*}\!\psi^{\textsc{Carla}}_t\,\beta_{d,t}(\pi)\le0,
&&\forall d\in D,\ \pi\in\Pi_d^{\mathrm{pool}},\ [\xi_{d,\pi}]
\label{eq:cg-dual-carla-xi}\\
&\;\eta^{\textsc{Carla}}_{a,t}+\mu^{\textsc{Carla}}_{a,t}=0,
&&\forall(a,t)\in\calA_T,\ [\lambda(a,t)]
\label{eq:cg-dual-carla-lambda}\\
&\;\sum_{(a,t)\in\calA_T}\mu^{\textsc{Carla}}_{a,t}=-1,
&&[z_1]
\label{eq:cg-dual-carla-z}\\
&\;\mu^{\textsc{Carla}}_{a,t}\ge-1,
&&\forall(a,t)\in\calA_T,\ [u_{a,t}]
\label{eq:cg-dual-carla-u}\\
&\;\mu^{\textsc{Carla}}_{a,t}\le0,\quad \psi^{\textsc{Carla}}_t\le0,\quad
\alpha^{\textsc{Carla}}_d,\ \eta^{\textsc{Carla}}_{a,t}\ \text{free}.
&&
\label{eq:cg-dual-carla-signs}
\end{align}
\end{subequations}
\end{proposition}

\begin{proof}
\[
\begin{array}{:c}
\begin{minipage}{0.97\textwidth}
We derive the objective first, then one relation per primal variable.

\medskip
\noindent\emph{Dual objective.} The dual objective is the inner product of the
duals with the right-hand sides, which are $1$ for each $(\mathrm S)_d$, $0$
for each $(\mathrm L)_{a,t}$, $\kappa(t)$ for each $(\mathrm B)_t$ and $0$ for
each $(\mathrm C)_{a,t}$. Hence
$$
\sum_{d\in D}\alpha^{\textsc{Carla}}_d\cdot 1
+\sum_{(a,t)}\eta^{\textsc{Carla}}_{a,t}\cdot 0
+\sum_{t\in T^*}\psi^{\textsc{Carla}}_t\,\kappa(t)
+\sum_{(a,t)}\mu^{\textsc{Carla}}_{a,t}\cdot 0
=\sum_{d\in D}\alpha^{\textsc{Carla}}_d
+\sum_{t\in T^*}\kappa(t)\,\psi^{\textsc{Carla}}_t,
$$
which is~\eqref{eq:cg-dual-carla-obj}. It is identical in form to the
\textsc{Stela} dual objective, the two formulations sharing the core rows; but
the dual feasible sets differ, so the optimal values of $\alpha$ and $\psi$
need not coincide.

\medskip
\noindent\emph{Variable $\lambda(a,t)$.} Free, objective coefficient $0$,
occurring in $(\mathrm L)_{a,t}$ with coefficient $+1$ and in
$(\mathrm C)_{a,t}$ with coefficient $+1$, and in no other row. Its dual
relation is an equality,
$$
\eta^{\textsc{Carla}}_{a,t}\cdot 1+\mu^{\textsc{Carla}}_{a,t}\cdot 1=0,
\qquad\text{i.e.}\qquad
\eta^{\textsc{Carla}}_{a,t}=-\mu^{\textsc{Carla}}_{a,t}\ \ge\ 0,
$$
the nonnegativity because $(\mathrm C)$ is a ``$\le$'' row of a minimisation,
so $\mu^{\textsc{Carla}}_{a,t}\le0$. This
is~\eqref{eq:cg-dual-carla-lambda}.

\medskip
\noindent\emph{Variable $z_1$.} Free, objective coefficient $1$, occurring in
every row of the family $(\mathrm C)$ with coefficient $-1$ and nowhere else.
Being free, its dual relation is an equality,
$$
\sum_{(a,t)\in\calA_T}\mu^{\textsc{Carla}}_{a,t}\cdot(-1)=1,
\qquad\text{i.e.}\qquad
\sum_{(a,t)\in\calA_T}\mu^{\textsc{Carla}}_{a,t}=-1,
$$
which is~\eqref{eq:cg-dual-carla-z}.

\medskip
\noindent\emph{Variables $u_{a,t}$.} \emph{Nonnegative}, objective coefficient
$1$, occurring only in $(\mathrm C)_{a,t}$ with coefficient $-1$. Because the
variable is sign-constrained, the dual relation is an \emph{inequality}, dual
row activity at most the objective coefficient:
$$
\mu^{\textsc{Carla}}_{a,t}\cdot(-1)\ \le\ 1,
\qquad\text{i.e.}\qquad
\mu^{\textsc{Carla}}_{a,t}\ \ge\ -1,
$$
which is~\eqref{eq:cg-dual-carla-u}, equivalently
$\eta^{\textsc{Carla}}_{a,t}\le1$. The reduced cost of $u_{a,t}$ is
$1+\mu^{\textsc{Carla}}_{a,t}\ge0$, and complementary slackness makes it
vanish wherever $u_{a,t}>0$: $\eta^{\textsc{Carla}}_{a,t}=1$ at any pair
carrying a strictly positive excess.

\medskip
\noindent\emph{Variable $\xi_{d,\pi}$.} Nonnegative, zero objective
coefficient, occurring in $(\mathrm S)_d$ with coefficient $+1$, in each
$(\mathrm L)_{a,t}$ with coefficient $-\Gamma_{d,a,t}(\pi)$, in each
$(\mathrm B)_t$ with coefficient $\beta_{d,t}(\pi)$, and \emph{not} in
$(\mathrm C)$. Its dual relation is therefore the
inequality~\eqref{eq:cg-dual-carla-xi}. Finally the declared
signs~\eqref{eq:cg-dual-carla-signs} are those of the rows: the equality rows
$(\mathrm S)$ and $(\mathrm L)$ give free duals, and the inequality rows
$(\mathrm B)$ and $(\mathrm C)$ of a minimisation give
$\psi^{\textsc{Carla}}_t\le0$ and $\mu^{\textsc{Carla}}_{a,t}\le0$.
\end{minipage}
\end{array}
\]
\end{proof}

\paragraph{Dual signs.}
Declared: $\alpha^{\textsc{Carla}}_d$ and $\eta^{\textsc{Carla}}_{a,t}$ free,
$\mu^{\textsc{Carla}}_{a,t}\le0$, $\psi^{\textsc{Carla}}_t\le0$. The derived
ones are what the pricing problem rests on.

\begin{proposition}
\label{prop:cg-simplex-carla}
At any dual feasible point of~\eqref{eq:cg-dual-carla},
\[
\eta^{\textsc{Carla}}_{a,t}=-\mu^{\textsc{Carla}}_{a,t}\in[0,1]
\quad\forall(a,t)\in\calA_T,
\qquad
\sum_{(a,t)\in\calA_T}\eta^{\textsc{Carla}}_{a,t}=1,
\]
that is, $\eta^{\textsc{Carla}}\in\Delta_{\calA_T}$.
\end{proposition}

\begin{proof}
\[
\begin{array}{:c}
\begin{minipage}{0.97\textwidth}
\emph{Nonnegativity.} Stationarity of $\lambda(a,t)$,
\eqref{eq:cg-dual-carla-lambda}, gives
$\eta^{\textsc{Carla}}_{a,t}=-\mu^{\textsc{Carla}}_{a,t}$, and
$\mu^{\textsc{Carla}}_{a,t}\le0$ because $(\mathrm C)$ is a ``$\le$'' row of a
minimisation.

\emph{Total mass.} Summing that identity over $\calA_T$ and
using~\eqref{eq:cg-dual-carla-z},
$\sum_{(a,t)}\eta^{\textsc{Carla}}_{a,t}
=-\sum_{(a,t)}\mu^{\textsc{Carla}}_{a,t}=1$.

\emph{Upper bound.} The dual relation of $u_{a,t}$,
\eqref{eq:cg-dual-carla-u}, gives $-\mu^{\textsc{Carla}}_{a,t}\le1$, i.e.\
$\eta^{\textsc{Carla}}_{a,t}\le1$.
\end{minipage}
\end{array}
\]
\end{proof}

\begin{remark}
\label{rem:carla-ub-rank1}
Given $\eta^{\textsc{Carla}}\ge0$ and
$\sum_{(a,t)}\eta^{\textsc{Carla}}_{a,t}=1$, the bound
$\eta^{\textsc{Carla}}_{a,t}\le1$ is automatic, so at rank $1$ the dual
relation of $u_{a,t}$ carries no information beyond the simplex. This ceases
to hold at rank $k$, where the total mass of $-\mu^{\textsc{Carla}}_k$ is $k$
and the constraint $-\mu^{\textsc{Carla}}_k(a,t)\le1$ becomes the binding
description of a genuinely different polytope.
\end{remark}

\begin{proposition}
\label{prop:cg-rc-carla}
Setting $\gamma^{\textsc{Carla}}_t:=-\psi^{\textsc{Carla}}_t\ge0$, the reduced
cost of a trajectory $\pi\in\Pi_d$, pooled or not, is
\begin{equation}
\label{eq:cg-rc-carla}
\rc^{\textsc{Carla}}_{d}(\pi)
=-\alpha^{\textsc{Carla}}_d
+\sum_{(a,t)\in\calA_T}\eta^{\textsc{Carla}}_{a,t}\,\Gamma_{d,a,t}(\pi)
+\sum_{t\in T^*}\gamma^{\textsc{Carla}}_t\,\beta_{d,t}(\pi).
\end{equation}
\end{proposition}

\begin{proof}
\[
\begin{array}{:c}
\begin{minipage}{0.97\textwidth}
The direct objective cost of $\xi_{d,\pi}$ is zero: the objective
$z_1+\sum_{(a,t)}u_{a,t}$ contains no selection variable. Its dual activity was
computed in the proof of \Cref{prop:cg-dual-carla}, so
$$
\rc^{\textsc{Carla}}_d(\pi)
=0-\Bigl[\alpha^{\textsc{Carla}}_d
-\sum_{(a,t)}\eta^{\textsc{Carla}}_{a,t}\Gamma_{d,a,t}(\pi)
+\sum_{t\in T^*}\psi^{\textsc{Carla}}_t\beta_{d,t}(\pi)\Bigr],
$$
and substituting $\gamma^{\textsc{Carla}}_t=-\psi^{\textsc{Carla}}_t$
gives~\eqref{eq:cg-rc-carla}.

Observe that $z_1$ and $u_{a,t}$, although they carry the entire objective, do
\emph{not} contribute to this reduced cost. They contribute to the dual
\emph{relations} that pin down $\eta^{\textsc{Carla}}$, but they share no row
with $\xi_{d,\pi}$, so their duals never multiply a coefficient of
$\xi_{d,\pi}$. This is the precise sense in which the auxiliary variables of
the rank block do not enter the reduced cost, and it is what makes
\eqref{eq:cg-rc-carla} identical in form to its \textsc{Stela} counterpart.
\end{minipage}
\end{array}
\]
\end{proof}

The derivation of \Cref{prop:cg-cert} and \Cref{prop:cg-theta} used only three
inputs: $\alpha_d$ constant over $\Pi_d$, $\eta\ge0$, and $\gamma\ge0$. All
three hold here, the second by \Cref{prop:cg-simplex-carla} and the third by
$\psi^{\textsc{Carla}}_t\le0$. Defining
\begin{equation}
\label{eq:cg-theta-carla-1}
\theta^{\textsc{Carla}}_d(i,j,t):=
\sum_{a\,:\,(a,t)\in\calA_T}\eta^{\textsc{Carla}}_{a,t}\,
\frac{\nu(d,t)}{c(a)}\,r(i,j,a,t)\ \ge\ 0,
\end{equation}
we therefore obtain, verbatim,
\begin{equation}
\label{eq:cg-SP-carla-1}
\val(\mathrm{SP}^{\textsc{Carla}}_d)=\min_{\pi\in\Pi_d}
\Bigl\{\sum_{t\in T}\theta^{\textsc{Carla}}_d\bigl(p^{(t)},t\bigr)
+\sum_{t\in T^*}\gamma^{\textsc{Carla}}_t\,
  \dist\bigl(p^{(t-1)},p^{(t)}\bigr)\Bigr\},
\end{equation}
together with
$\min_{\pi\in\Pi_d}\rc^{\textsc{Carla}}_d(\pi)
=-\alpha^{\textsc{Carla}}_d+\val(\mathrm{SP}^{\textsc{Carla}}_d)$, and a
negative-reduced-cost column exists if and only if
$\val(\mathrm{SP}^{\textsc{Carla}}_d)<\alpha^{\textsc{Carla}}_d$. The two
formulations thus feed the same pricing oracle; only the numerical values of
$\theta^{\textsc{Carla}}_d$ and $\gamma^{\textsc{Carla}}_t$ differ.

\subsection{Common pricing structure}
\label{sec:cg-dual-compare}

Putting~\eqref{eq:cg-rc-stela} and~\eqref{eq:cg-rc-carla} side by side reveals
the structural invariant:
\begin{equation}
\label{eq:cg-rc-sidebyside}
\rc^{\mathsf X}_{d}(\pi)
\;=\;-\alpha^{\mathsf X}_d
+\sum_{(a,t)\in\calA_T}\eta^{\mathsf X}_{a,t}\,\Gamma_{d,a,t}(\pi)
+\sum_{t\in T^*}\gamma^{\mathsf X}_t\,\beta_{d,t}(\pi),
\qquad \mathsf X\in\{\textsc{Stela},\textsc{Carla}\}.
\end{equation}
In both formulations the reduced cost is an affine function of the \emph{same}
pair of column attributes
\[
\pi\ \longmapsto\
\bigl(\Gamma_{d,\bullet}(\pi),\ \beta_{d,\bullet}(\pi)\bigr)
\in\mathbb R^{\calA_T}\times\mathbb R^{T^*},
\]
a map defined by \Cref{def:cg-traj}, \eqref{eq:cg-Gamma} and~\eqref{eq:cg-beta}
alone and independent of the MLU encoding. What differs is the triple of
coefficients $(\alpha^{\mathsf X},\eta^{\mathsf X},\gamma^{\mathsf X})$ against
which this map is evaluated: each triple is produced by the dual of a different
restricted master and carries the imprint of its encoding. We therefore separate
two kinds of information:
\begin{itemize}\itemsep2pt
\item \emph{formulation-dependent master-side prices}: the triple
      $(\alpha^{\mathsf X},\eta^{\mathsf X},\gamma^{\mathsf X})$ with
      $\eta^{\mathsf X}\in\Delta_{\calA_T}$ and $\gamma^{\mathsf X}\ge0$,
      extracted from the dual of $\mathrm{RMP}^{\mathsf X}_1$;
\item \emph{formulation-independent pricing structure}: the minimisation over
      $\Pi_d$ of a nonnegative linear functional of
      $\bigl(\Gamma_{d,\bullet}(\pi),\beta_{d,\bullet}(\pi)\bigr)$, whose
      combinatorial content; one bounded-hop segment path per period, coupled
      along the reconfiguration chain; is the same for both~$\mathsf X$.
\end{itemize}
In particular, the two formulations need not produce the same numerical dual
values: degeneracy of the restricted master is the rule rather than the
exception when several pairs $(a,t)$ attain the maximum load, and two solver
runs, one on each master, typically return different optimal triples. The
numerical reduced cost of a given trajectory therefore need not agree between
the two formulations; only its functional form~\eqref{eq:cg-rc-sidebyside} does.
This is the exact sense in which different restricted masters produce different
dual solutions while sharing one pricing problem.

\subsubsection{The factorized pricing problem}
\label{sec:cg-dual-pricing}

The comparison above licenses a single pricing development. Fix
$\mathsf X\in\{\textsc{Stela},\textsc{Carla}\}$ and let
$(\alpha^{\mathsf X},\eta^{\mathsf X},\gamma^{\mathsf X})$ be an optimal dual
solution of $\mathrm{RMP}^{\mathsf X}_1$, so that
$\eta^{\mathsf X}\in\Delta_{\calA_T}$ and $\gamma^{\mathsf X}\ge0$. The pricing
subproblem of demand~$d$ is
\begin{equation}
\label{eq:cg-SP-generic}
(\mathrm{SP}^{\mathsf X}_d)\qquad
\operatorname{val}(\mathrm{SP}^{\mathsf X}_d)
:=\min_{\pi\in\Pi_d}
\Bigl\{\sum_{(a,t)\in\calA_T}\eta^{\mathsf X}_{a,t}\,\Gamma_{d,a,t}(\pi)
+\sum_{t\in T^*}\gamma^{\mathsf X}_t\,\beta_{d,t}(\pi)\Bigr\}.
\end{equation}
By~\eqref{eq:cg-rc-sidebyside} the term $-\alpha^{\mathsf X}_d$ is constant over
$\Pi_d$, so
\begin{equation}
\label{eq:cg-cert-generic}
\min_{\pi\in\Pi_d}\rc^{\mathsf X}_d(\pi)
=-\alpha^{\mathsf X}_d+\operatorname{val}(\mathrm{SP}^{\mathsf X}_d).
\end{equation}

\begin{proposition}\label{prop:cert}
For a fixed dual solution of the LP relaxation of\/
$\mathrm{RMP}^{\mathsf X}_1$ and a fixed demand~$d$, the minimum reduced cost
among all columns of $d$ is
$\min_{\pi\in\Pi_d}\rc^{\mathsf X}_d(\pi)
=-\alpha^{\mathsf X}_d+\operatorname{val}(\mathrm{SP}^{\mathsf X}_d)$.
Hence a negative reduced-cost column for $d$ exists if and only if
$\operatorname{val}(\mathrm{SP}^{\mathsf X}_d)<\alpha^{\mathsf X}_d$, and every
optimiser of $(\mathrm{SP}^{\mathsf X}_d)$ gives a minimum reduced-cost column.
\end{proposition}

\begin{proof} \[ \begin{array}{:c} \begin{minipage}{0.97\textwidth}
In the reduced cost~\eqref{eq:cg-rc-sidebyside}, the term
$-\alpha^{\mathsf X}_d$ is independent of the trajectory~$\pi$. Minimising
$\rc^{\mathsf X}_d(\pi)$ over $\Pi_d$ therefore amounts to minimising the
pricing objective of~\eqref{eq:cg-SP-generic}, whose optimum is
$\operatorname{val}(\mathrm{SP}^{\mathsf X}_d)$; the minimum reduced cost
equals $-\alpha^{\mathsf X}_d+\operatorname{val}(\mathrm{SP}^{\mathsf X}_d)$,
negative if and only if
$\operatorname{val}(\mathrm{SP}^{\mathsf X}_d)<\alpha^{\mathsf X}_d$.
\end{minipage} \end{array} \]\end{proof}

It remains to state when the pricing loop may stop. The following corollary
turns the subproblems $(\mathrm{SP}^{\mathsf X}_d)$ into an exact optimality
certificate for the LP relaxation: once no demand prices out a
negative-reduced-cost column, the restricted master already solves the full LP.

\begin{corollary}\label{cor:lp}
If\/ $(\mathrm{SP}^{\mathsf X}_d)$ is solved for every $d\in D$ and
$\operatorname{val}(\mathrm{SP}^{\mathsf X}_d)\ge\alpha^{\mathsf X}_d$ for
all~$d$, then no column of the full trajectory master has negative reduced
cost, and the current restricted-master LP solution is optimal for the full LP
relaxation.
\end{corollary}

\begin{proof} \[ \begin{array}{:c} \begin{minipage}{0.97\textwidth}
By \Cref{prop:cert},
$\operatorname{val}(\mathrm{SP}^{\mathsf X}_d)\ge\alpha^{\mathsf X}_d$ means
$\rc^{\mathsf X}_d(\pi)\ge0$ for every $\pi\in\Pi_d$, i.e.\ every column of
demand~$d$ has nonnegative reduced cost. When this holds for all $d\in D$, the
standard reduced-cost optimality condition gives optimality of the restricted LP
for the full LP.
\end{minipage} \end{array} \]\end{proof}

\medskip

It remains to observe that~\eqref{eq:cg-SP-generic} is the problem already
analysed in \Cref{sec:cg-pricing}. Substituting the load
footprint~\eqref{eq:cg-Gamma} and exchanging summations isolates the effective
segment cost
\begin{equation}
\label{eq:cg-theta-generic}
\theta^{\mathsf X}_d(i,j,t)
:=\sum_{a\in A}\eta^{\mathsf X}_{a,t}\,\frac{\nu(d,t)}{c(a)}\,r(i,j,a,t)\ \ge\ 0,
\end{equation}
nonnegative because $\eta^{\mathsf X}\ge0$, $\nu(d,t)\ge0$, $c(a)>0$ and
$r(i,j,a,t)\ge0$; writing $\theta^{\mathsf X}_d(p,t)
=\sum_{(i,j)\in\segs(p)}\theta^{\mathsf X}_d(i,j,t)$ and
$\beta_{d,t}(\pi)=\dist(p^{(t-1)},p^{(t)})$, \eqref{eq:cg-SP-generic} becomes
\begin{equation}
\label{eq:cg-SP-generic-path}
\operatorname{val}(\mathrm{SP}^{\mathsf X}_d)
=\min_{\pi\in\Pi_d}
\Bigl\{\sum_{t\in T}\theta^{\mathsf X}_d\bigl(p^{(t)},t\bigr)
+\sum_{t\in T^*}\gamma^{\mathsf X}_t\,\dist\bigl(p^{(t-1)},p^{(t)}\bigr)\Bigr\}.
\end{equation}
The formulation index $\mathsf X$ enters~\eqref{eq:cg-SP-generic-path} only through
the numerical values of $\theta^{\mathsf X}_d$ and $\gamma^{\mathsf X}_t$: the
feasible set $\Pi_d$, the per-period bounded-hop path structure, the
$\texttt{maxSeg}$ restriction absorbed in it, and the reconfiguration chain
coupling consecutive periods are the same objects in both cases.

\subsubsection{The overlap-reward structure}
\label{sec:cg-overlap}

The overlap identity is a property of segment paths alone and carries no
formulation index. For a segment path~$p$, let
$\segs(p)=\{(i,j):\delta(p,ij)=1\}$ be its set of segments, so $|\segs(p)|$
is their number and $\segs(p)\cap\segs(q)$ the segments shared by two paths.
Since each $\delta(p,ij)\in\{0,1\}$, one has
$|\delta(p,ij)-\delta(q,ij)|=\delta(p,ij)+\delta(q,ij)
-2\,\delta(p,ij)\delta(q,ij)$ on every pair $(i,j)$; summing over
$(i,j)\in V\times V$ turns the reconfiguration
cost~\eqref{eq:cg-beta} into an overlap form,
\begin{equation}\label{eq:cg-dist-overlap}
\dist(p,q)=|\segs(p)|+|\segs(q)|-2\,\bigl|\segs(p)\cap\segs(q)\bigr|.
\end{equation}

This identity exposes the structure of the pricing
problem~\eqref{eq:cg-SP-generic-path} and motivates the oracles. Write
$e=(i,j)$ for a segment and $\theta^{\mathsf X}_d(e,t)$ for its dual
price~\eqref{eq:cg-theta-generic}, which depends on the formulation
$\mathsf X\in\{\textsc{Stela},\textsc{Carla}\}$ through the load duals
$\eta^{\mathsf X}_{a,t}$. Using~\eqref{eq:cg-dist-overlap} on each chain
link and collecting the per-period segment charges,
\begin{equation}
\label{eq:cg-overlap}
\sum_{t\in T}\theta^{\mathsf X}_d(p^{(t)},t)
+\sum_{t\in T^*}\gamma^{\mathsf X}_t\,\dist\bigl(p^{(t-1)},p^{(t)}\bigr)
=\sum_{t\in T}\sum_{e\in\segs(p^{(t)})}\hat\theta^{\,\mathsf X,t}_e
-\;2\!\sum_{t\in T^*}\!\gamma^{\mathsf X}_t
  \,\bigl|\segs(p^{(t-1)})\cap\segs(p^{(t)})\bigr|,
\end{equation}
where the \emph{augmented per-period cost}
$\hat\theta^{\,\mathsf X,t}_e$ charges each segment its dual price
$\theta^{\mathsf X}_d(e,t)$ plus the budget price $\gamma^{\mathsf X}_\cdot$
of every chain link adjacent to period~$t$: both $\gamma^{\mathsf X}_t$ and
$\gamma^{\mathsf X}_{t+1}$ for an interior period, and only the single
adjacent one for $t=0$ or $t=h{-}1$. The per-period terms are $h$
independent bounded-hop shortest-path problems with nonnegative costs; the
entire coupling is the sum of consecutive overlap rewards. The pricing
problem~\eqref{eq:cg-SP-generic-path} is thus: \emph{choose one cheap
bounded-hop path per period so that consecutive paths overlap as much as
possible.} This is why a greedy per-period choice is incomplete, a path
individually unattractive can become optimal because it shares segments with
a neighbour, and why no separable single-period oracle is exact.

\paragraph{Solving by column generation.}

Fix $\mathsf X\in\{\textsc{Stela},\textsc{Carla}\}$. The LP relaxation of
$\mathrm{RMP}^{\mathsf X}_1$ replaces $\xi_{d,\pi}\in\{0,1\}$ by
$\xi_{d,\pi}\ge0$; the selection rows $(\mathrm S)$ then force
$\xi_{d,\pi}\le1$ automatically. It is solved over the pools and grown by
\emph{pricing}: given the optimal duals
$(\alpha^{\mathsf X}_d,\eta^{\mathsf X}_{a,t},\gamma^{\mathsf X}_t)$ of the
current restricted master, the effective segment costs
$\theta^{\mathsf X}_d$ of~\eqref{eq:cg-theta-generic} and the budget prices
$\gamma^{\mathsf X}_t$ are fed to the pricing
problem~$(\mathrm{SP}^{\mathsf X}_d)$
of~\eqref{eq:cg-SP-generic-path}, which searches $\Pi_d$ for a trajectory of
negative reduced cost~\eqref{eq:cg-rc-sidebyside}; any such column is
appended to $\Pi^{\mathrm{pool}}_d$, and the loop repeats until no demand
admits one (\Cref{cor:lp}). Because the search ranges over~$\Pi_d$, every
generated column satisfies~\eqref{ctn:flow} and~\eqref{ctn:segments} by
construction: the two absorbed constraints are enforced at generation time
rather than in the master.

The pricing problem $(\mathrm{SP}^{\mathsf X}_d)$ and the oracles that solve
it are formulation-independent up to their input data: the same algorithms
apply to both $\mathsf X=\textsc{Stela}$ and $\mathsf X=\textsc{Carla}$, the
sole difference being the numerical values of $\theta^{\mathsf X}_d$ and
$\gamma^{\mathsf X}_t$ extracted from the respective restricted masters.

\subsection{Three pricing oracles}
\label{sec:cg-pricing}

We develop the oracle for $(\mathrm{SP}^{\mathsf X}_d)$
of~\eqref{eq:cg-SP-generic-path} incrementally, along a speed--rigour
spectrum. We start from a fast, fully heuristic top-$K$ search
(\textsc{TCG-TopK}); we then refine it into an exact-diagonal search with a
top-$K$ completion (\textsc{TCG-Diag}), which recovers exactness on the
diagonal regime while keeping the off-diagonal part heuristic; and we finally
provide a fully exact reformulation, the layered MIP (\textsc{TCG-Exact}),
which prices $(\mathrm{SP}^{\mathsf X}_d)$ to certified optimality. Each
oracle strengthens the guarantee of the previous one at an added
computational cost, so the three together let pricing trade speed against
rigour.

\begin{center}
\begin{tabular}{@{}lllll@{}}
\toprule
Method & Pricing oracle & Guarantee & Per-demand cost & Subsections \\
\midrule
\textsc{TCG-TopK} & fully heuristic top-$K$ chain & heuristic~\xmark
  & $O(hK_{\mathrm{cand}}^2)$ & \ref{sec:cg-topk} \\
\textsc{TCG-Diag} & diagonal $\oplus$ top-$K$ chain off-diagonal
  & heuristic~\xmark
  & diag.\ SP $+\;O(hK_{\mathrm{cand}}^2)$
  & \ref{sec:cg-diag}--\ref{sec:cg-topk} \\
\textsc{TCG-Exact} & exact layered MIP & exact~\cmark
  & one layered MIP, $O(hK|\Es|)$ & \ref{sec:cg-exact} \\
\bottomrule
\end{tabular}
\end{center}

\noindent
All three produce routing schemes that are optimal over their generated pool;
what distinguishes them is what they can \emph{certify} about the full
trajectory set, and at what computational price. Since each oracle takes as
input a nonnegative segment cost $\theta^{\mathsf X}_d(\cdot,\cdot,t)$ and a
nonnegative budget price $\gamma^{\mathsf X}_t$, and is otherwise indifferent
to how those prices were produced, a single implementation serves both
$\mathsf X=\textsc{Stela}$ and $\mathsf X=\textsc{Carla}$. All oracles run
verbatim with the stratified rank-$k$ costs $\theta^{\mathsf X}_{d,k}$ of
\Cref{sec:cg-lex}; we develop them here for the rank-$1$ (MLU) duals
$\theta^{\mathsf X}_d$.

\subsubsection{\textsc{TCG-TopK}: fully heuristic top-$K$ pricing}
\label{sec:cg-topk}

The overlap-reward structure~\eqref{eq:cg-overlap} shows that the pricing
problem $(\mathrm{SP}^{\mathsf X}_d)$ of~\eqref{eq:cg-SP-generic-path}
decomposes into $h$ bounded-hop shortest-path problems coupled only through
consecutive overlaps. The simplest oracle restricts each period to a short
list of promising paths and optimises the chain over those lists.

\emph{Candidate lists.} For each period $t\in T$ and demand $d$, let
$\mathcal P_{d,t}^{K_{\mathrm{cand}}}$ denote the set of
$K_{\mathrm{cand}}$ segment paths of smallest per-period cost
$\theta^{\mathsf X}_d(p,t)=\sum_{(i,j)\in\segs(p)}\theta^{\mathsf X}_d(i,j,t)$
among all $s_d$--$t_d$ paths in $G_t$ with at most $\texttt{maxSeg}$
segments. These are obtained by a standard top-$K$ bounded-hop shortest-path
computation.

\emph{Chain dynamic program.} Over the product
$\mathcal P_{d,0}^{K_{\mathrm{cand}}}\times\cdots\times
\mathcal P_{d,h-1}^{K_{\mathrm{cand}}}$, the overlap coupling
of~\eqref{eq:cg-overlap} is optimised exactly by a forward dynamic program
(\Cref{alg:cg-topk}). Let $D[t,p]$ be the minimum pricing cost of a partial
trajectory from period~$0$ to period~$t$ ending with path~$p$. The base case
is $D[0,p]=\theta^{\mathsf X}_d(p,0)$ for each
$p\in\mathcal P_{d,0}^{K_{\mathrm{cand}}}$. The transition from period
$t{-}1$ to period $t$ adds the per-period path cost and the reconfiguration
cost of the chain link:
\[
D[t,p]=\theta^{\mathsf X}_d(p,t)
+\min_{p'\in\mathcal P_{d,t-1}^{K_{\mathrm{cand}}}}
\bigl\{D[t{-}1,p']+\gamma^{\mathsf X}_t\,\dist(p',p)\bigr\},
\qquad p\in\mathcal P_{d,t}^{K_{\mathrm{cand}}}.
\]
The optimum $v_{\mathrm{topK}}=\min_{p}D[h{-}1,p]$, together with a
standard backtracking pass, yields the best trajectory
$\pi_{\mathrm{topK}}$ whose every period lies in its candidate list, with
reduced cost $-\alpha^{\mathsf X}_d+v_{\mathrm{topK}}$. We call the column
generation driven by this oracle \textsc{TCG-TopK}.

\emph{Guarantees.} \textsc{TCG-TopK} runs in $O(h\,K_{\mathrm{cand}}^2)$
per demand, since each of the $h{-}1$ transitions evaluates
$K_{\mathrm{cand}}^2$ predecessor--successor pairs. It is \emph{sound}:
every trajectory it returns is a genuine sequence of segment paths, scored by
its exact reduced cost~\eqref{eq:cg-rc-sidebyside}, so any
negative-reduced-cost column it reports is truly improving. It is not
\emph{complete}: the globally optimal trajectory may use, in some period, a
path outside that period's top-$K$ list, in which case the oracle misses it.
Consequently a pricing round in which \textsc{TCG-TopK} finds no improving
column does \emph{not} certify LP optimality; the stopping guarantee of
\Cref{cor:lp} requires an oracle that searches all of~$\Pi_d$, which is the
role of the exact oracle developed later.

\begin{algorithm}[H]
\caption{\textsc{TCG-TopK}: heuristic top-$K$ chain pricing for demand~$d$}
\label{alg:cg-topk}
\begin{algorithmic}[1]
\Require dual prices
$(\alpha^{\mathsf X}_d,\,\theta^{\mathsf X}_d,\,\gamma^{\mathsf X})$,
candidate-list size $K_{\mathrm{cand}}$, segment bound $\texttt{maxSeg}$.
\For{$t\in T$}
  \State Compute the $K_{\mathrm{cand}}$ cheapest $s_d$--$t_d$ paths in $G_t$
         under $\theta^{\mathsf X}_d(\cdot,t)$ with $\le\texttt{maxSeg}$
         segments $\to\mathcal P_{d,t}^{K_{\mathrm{cand}}}$.
\EndFor
\State $D[0,p]\gets\theta^{\mathsf X}_d(p,0)$ for each
       $p\in\mathcal P_{d,0}^{K_{\mathrm{cand}}}$.
\For{$t=1,\dots,h-1$}
  \For{$p\in\mathcal P_{d,t}^{K_{\mathrm{cand}}}$}
    \State $D[t,p]\gets\theta^{\mathsf X}_d(p,t)
            +\min_{p'\in\mathcal P_{d,t-1}^{K_{\mathrm{cand}}}}
            \bigl\{D[t{-}1,p']
            +\gamma^{\mathsf X}_t\,\dist(p',p)\bigr\}$
  \EndFor
\EndFor
\State $v_{\mathrm{topK}}\gets\min_{p}D[h{-}1,p]$; backtrack
       $\to\pi_{\mathrm{topK}}$.
\State \Return $\pi_{\mathrm{topK}}$ with reduced cost
       $-\alpha^{\mathsf X}_d+v_{\mathrm{topK}}$.
\end{algorithmic}
\end{algorithm}

\subsubsection{\textsc{TCG-Diag}: exact diagonal search with top-$K$ completion
for the off-diagonal}
\label{sec:cg-diag}

Recall the pricing objective in overlap form~\eqref{eq:cg-overlap}: choose
one bounded-hop path per period minimising the per-period costs
$\theta^{\mathsf X}_d(\cdot,t)$ minus the reward
$2\gamma^{\mathsf X}_t\,|\segs(p^{(t-1)})\cap\segs(p^{(t)})|$ for overlap
between consecutive periods. \textsc{TCG-TopK} keeps, in each period, only
the paths that are cheap \emph{on their own}; it therefore misses exactly the
trajectories that are attractive only through the overlap reward, paths
individually mediocre in some period but heavily shared with their neighbours.
The extreme such case is a trajectory that uses the \emph{same} path in every
period: its overlap reward is maximal (full overlap at every chain link, so
every $\dist=0$), yet that path need not be top-ranked in any single period,
so \textsc{TCG-TopK} can overlook it entirely.

This motivates a two-part oracle. We first solve the same-path
(\emph{diagonal}) regime \emph{exactly}, it is where \textsc{TCG-TopK} is
weakest and, conveniently, where the chain decouples into a single shortest
path. We then run the top-$K$ chain search of \Cref{sec:cg-topk} for the
remaining, off-diagonal trajectories, and keep the better of the two. The
diagonal part removes \textsc{TCG-TopK}'s main blind spot at negligible cost;
the top-$K$ part covers trajectories that do vary across periods. We call the
resulting column generation \textsc{TCG-Diag}.

\paragraph{Diagonal regime (exact).}

Restrict attention to trajectories using the \emph{same} path in every
period, $p^{(0)}=\cdots=p^{(h-1)}=p$, so that every $\dist(p,p)=0$ and the
chain decouples. On the diagonal the pricing
objective~\eqref{eq:cg-SP-generic-path} reduces to
$\sum_{t\in T}\theta^{\mathsf X}_d(p,t)
=\sum_{(i,j)\in\segs(p)}\theta^{\mathsf X,\cap}_d(i,j)$,
where
\begin{equation}\label{eq:cg-theta-diag}
\theta^{\mathsf X,\cap}_d(i,j)
\;:=\;\sum_{t\in T}\theta^{\mathsf X}_d(i,j,t)
\end{equation}
is the \emph{aggregated diagonal segment cost}, summing the dual prices of
segment $(i,j)$ over all periods. The path $p$ must be feasible in every
$G_t$ simultaneously; we therefore define the \emph{all-periods segment
graph} $H_d^{\cap}$ as the directed graph on $V$ whose arc set is
\begin{equation}\label{eq:cg-Hcap}
A(H_d^{\cap})
\;:=\;\bigl\{(i,j)\in V\times V\ :\ i\neq j,\;
\text{$(i,j)$ is a valid segment in $G_t$ for every $t\in T$}\bigr\},
\end{equation}
where $(i,j)$ is a valid segment in $G_t$ if $j$ is reachable from $i$ by a
shortest path in $G_t$ (equivalently, the forwarding graph $FG_t(i,j)$ is
nonempty). The diagonal problem is then a bounded-hop shortest-path problem
on $H_d^{\cap}$.

\begin{proposition}
\label{prop:cg-diagonal}
The diagonal restriction of $(\mathrm{SP}^{\mathsf X}_d)$ is
$\min_{p}\sum_{(i,j)\in\segs(p)}\theta^{\mathsf X,\cap}_d(i,j)$ over
$s_d$--$t_d$ paths $p$ in $H_d^{\cap}$ with at most $\texttt{maxSeg}$
segments. It is solved exactly by the recurrence
\[
D[v,0]=\begin{cases}0,&v=s_d,\\+\infty,&v\neq s_d,\end{cases}
\qquad
D[j,k{+}1]=\min_{(i,j)\in A(H_d^{\cap})}
\bigl\{D[i,k]+\theta^{\mathsf X,\cap}_d(i,j)\bigr\},
\]
for $k=0,\dots,\texttt{maxSeg}-1$, where $D[v,k]$ is the minimum cost of
reaching node~$v$ from $s_d$ using exactly $k$ segments in $H_d^{\cap}$. The
diagonal value is $v_{\mathrm{diag}}=\min_{0\le k\le\texttt{maxSeg}}D[t_d,k]$
and the reduced cost is $-\alpha^{\mathsf X}_d+v_{\mathrm{diag}}$. The
complexity is $O(\texttt{maxSeg}\cdot|A(H_d^{\cap})|)$, i.e.\
$O(\texttt{maxSeg}\cdot|V|^2)$ on dense segment graphs, independently
of~$h$.
\end{proposition}

\begin{proof} \[ \begin{array}{:c} \begin{minipage}{0.97\textwidth}
On the diagonal every reconfiguration term vanishes, so the objective reduces
to the additive segment cost $\theta^{\mathsf X,\cap}_d$ over bounded-hop
$s_d$--$t_d$ paths. The recurrence is the standard Bellman--Ford iteration on
the layered graph indexed by the hop count: at layer $k{+}1$ one extends
every $k$-segment partial path by one arc of $H_d^{\cap}$. Each of the
$\texttt{maxSeg}$ layers relaxes every arc once. The formulation index
$\mathsf X$ enters only through the numerical values of
$\theta^{\mathsf X,\cap}_d$; the graph $H_d^{\cap}$ and the bounded-hop
structure are formulation-independent.
\end{minipage} \end{array} \]\end{proof}

The diagonal optimum is exact over same-path trajectories and cheap; it also
upper-bounds $\operatorname{val}(\mathrm{SP}^{\mathsf X}_d)$ and seeds the
incumbent of the off-diagonal search below.

\paragraph{Off-diagonal completion (top-$K$ along the chain).}

Off the diagonal, \textsc{TCG-Diag} falls back on the candidate-list
strategy of \textsc{TCG-TopK}. For each period $t\in T$ the same candidate
lists $\mathcal P_{d,t}^{K_{\mathrm{cand}}}$ of \Cref{sec:cg-topk} are
computed under the single-period cost $\theta^{\mathsf X}_d(\cdot,t)$, and
the chain dynamic program of \Cref{alg:cg-topk} returns the best trajectory
$\pi_{\mathrm{topK}}$ with value $v_{\mathrm{topK}}$ over those lists.
\textsc{TCG-Diag} then returns the better of the two:

\begin{algorithm}[H]
\caption{\textsc{TCG-Diag}: exact diagonal $\oplus$ top-$K$ chain pricing
for demand~$d$}
\label{alg:cg-diag}
\begin{algorithmic}[1]
\Require dual prices
$(\alpha^{\mathsf X}_d,\,\theta^{\mathsf X}_d,\,\gamma^{\mathsf X})$,
candidate-list size $K_{\mathrm{cand}}$, segment bound $\texttt{maxSeg}$.
\Statex\hrulefill\ \textsc{Diagonal part}\ \hrulefill
\State Build $H_d^{\cap}$ and compute
       $\theta^{\mathsf X,\cap}_d(i,j)=\sum_{t\in T}\theta^{\mathsf X}_d(i,j,t)$
       for each $(i,j)\in A(H_d^{\cap})$.
\State Solve the bounded-hop shortest path on $H_d^{\cap}$
       (\Cref{prop:cg-diagonal}) $\to(v_{\mathrm{diag}},\,p^{\star})$;
       set $\pi_{\mathrm{diag}}=(p^{\star},\dots,p^{\star})$.
\Statex\hrulefill\ \textsc{Off-diagonal part}\ \hrulefill
\State Run the chain DP of \Cref{alg:cg-topk}
       $\to(v_{\mathrm{topK}},\,\pi_{\mathrm{topK}})$.
\Statex\hrulefill\ \textsc{Selection}\ \hrulefill
\State \textbf{if} $v_{\mathrm{diag}}\le v_{\mathrm{topK}}$ \textbf{then}
       $(\pi,v)\gets(\pi_{\mathrm{diag}},v_{\mathrm{diag}})$
       \textbf{else}
       $(\pi,v)\gets(\pi_{\mathrm{topK}},v_{\mathrm{topK}})$.
\State \Return $\pi$ with reduced cost $-\alpha^{\mathsf X}_d+v$.
\end{algorithmic}
\end{algorithm}

\subsubsection{\textsc{TCG-Exact}: exact layered-MIP pricing}
\label{sec:cg-exact}

The coupling in $(\mathrm{SP}^{\mathsf X}_d)$ is entirely due to the
reconfiguration chain. A complete binary formulation is obtained through a
layered representation: one layered block per period and one distance block
per chain link. We first introduce the notation specific to this formulation,
then state the MIP.

\paragraph{Segment sets.}

Fix a demand $d$ and write $s=s_d$, $r=t_d$, $K:=\texttt{maxSeg}$. Recall
from~\eqref{eq:cg-Hcap} that a pair $(i,j)$ with $i\neq j$ is a
\emph{valid segment} in the period-$t$ graph $G_t$ whenever $j$ is reachable
from $i$ by a shortest path in $G_t$ (equivalently, the forwarding graph
$FG_t(i,j)$ is nonempty). For each period $t\in T$ define the
\emph{per-period segment set}
\begin{equation}\label{eq:cg-Et}
E_d(t)\;:=\;\bigl\{(i,j)\in V\times V\ :\ i\neq j,\;
(i,j)\text{ is a valid segment in }G_t\bigr\},
\end{equation}
and the \emph{extended segment set}
\begin{equation}\label{eq:cg-Eext}
\overline E_d(t)\;:=\;\bigl\{(i,j)\in E_d(t)\ :\ i\neq r\bigr\}
\;\cup\;\{(r,r)\},
\end{equation}
where the artificial segment $(r,r)$ is a dummy with zero cost: it pads
paths that reach the destination before using all $K$ slots, and since all
real outgoing segments from $r$ are excluded, once a path reaches $r$ it
stays there. Finally, let
\begin{equation}\label{eq:cg-Eall}
E_d\;:=\;\bigcup_{t\in T}E_d(t)
\end{equation}
be the union of the segment sets over all periods. Note that $E_d(t)$ is the
per-period analogue of the all-periods arc set
$A(H_d^{\cap})=\bigcap_{t\in T}E_d(t)$ of~\eqref{eq:cg-Hcap} used in the
diagonal oracle.

\paragraph{Variables.}

Three families of binary variables encode one trajectory:
\begin{itemize}\itemsep2pt
\item \emph{Layer variables}
      $y^t_{\ell,e}\in\{0,1\}$
      for each period $t\in T$, layer $\ell=1,\dots,K$ and segment
      $e\in\overline E_d(t)$: $y^t_{\ell,e}=1$ if segment $e$ is assigned to
      layer $\ell$ of the period-$t$ path.
\item \emph{Usage indicators}
      $x^t_e\in\{0,1\}$
      for each period $t\in T$ and real segment $e\in E_d$:
      $x^t_e=1$ if $e$ appears somewhere in the period-$t$ path (in any
      layer).
\item \emph{Change indicators}
      $h^t_e\in\{0,1\}$
      for each chain link $t\in T^*$ and segment $e\in E_d$:
      $h^t_e=1$ if the usage of $e$ differs between periods $t{-}1$ and $t$,
      i.e.\ $h^t_e=|x^{t-1}_e-x^t_e|$.
\end{itemize}

\paragraph{Formulation (EP).}

With the segment costs $\theta^{\mathsf X}_d(e,t)$ of~\eqref{eq:cg-theta-generic}
and the budget prices $\gamma^{\mathsf X}_t$ extracted from the dual of
$\mathrm{RMP}^{\mathsf X}_1$, the exact pricing MIP reads
\begin{subequations}\label{eq:ep}
\begin{align}
\min\quad
& \sum_{t\in T}\sum_{\ell=1}^{K}\sum_{e\in\overline E_d(t)}
  \theta^{\mathsf X}_d(e,t)\,y^t_{\ell,e}
  +\sum_{t\in T^*}\gamma^{\mathsf X}_t\sum_{e\in E_d}h^t_e
\notag\\
\text{s.t.}\quad
& \sum_{e\in\overline E_d(t)}y^t_{\ell,e}=1
&& t\in T,\ \ell=1,\dots,K,
\label{ep:layer}\\
& \sum_{e=(s,j)\in\overline E_d(t)}y^t_{1,e}=1
&& t\in T,
\label{ep:start}\\
& \sum_{e=(i,v)\in\overline E_d(t)}y^t_{\ell,e}
  =\sum_{e=(v,j)\in\overline E_d(t)}y^t_{\ell+1,e}
&& t\in T,\ \ell=1,\dots,K{-}1,\ v\in V,
\label{ep:cont}\\
& \sum_{e=(i,r)\in\overline E_d(t)}y^t_{K,e}=1
&& t\in T,
\label{ep:sink}\\[2pt]
& \sum_{\ell=1}^{K}y^t_{\ell,e}\le 1
&& t\in T,\ e\in E_d(t),
\label{ep:simple}\\[2pt]
& x^t_e\ge y^t_{\ell,e}
&& t\in T,\ e\in E_d(t),\ \ell=1,\dots,K,
\label{ep:xlow}\\
& x^t_e\le\sum_{\ell=1}^{K}y^t_{\ell,e}
&& t\in T,\ e\in E_d(t),
\label{ep:xup}\\
& x^t_e=0
&& t\in T,\ e\in E_d\setminus E_d(t),
\label{ep:xzero}\\
& h^t_e\ge x^{t-1}_e-x^t_e,\quad h^t_e\ge x^t_e-x^{t-1}_e
&& t\in T^*,\ e\in E_d,
\label{ep:hlow}\\
& h^t_e\le x^{t-1}_e+x^t_e,\quad h^t_e\le 2-x^{t-1}_e-x^t_e
&& t\in T^*,\ e\in E_d,
\label{ep:hup}\\
& y^t_{\ell,e},\;x^t_e,\;h^t_e\in\{0,1\}.
&& \notag
\end{align}
\end{subequations}

\emph{Reading the constraints.}
Constraints~\eqref{ep:layer} choose exactly one segment per layer;
\eqref{ep:start} starts each period-$t$ walk at $s=s_d$;
\eqref{ep:cont} enforces continuity (the head of layer $\ell$ equals the tail
of layer $\ell{+}1$); \eqref{ep:sink} requires the layer-$K$ segment to end
at $r=t_d$. Together these four families build a walk from $s$ to $r$ through
$K$ layers. Constraint~\eqref{ep:simple} forbids the same real segment in
more than one layer, so each period path is a trail (no repeated segment; a
repeated node is still allowed). The dummy segment $(r,r)$ is exempt
from~\eqref{ep:simple} and can fill arbitrarily many layers, which is how
paths shorter than $K$ segments are represented.

Constraints~\eqref{ep:xlow}--\eqref{ep:xup} link the layer variables to the
usage indicators: $x^t_e=1$ if and only if $e$ is selected in some layer at
period $t$. Constraint~\eqref{ep:xzero} forces $x^t_e=0$ for segments not
valid in $G_t$. Constraints~\eqref{ep:hlow}--\eqref{ep:hup} enforce
$h^t_e=|x^{t-1}_e-x^t_e|$ on each chain link, so that the reconfiguration
cost $\sum_{e\in E_d}h^t_e=\dist(p^{(t-1)},p^{(t)})$ is computed exactly.

\emph{Size.} The formulation has $O(h\,K\,|E_d|)$ variables and constraints,
i.e.\ size \emph{linear} in the horizon $h$.

\emph{Objective.} The two terms of the objective mirror the pricing
problem~\eqref{eq:cg-SP-generic-path}: the first sums the segment costs
$\theta^{\mathsf X}_d(e,t)$ over the selected layers, and the second charges
the reconfiguration cost at each chain link with the budget price
$\gamma^{\mathsf X}_t$. As in the other oracles, the formulation index
$\mathsf X$ enters only through the numerical values of these prices; the
constraint structure is formulation-independent.

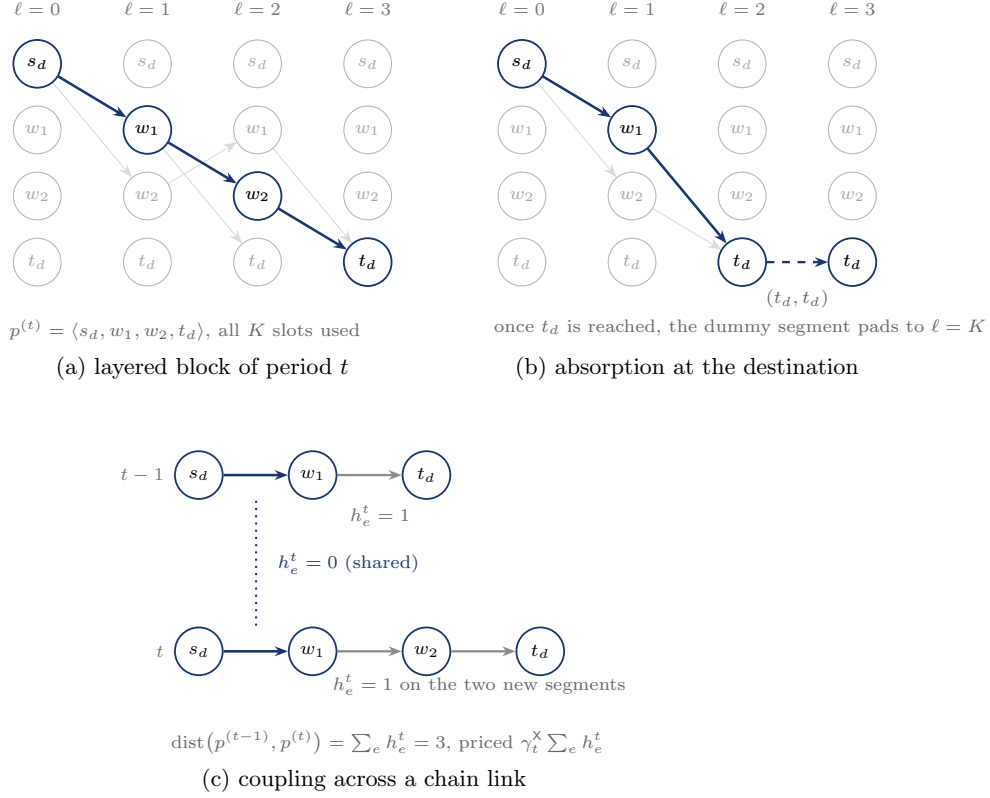
\begin{figure*}[t]
\centering
\resizebox{0.8\textwidth}{!}{%
\begin{tikzpicture}[
  >={Stealth[length=1.8mm]},
  st/.style={circle,draw=black!30,inner sep=0pt,minimum size=6.5mm,
             font=\scriptsize,text=black!35},
  sel/.style={circle,draw=acc,line width=0.7pt,inner sep=0pt,minimum size=6.5mm,
              font=\scriptsize,text=black},
  bg/.style={->,draw=black!15,thin},
  route/.style={->,draw=acc,line width=1pt},
  dum/.style={->,draw=acc,dashed,line width=0.9pt},
  lay/.style={font=\scriptsize,text=black!55},
  subcap/.style={font=\small},
  ann/.style={font=\scriptsize,text=black!60}
]

\begin{scope}
\foreach \l in {0,1,2,3}{
  \node[lay] at (1.5*\l,0.75) {$\ell=\l$};
  \foreach \v/\y/\lb in {s/0/{s_d},u/-0.9/{w_1},w/-1.8/{w_2},d/-2.7/{t_d}}{
    \node[st] (A-\v-\l) at (1.5*\l,\y) {$\lb$};}}
\draw[bg] (A-s-0) -- (A-w-1);
\draw[bg] (A-u-1) -- (A-d-2);
\draw[bg] (A-w-1) -- (A-u-2);
\draw[bg] (A-u-2) -- (A-d-3);
\draw[route] (A-s-0) -- (A-u-1);
\draw[route] (A-u-1) -- (A-w-2);
\draw[route] (A-w-2) -- (A-d-3);
\node[sel] at (A-s-0) {$s_d$}; \node[sel] at (A-u-1) {$w_1$};
\node[sel] at (A-w-2) {$w_2$}; \node[sel] at (A-d-3) {$t_d$};
\node[ann,anchor=north west] at (-0.5,-3.35)
  {$p^{(t)}=\langle s_d,w_1,w_2,t_d\rangle$, all $K$ slots used};
\node[subcap] at (2.25,-4.15) {(a) layered block of period $t$};
\end{scope}

\begin{scope}[xshift=6.6cm]
\foreach \l in {0,1,2,3}{
  \node[lay] at (1.5*\l,0.75) {$\ell=\l$};
  \foreach \v/\y/\lb in {s/0/{s_d},u/-0.9/{w_1},w/-1.8/{w_2},d/-2.7/{t_d}}{
    \node[st] (B-\v-\l) at (1.5*\l,\y) {$\lb$};}}
\draw[bg] (B-s-0) -- (B-w-1);
\draw[bg] (B-w-1) -- (B-d-2);
\draw[route] (B-s-0) -- (B-u-1);
\draw[route] (B-u-1) -- (B-d-2);
\draw[dum]  (B-d-2) -- (B-d-3);
\node[sel] at (B-s-0) {$s_d$}; \node[sel] at (B-u-1) {$w_1$};
\node[sel] at (B-d-2) {$t_d$}; \node[sel] at (B-d-3) {$t_d$};
\node[ann,anchor=north] at (3.75,-2.95) {$(t_d,t_d)$};
\node[ann,anchor=north west] at (-0.5,-3.35)
  {once $t_d$ is reached, the dummy segment pads to $\ell=K$};
\node[subcap] at (2.25,-4.15) {(b) absorption at the destination};
\end{scope}

\begin{scope}[xshift=2.2cm,yshift=-5.6cm]
\foreach \i/\x/\lb in {0/0/{s_d},1/1.55/{w_1},2/3.1/{t_d}}{
  \node[sel] (P-\i) at (\x,0) {$\lb$};}
\foreach \i/\x/\lb in {0/0/{s_d},1/1.55/{w_1},2/3.1/{w_2},3/4.65/{t_d}}{
  \node[sel] (Q-\i) at (\x,-2.4) {$\lb$};}
\draw[route] (P-0) -- (P-1);
\draw[->,draw=black!45,line width=0.9pt] (P-1) -- (P-2);
\draw[route] (Q-0) -- (Q-1);
\draw[->,draw=black!45,line width=0.9pt] (Q-1) -- (Q-2);
\draw[->,draw=black!45,line width=0.9pt] (Q-2) -- (Q-3);
\draw[dotted,draw=acc,line width=0.8pt] (0.78,-0.35) -- (0.78,-2.05);
\node[ann,anchor=west,text=acc] at (0.95,-1.2) {$h^t_e=0$ (shared)};
\node[ann,anchor=east] at (3.0,-0.55) {$h^t_e=1$};
\node[ann,anchor=west] at (1.7,-2.85) {$h^t_e=1$ on the two new segments};
\node[lay,anchor=east] at (-0.35,0)    {$t-1$};
\node[lay,anchor=east] at (-0.35,-2.4) {$t$};
\node[ann,anchor=north west] at (-0.5,-3.35)
  {$\dist\bigl(p^{(t-1)},p^{(t)}\bigr)=\sum_e h^t_e=3$, priced
   $\gamma^{\mathsf X}_t\sum_e h^t_e$};
\node[subcap] at (2.25,-4.15) {(c) coupling across a chain link};
\end{scope}

\end{tikzpicture}%
}
\caption{Structure of the exact pricing formulation (EP) for a demand $d$,
with $K=\texttt{maxSeg}=3$. (a)~Each period $t$ owns one layered block: a
copy of the router set per layer $\ell=0,\dots,K$, and $y^t_{\ell,e}=1$
selects one segment per layer
(\eqref{ep:layer}--\eqref{ep:cont}), so a feasible assignment is a walk
leaving $s_d$ that must be at $t_d$ at layer $K$ (\eqref{ep:sink});
\eqref{ep:simple} forbids reusing a real segment, making the walk a trail.
(b)~When $t_d$ is reached before layer $K$, the dummy segment $(t_d,t_d)$
is the only one leaving $t_d$ in
$\overline E_d(t)$~\eqref{eq:cg-Eext}, so the trajectory stays at $t_d$ for
all remaining layers at zero cost. (c)~Blocks of consecutive periods are
coupled only through the change indicators: $h^t_e=|x^{t-1}_e-x^t_e|$ is
zero on shared segments and one on changed ones, and the reconfiguration is
\emph{priced} by $\gamma^{\mathsf X}_t$ rather than capped, which is what
keeps (EP) linear in the horizon.}
\label{fig:cg-layered}
\end{figure*}

The correctness of (EP) is established in three steps: the layered block
encodes exactly the admissible paths of one period, the usage variables
aggregate those layers correctly, and the distance block linearises the
reconfiguration cost. The first step is the following.

\begin{lemma}\label{lem:layer}
For each period $t\in T$, constraints
\eqref{ep:layer}--\eqref{ep:sink} together with~\eqref{ep:simple} are in
one-to-one correspondence with the trails in $\Pi^t_d$ padded with dummy
segments up to length $K$.
\end{lemma}

\begin{proof} \[ \begin{array}{:c} \begin{minipage}{0.97\textwidth}
Let $y^t$ satisfy \eqref{ep:layer}--\eqref{ep:sink}. By~\eqref{ep:layer}
exactly one segment is chosen per layer. By~\eqref{ep:start} the first
leaves $s$. By~\eqref{ep:cont}, the segment at layer $\ell{+}1$ leaves the
node entered at layer $\ell$; the selected segments form a directed walk
from $s$. By~\eqref{ep:sink} the last enters $r$, and since $(r,r)$ is the
only segment leaving $r$ in
$\overline E_d(t)$~\eqref{eq:cg-Eext}, every layer after the first visit
to $r$ uses the dummy segment. Removing dummies yields an $s$--$r$ walk with
at most $K$ real segments; \eqref{ep:simple} makes each real segment appear
at most once, so the walk is a trail. Conversely, a trail
$p=(v_0,\dots,v_q)\in\Pi^t_d$ with $q\le K$ has pairwise distinct segments;
setting $y^t_{\ell,(v_{\ell-1},v_\ell)}=1$ for $\ell\le q$ and
$y^t_{\ell,(r,r)}=1$ for $\ell>q$ satisfies
\eqref{ep:layer}--\eqref{ep:simple}.
\end{minipage} \end{array} \]\end{proof}

The layered block certifies \emph{which} paths are representable, but the
objective and the budget are written in terms of the usage indicators $x^t_e$
rather than the layer variables. The next lemma checks that the two agree:
$x^t_e$ is exactly the indicator that segment $e$ belongs to the period-$t$
path, which is where the trail constraint~\eqref{ep:simple} earns its place.

\begin{lemma}\label{lem:usage}
Under~\eqref{ep:simple}, constraints \eqref{ep:xlow}--\eqref{ep:xup} force,
for every period $t\in T$ and real segment $e\in E_d(t)$,
\[
x^t_e\;=\;\sum_{\ell=1}^{K}y^t_{\ell,e}\;=\;\mathbf 1\bigl[e\in\segs(p^{(t)})\bigr].
\]
\end{lemma}

\begin{proof} \[ \begin{array}{:c} \begin{minipage}{0.97\textwidth}
By~\eqref{ep:simple}, $\sum_\ell y^t_{\ell,e}\in\{0,1\}$. If it is $0$,
\eqref{ep:xup} gives $x^t_e\le0$; if it is $1$, \eqref{ep:xlow} gives
$x^t_e\ge1$ and \eqref{ep:xup} gives $x^t_e\le1$. In both cases
$x^t_e=\sum_\ell y^t_{\ell,e}$, which equals
$\mathbf 1[e\in\segs(p^{(t)})]$ by \cref{lem:layer}.
\end{minipage} \end{array} \]\end{proof}

It remains to handle the term coupling consecutive periods. The
reconfiguration cost is an absolute value, which is not linear; the following
lemma shows that the four inequalities~\eqref{ep:hlow}--\eqref{ep:hup}
reproduce it exactly on binary usages, so the chain link costs
$\gamma^{\mathsf X}_t\sum_e h^t_e$ with no additional modelling.

\begin{lemma}\label{lem:dist}
For binary $x^{t-1}_e,x^t_e$ and $t\in T^*$, constraints
\eqref{ep:hlow}--\eqref{ep:hup} with $h^t_e\in\{0,1\}$ are equivalent to
$h^t_e=|x^{t-1}_e-x^t_e|$.
\end{lemma}

\begin{proof} \[ \begin{array}{:c} \begin{minipage}{0.97\textwidth}
Four cases for $(x^{t-1}_e,x^t_e)$. $(0,0)$: \eqref{ep:hup} gives
$h^t_e\le0$. $(1,1)$: the second inequality of~\eqref{ep:hup} gives
$h^t_e\le0$. $(1,0)$: the first of~\eqref{ep:hlow} gives $h^t_e\ge1$.
$(0,1)$: the second of~\eqref{ep:hlow} gives $h^t_e\ge1$. These match
$|x^{t-1}_e-x^t_e|$.
\end{minipage} \end{array} \]\end{proof}

The three lemmas now combine: paths per period, usages per segment, and
distances per chain link. Together they give a bijection between the feasible
points of (EP) and the trajectories of
$(\mathrm{SP}^{\mathsf X}_d)$~\eqref{eq:cg-SP-generic-path} that preserves
the objective, hence exactness of the pricing formulation.

\begin{theorem}\label{thm:exact}
The optimal value of\/ \textup{(EP)} equals
$\operatorname{val}(\mathrm{SP}^{\mathsf X}_d)$. Moreover, every optimal
solution of\/ \textup{(EP)} induces an optimal trajectory for
$(\mathrm{SP}^{\mathsf X}_d)$, and every optimal trajectory admits a feasible
representation in \textup{(EP)} with the same objective.
\end{theorem}

\begin{proof} \[ \begin{array}{:c} \begin{minipage}{0.97\textwidth}
\emph{(EP)$\to(\mathrm{SP}^{\mathsf X}_d)$.} Let $(y,x,h)$ be feasible for
(EP). By \cref{lem:layer} the $y$-variables define, after removal of
dummies, one trail $p^{(t)}$ per period. By \cref{lem:usage},
$x^t_e=\mathbf 1[e\in\segs(p^{(t)})]$ for real $e\in E_d(t)$ (and
$x^t_e=0$ for $e\in E_d\setminus E_d(t)$ by~\eqref{ep:xzero}). By
\cref{lem:dist}, for each $t\in T^*$,
$\sum_{e\in E_d}h^t_e
=\sum_{e\in E_d}|x^{t-1}_e-x^t_e|
=\dist(p^{(t-1)},p^{(t)})$.
For the segment-cost term, using \cref{lem:usage} and the fact that the
dummy segment $(r,r)$ carries zero cost,
\[
\sum_{t\in T}\sum_{\ell=1}^{K}
\sum_{e\in\overline E_d(t)}\theta^{\mathsf X}_d(e,t)\,y^t_{\ell,e}
\;=\;\sum_{t\in T}\sum_{e\in E_d(t)}\theta^{\mathsf X}_d(e,t)\,x^t_e
\;=\;\sum_{t\in T}\theta^{\mathsf X}_d\bigl(p^{(t)},t\bigr),
\]
so the (EP) objective equals the
$(\mathrm{SP}^{\mathsf X}_d)$~objective at
$(p^{(0)},\dots,p^{(h-1)})$.

\emph{$(\mathrm{SP}^{\mathsf X}_d)\to$\textup{(EP)}.} Conversely, any
trajectory is represented via the converse of \cref{lem:layer}, with
$x^t_e=\mathbf 1[e\in\segs(p^{(t)})]$ and $h^t_e=|x^{t-1}_e-x^t_e|$, of
identical objective value. The two maps are mutually inverse and
objective-preserving, so the optima coincide.
\end{minipage} \end{array} \]\end{proof}

\paragraph{Column-generation loop.}

The MLU column generation is then the standard loop, stated for an arbitrary
formulation $\mathsf X\in\{\textsc{Stela},\textsc{Carla}\}$.

\begin{algorithm}[H]
\caption{\textsc{TCG-Exact}: exact trajectory column generation (root LP)
under formulation $\mathsf X$}
\label{alg:cg-exact}
\begin{algorithmic}[1]
\Require formulation $\mathsf X\in\{\textsc{Stela},\textsc{Carla}\}$,
         demands $D$, segment bound $\texttt{maxSeg}$.
\State For each $d$, initialise a nonempty feasible pool
       $\Pi^{\mathrm{pool}}_d$.
\Repeat
  \State Solve the LP relaxation of $\mathrm{RMP}^{\mathsf X}_1$; extract
         optimal duals
         $(\alpha^{\mathsf X}_d,\eta^{\mathsf X}_{a,t},\gamma^{\mathsf X}_t)$.
  \State Build segment costs
         $\theta^{\mathsf X}_d(e,t)$
         via~\eqref{eq:cg-theta-generic} for each demand~$d$.
  \For{each demand $d\in D$}
    \State Solve (EP) for demand $d$
           $\to\operatorname{val}(\mathrm{SP}^{\mathsf X}_d)$ and
           trajectory $\pi^\star_d$.
    \If{$\operatorname{val}(\mathrm{SP}^{\mathsf X}_d)
         <\alpha^{\mathsf X}_d$}
      \State $\Pi^{\mathrm{pool}}_d\gets
             \Pi^{\mathrm{pool}}_d\cup\{\pi^\star_d\}$.
    \EndIf
  \EndFor
\Until{no column is added.}
\end{algorithmic}
\end{algorithm}

Exactness of the oracle (\cref{thm:exact}) settles a single pricing round;
it does not yet say that the loop stops, nor what it has proved when it does.
The next theorem closes both points: the column set being finite,
\cref{alg:cg-exact} cannot iterate forever, and the exact oracle turns its
termination test into the certificate of \cref{cor:lp}. This is what
distinguishes \textsc{TCG-Exact} from the two heuristic oracles, whose
failure to find a column proves nothing.

\begin{theorem}\label{thm:conv}
If each \textup{(EP)} is solved to global optimality,
\cref{alg:cg-exact} terminates finitely, and at termination the
restricted-master LP is optimal for the full trajectory LP relaxation of\/
$\mathrm{RMP}^{\mathsf X}_1$.
\end{theorem}

\begin{proof} \[ \begin{array}{:c} \begin{minipage}{0.97\textwidth}
The full column set $\bigcup_d\Pi_d$ is finite ($V,D,K,h$ finite); each
non-terminal iteration adds at least one column, so the loop terminates. At
termination
$\operatorname{val}(\mathrm{SP}^{\mathsf X}_d)\ge\alpha^{\mathsf X}_d$ for
all $d$; by \cref{cor:lp} the restricted LP is optimal for the full LP.
\end{minipage} \end{array} \]\end{proof}

\begin{remark}\label{rem:cg-integer}
Root-node exact pricing certifies the LP relaxation, not integer optimality:
a column with nonnegative reduced cost at the root LP optimum may still be
required in an optimal integer solution. Global integer exactness requires
branch-and-price, deferred to \cref{bp-sec:main}
\end{remark}

\subsection{Integration with the lexicographic minimax objective}
\label{sec:cg-lex-obj}

Everything developed so far prices and solves the MLU problem, i.e.\ the
\emph{first} component $\lambda^{\downarrow}_1$ of the lexicographic minimax
objective~\eqref{ctn:objlex}: the restricted master
$\mathrm{RMP}^{\mathsf X}_1$ carries a single MLU block
$(\mathcal M^{\mathsf X})$, and the oracles of
\cref{sec:cg-topk}--\cref{sec:cg-exact} price against the duals of that
block. The question is now how to recover the full lexicographic order
$\lambda^{\downarrow}_1,\lambda^{\downarrow}_2,\dots$

The difficulty is that the lexicographic minimax objective is not a single linear
criterion. Beyond the first rank, one must fix the optimal value
$\bar\lambda_{1}$ already achieved, then minimise the second largest load among
the schemes attaining it, and so on, and, as noted in \cref{ss:obj}, the
sorting permutation is \emph{endogenous}, so the identity of the pair
carrying rank $k$ is itself a decision. The route taken here is to keep the
trajectory master intact and change only the rows it is priced against,
exploiting the sequential lexicographic minimax decomposition of \cref{sec:milp}.

\paragraph{Sequential lexicographic principle.}

The lexicographic minimax objective is solved by minimising the sorted components one at a
time. At rank $k$, the problem is
\begin{equation}\label{eq:cg-lex-principle}
\bar\lambda_k
=
\min\Bigl\{
\lambda_k^{\downarrow}(P):
P\in\mathcal F,\;
\lambda_j^{\downarrow}(P)=\bar\lambda_j,\ 
j=1,\dots,k-1
\Bigr\}.
\end{equation}
where $(\bar\lambda_1,\dots,\bar\lambda_{k-1})$ is the \emph{certified prefix}
produced by the previous ranks. At $k=1$ there is no prefix and the problem
reduces to the MLU. Each formulation
$\mathsf X\in\{\textsc{Stela},\textsc{Carla}\}$ turns the equality
constraints of~\eqref{eq:cg-lex-principle} into a tractable block of primal
rows; the two blocks are derived in \cref{sec:milp} for the compact model
and are now embedded into the trajectory master. Since the certified prefix
is the same object in both formulations (\cref{prop:same-sequence}), only
the encoding of the fixation differs.

\begin{proposition}
\label{prop:cg-monotone}
A trajectory generated at rank $k$ remains valid and reusable at every rank
$k'>k$. Consequently a single pool $\Pi_d^{\mathrm{pool}}$ per demand may
grow monotonically across ranks without ever discarding a column, and
rank~$k'$ starts from the full accumulated pool.
\end{proposition}

\begin{proof} \[ \begin{array}{:c} \begin{minipage}{0.97\textwidth}
The coefficients $\Gamma_{d,a,t}(\pi)$ and $\beta_{d,t}(\pi)$ of
\eqref{eq:cg-Gamma}--\eqref{eq:cg-beta} are intrinsic to $\pi$ and
independent of the rank: they depend only on topology, precomputed split
coefficients, volumes and the chosen paths. The rank~$k$ affects only the
objective and the threshold/cumulative rows, which constrain the auxiliary
variables $(z,y)$ or $(z,u)$, not the admissible columns. Adding a column
only enlarges a restricted master's feasible region, so reuse never violates
feasibility; validity follows from \cref{prop:cg-valid}.
\end{minipage} \end{array} \]\end{proof}
\subsubsection{Dual formulations of rank $k$ master problem}

\paragraph{Dual of $\mathrm{RMP}^{\textsc{Stela}}_k$ and rank-$k$ pricing.}
\label{sec:cg-lex-dual-stela}

Two points must be settled before the block of~\eqref{eq:cg-Pk} can be
dualised. First, the global row $\lambda(a,t)\le\bar\lambda_1$ and the
inherited family exist only for $k\ge2$: their constants are \emph{produced}
by the earlier ranks, and $\bar\lambda_1$ also serves as the big-$M$ of the
current row. At $k=1$ the block reduces to the MLU encoding
$(\mathcal M^{\textsc{Stela}})$ of~\eqref{eq:cg-block-stela}. Second, the
exceedance variables must be \emph{relaxed} to $y^{(j)}_{a,t}\in[0,1]$: the LP
relaxation is what carries the duals, and the integrality of $y$ is recovered
only in the MILP closing the rank. With these conventions, for $k\ge2$,

\begin{equation}
\label{eq:cg-Pk}
\mathrm{RMP}^{\textsc{Stela}}_k:\quad
\begin{array}{llll}
\min & z &&\\[4pt]
\text{s.t.} & (\mathrm S),\ (\mathrm L),\ (\mathrm B)
  && \text{[core, unchanged]}\\[4pt]
& \lambda(a,t)\le\bar\lambda_1,
  & \forall(a,t)\in\calA_T, & [\nu^{\textsc{Stela}}_1(a,t)\le0]\quad(\mathrm G)\\[3pt]
& \lambda(a,t)-(\bar\lambda_1-\bar\lambda_j)\,y^{(j)}_{a,t}\le\bar\lambda_j,
  & \forall(a,t),\ 2\le j\le k-1, & [\nu^{\textsc{Stela}}_j(a,t)\le0]\quad(\mathrm F_j)\\[3pt]
& \sum_{(a,t)}y^{(j)}_{a,t}\le j-1,
  & 2\le j\le k-1, & [\rho^{\textsc{Stela}}_j\le0]\quad(\mathrm K_j)\\[3pt]
& \lambda(a,t)-z-\bar\lambda_1\,y^{(k)}_{a,t}\le0,
  & \forall(a,t)\in\calA_T, & [\mu^{\textsc{Stela}}_k(a,t)\le0]\quad(\mathrm C_k)\\[3pt]
& \sum_{(a,t)}y^{(k)}_{a,t}\le k-1,
  & & [\rho^{\textsc{Stela}}_k\le0]\quad(\mathrm K_k)\\[3pt]
& z\le\bar\lambda_{k-1},
  & & [\sigma^{\textsc{Stela}}\le0]\quad(\mathrm Z)\\[3pt]
& y^{(j)}_{a,t}\le1,
  & \forall(a,t),\ 2\le j\le k, & [\omega^{\textsc{Stela}}_j(a,t)\le0]\quad(\mathrm Y_j)\\[3pt]
& z\ge0,\quad y^{(j)}_{a,t}\ge0,\quad \xi_{d,\pi}\ge0,\quad
  \lambda(a,t)\ \text{free.} &&
\end{array}
\end{equation}

\begin{lemma}
\label{lem:feasible-k}
Suppose the pool is grown monotonically across ranks
(\cref{prop:cg-monotone}) and let $P^\star_{k-1}$ be the integer incumbent
returned at rank $k-1$, with
$\lambda^{\downarrow}_j(P^\star_{k-1})=\bar\lambda_j$ for all $j\le k-1$. Then
$\mathrm{RMP}^{\textsc{Stela}}_k$ over the current pool is feasible, and so is
$\mathrm{RMP}^{\textsc{Carla}}_k$. Both being bounded below by $0$, strong
duality applies at every rank.
\end{lemma}

\begin{proof}
\[
\begin{array}{:c}
\begin{minipage}{0.97\textwidth}
The trajectories realising $P^\star_{k-1}$ belong to the pool, columns being
rank-independent and never discarded (\cref{prop:cg-monotone}). Take $\xi$ to
be their indicator, $\lambda$ the induced loads,
$y^{(j)}_{a,t}=\mathbf 1[\lambda(a,t)>\bar\lambda_j]$ and
$z=\lambda^{\downarrow}_k(P^\star_{k-1})$. Since
$\lambda^{\downarrow}_j=\bar\lambda_j$, at most $j-1$ loads exceed
$\bar\lambda_j$ by \Cref{prop:stela-threshold}, so every cardinality row holds;
the threshold rows hold by construction; the budget rows hold because
$P^\star_{k-1}$ is feasible; and
$z=\lambda^{\downarrow}_k\le\lambda^{\downarrow}_{k-1}=\bar\lambda_{k-1}$
gives~$(\mathrm Z)$. The same construction with
$u^{(j)}_{a,t}=(\lambda(a,t)-\bar\lambda_j)_+$ works for \textsc{Carla}.
\end{minipage}
\end{array}
\]
\end{proof}

\paragraph{Identification of the primal variables.}
As at rank~$1$, the dual has one relation per primal variable, determined by
its domain, its objective coefficient, and the rows in which it occurs. The
rank block adds three families of variables to the core, and it is their
relations that change from rank~$1$.

\begin{table}[ht]
\centering
\small
\renewcommand{\arraystretch}{1.15}
\begin{tabular}{@{}l c c >{\raggedright\arraybackslash}p{0.50\linewidth}@{}}
\toprule
Variable & Domain & Obj. & Rows and coefficients \\
\midrule
$\xi_{d,\pi}$
  & $\ge 0$ & $0$
  & $(\mathrm S)_d$: $+1$;\ \ $(\mathrm L)_{a,t}$: $-\Gamma_{d,a,t}(\pi)$;\ \
    $(\mathrm B)_t$: $+\beta_{d,t}(\pi)$;\ \ \emph{no row of the rank block}
    \\[2pt]
$\lambda(a,t)$
  & free & $0$
  & $(\mathrm L)_{a,t}$: $+1$;\ \ $(\mathrm G)_{a,t}$: $+1$;\ \
    $(\mathrm F_j)_{a,t}$: $+1$ for $2\le j\le k-1$;\ \
    $(\mathrm C_k)_{a,t}$: $+1$ \\[2pt]
$z$
  & $\ge 0$ & $1$
  & $(\mathrm C_k)_{a,t}$: $-1$, for every $(a,t)$;\ \ $(\mathrm Z)$: $+1$
    \\[2pt]
$y^{(k)}_{a,t}$
  & $[0,1]$ & $0$
  & $(\mathrm C_k)_{a,t}$: $-\bar\lambda_1$;\ \ $(\mathrm K_k)$: $+1$;\ \
    $(\mathrm Y_k)$: $+1$ \\[2pt]
$y^{(j)}_{a,t}$, $j<k$
  & $[0,1]$ & $0$
  & $(\mathrm F_j)_{a,t}$: $-(\bar\lambda_1-\bar\lambda_j)$;\ \
    $(\mathrm K_j)$: $+1$;\ \ $(\mathrm Y_j)$: $+1$ \\
\bottomrule
\end{tabular}
\caption{\small Incidence of the primal variables of
$\mathrm{RMP}^{\textsc{Stela}}_k$~\eqref{eq:cg-Pk}. Each variable occurs
in no row other than those shown. The decisive line is the first: the rank
block reaches $\xi_{d,\pi}$ only through the free variables $\lambda(a,t)$.}
\label{tab:cg-incidence-stela-k}
\end{table}

\paragraph{Dual derivation.}
We treat each variable in turn. Two of the five relations differ from their
rank-1 counterparts, and it is worth saying which and why.

\emph{Variable $\lambda(a,t)$.} Free, objective coefficient $0$, occurring
with coefficient $+1$ in the load-assembly row of its own pair and in
\emph{every} load-bounding row of its own pair: one global, $k-2$ inherited,
one current. Its dual relation is therefore an equality,
\begin{equation}\label{eq:cg-stela-stationarity-k}
\eta^{\textsc{Stela}}_{a,t}
+\mu^{\textsc{Stela}}_k(a,t)
+\sum_{j<k}\nu^{\textsc{Stela}}_j(a,t)=0 .
\end{equation}
This is the only structural change from rank~$1$, where the sum was empty.

\emph{Variable $z$.} Objective coefficient $1$, occurring with coefficient
$-1$ in the $|\calA_T|$ current rows and with $+1$ in $(\mathrm Z)$. It is
\emph{sign-constrained}, $z\ge0$, so its dual relation is an
\textbf{inequality}, not the equality one would get from a free variable:
\[
\sum_{(a,t)\in\calA_T}\mu^{\textsc{Stela}}_k(a,t)\cdot(-1)
+\sigma^{\textsc{Stela}}\cdot(+1)\ \le\ 1,
\qquad\text{i.e.}\qquad
\sum_{(a,t)\in\calA_T}\bigl(-\mu^{\textsc{Stela}}_k(a,t)\bigr)
\ \le\ 1-\sigma^{\textsc{Stela}} .
\]
Complementary slackness on the sign constraint gives equality when
$z^\star>0$, and complementary slackness on $(\mathrm Z)$ gives
$\sigma^{\textsc{Stela}}=0$ when $z^\star<\bar\lambda_{k-1}$. In the generic case
$0<z^\star<\bar\lambda_{k-1}$ one gets exactly
$\sum_{(a,t)}(-\mu^{\textsc{Stela}}_k(a,t))=1$.

\emph{Variable $y^{(k)}_{a,t}$.} Objective coefficient $0$, nonnegative,
occurring in $(\mathrm C_k)_{a,t}$ with coefficient $-\bar\lambda_1$, in
$(\mathrm K_k)$ with $+1$ and in its own upper bound with $+1$:
\[
-\bar\lambda_1\,\mu^{\textsc{Stela}}_k(a,t)+\rho^{\textsc{Stela}}_k
+\omega^{\textsc{Stela}}_k(a,t)\le0
\quad\Longleftrightarrow\quad
\bar\lambda_1\bigl(-\mu^{\textsc{Stela}}_k(a,t)\bigr)
\le-\rho^{\textsc{Stela}}_k-\omega^{\textsc{Stela}}_k(a,t).
\]
This relation is worth reading: the load pressure at a pair is bounded by the
price of the exceedance budget divided by the big-$M$. A loose $\bar\lambda_1$
therefore \emph{flattens} the load duals, which is the dual manifestation of
the weakness of a big-$M$ relaxation.

\emph{Variable $y^{(j)}_{a,t}$, $j<k$.} The same computation with the
coefficient $-(\bar\lambda_1-\bar\lambda_j)$ in the inherited row.

\emph{Variable $\xi_{d,\pi}$.} Nonnegative, objective coefficient $0$, and it
occurs in \textbf{no row of the rank block}: that block constrains only
$\lambda$, $z$ and the $y$'s. Its dual relation is therefore unchanged from
rank~$1$,
\[
\alpha^{\textsc{Stela}}_d
-\sum_{(a,t)\in\calA_T}\eta^{\textsc{Stela}}_{a,t}\,\Gamma_{d,a,t}(\pi)
+\sum_{t\in T^*}\psi^{\textsc{Stela}}_t\,\beta_{d,t}(\pi)\ \le\ 0,
\qquad \forall d\in D,\ \pi\in\Pi_d^{\mathrm{pool}},
\]
which is exactly why the cardinality duals $\rho^{\textsc{Stela}}_j$ and the box
duals $\sigma^{\textsc{Stela}},\omega^{\textsc{Stela}}_j$ never enter the reduced
cost.

\paragraph{Dual objective at rank $k$.}
\[
\begin{aligned}
\max\quad
& \sum_{d\in D}\alpha^{\textsc{Stela}}_d
 + \sum_{t\in T^*}\kappa(t)\,\psi^{\textsc{Stela}}_t
 + \bar\lambda_1\!\sum_{(a,t)}
   \nu^{\textsc{Stela}}_1(a,t)
\\
&\quad
 + \sum_{j=2}^{k-1}\bar\lambda_j
   \sum_{(a,t)}\nu^{\textsc{Stela}}_j(a,t)
 + \sum_{j=2}^{k}(j-1)\,\rho^{\textsc{Stela}}_j
\\
&\quad
 + \bar\lambda_{k-1}\sigma^{\textsc{Stela}}
 + \sum_{j=2}^{k}\sum_{(a,t)}
   \omega^{\textsc{Stela}}_j(a,t).
\end{aligned}
\]
Every term added to the rank-1 objective is $\le0$, consistently with the fact
that the rank-$k$ master has a larger optimal value than the pure MLU master.

\begin{proposition}
\label{prop:cg-stratified-stela}
Define the \emph{stratified dual} of $\mathrm{RMP}^{\textsc{Stela}}_k$
\begin{equation}\label{eq:cg-stratified-stela}
\bar\mu^{\textsc{Stela}}_k(a,t)
\;:=\;\mu^{\textsc{Stela}}_k(a,t)+\sum_{j<k}\nu^{\textsc{Stela}}_j(a,t)\;\le\;0 .
\end{equation}
Then $\eta^{\textsc{Stela}}_{a,t}=-\bar\mu^{\textsc{Stela}}_k(a,t)\ge0$
by~\eqref{eq:cg-stela-stationarity-k}, and the total dual mass satisfies
\begin{equation}
\label{eq:mass-stela-k}
\sum_{(a,t)\in\calA_T}\eta^{\textsc{Stela}}_{a,t}
=\underbrace{\sum_{(a,t)}\bigl(-\mu^{\textsc{Stela}}_k(a,t)\bigr)}
   _{\le\ 1-\sigma^{\textsc{Stela}}}
+\sum_{j<k}\sum_{(a,t)}\bigl(-\nu^{\textsc{Stela}}_j(a,t)\bigr).
\end{equation}
At an optimal primal--dual pair with $0<z^\star<\bar\lambda_{k-1}$ one has
$\sigma^{\textsc{Stela}}=0$ and $\sum_{(a,t)}(-\mu^{\textsc{Stela}}_k(a,t))=1$, hence
\begin{equation}
\label{eq:mass-stela-k-exact}
\sum_{(a,t)\in\calA_T}\eta^{\textsc{Stela}}_{a,t}
\;=\;1+\sum_{j<k}\bigl\lVert\nu^{\textsc{Stela}}_j\bigr\rVert_1\ \ \ge\ 1,
\end{equation}
with equality to $1$ if and only if every inherited dual vanishes. In
particular $\eta^{\textsc{Stela}}\notin\Delta_{\calA_T}$ in general as soon as
$k\ge2$: the simplex property established at rank~$1$
(\Cref{prop:cg-simplex-stela}) is a rank-1 property.
\end{proposition}

\begin{proof}
\[
\begin{array}{:c}
\begin{minipage}{0.97\textwidth}
Nonpositivity of $\bar\mu^{\textsc{Stela}}_k$ is immediate: it is a finite sum of
duals of ``$\le$'' rows of a minimisation, each $\le0$. The identity
$\eta^{\textsc{Stela}}_{a,t}=-\bar\mu^{\textsc{Stela}}_k(a,t)$
is~\eqref{eq:cg-stela-stationarity-k} rearranged. Summing it over
$(a,t)\in\calA_T$ and inserting the dual relation of $z$ derived above
gives~\eqref{eq:mass-stela-k}, and~\eqref{eq:mass-stela-k-exact} in the generic
case, using
$\lVert\nu^{\textsc{Stela}}_j\rVert_1
=\sum_{(a,t)}\bigl(-\nu^{\textsc{Stela}}_j(a,t)\bigr)$ since
$\nu^{\textsc{Stela}}_j\le0$.

Note that the sign of $\bar\mu^{\textsc{Stela}}_k$ is not an assumption: it is
forced by the sign of every summand, itself forced by the direction of the
corresponding inequality. Had any load-bounding row been written as ``$\ge$'',
the sign would flip and $\theta^{\textsc{Stela}}_{d,k}$ could become negative,
destroying the shortest-path structure of the pricing.
\end{minipage}
\end{array}
\]
\end{proof}

\begin{proposition}
\label{prop:cg-rc-stela-k}
For every $\pi\in\Pi_d$, with
$\gamma^{\textsc{Stela}}_t:=-\psi^{\textsc{Stela}}_t\ge0$,
\begin{equation}\label{eq:cg-rc-stela-k}
\rc^{\textsc{Stela}}_{d,k}(\pi)
\;=\;-\alpha^{\textsc{Stela}}_d
+\sum_{(a,t)\in\calA_T}
  \bigl(-\bar\mu^{\textsc{Stela}}_k(a,t)\bigr)\,\Gamma_{d,a,t}(\pi)
+\sum_{t\in T^*}\gamma^{\textsc{Stela}}_t\,\beta_{d,t}(\pi).
\end{equation}
The effective segment cost at rank~$k$ is
\begin{equation}\label{eq:cg-theta-stela-k}
\theta^{\textsc{Stela}}_{d,k}(i,j,t)
\;:=\;\sum_{a\,:\,(a,t)\in\calA_T}
  \bigl(-\bar\mu^{\textsc{Stela}}_k(a,t)\bigr)
  \,\frac{\nu(d,t)}{c(a)}\,r(i,j,a,t)
\;\ge\;0,
\end{equation}
and the rank-$k$ pricing subproblem is
\begin{equation}\label{eq:cg-SP-stela-k}
(\mathrm{SP}^{\textsc{Stela}}_{d,k})\qquad
\operatorname{val}(\mathrm{SP}^{\textsc{Stela}}_{d,k})
=\min_{\pi\in\Pi_d}
\Bigl\{\sum_{t\in T}\theta^{\textsc{Stela}}_{d,k}\bigl(p^{(t)},t\bigr)
+\sum_{t\in T^*}\gamma^{\textsc{Stela}}_t\,
  \dist\bigl(p^{(t-1)},p^{(t)}\bigr)\Bigr\}.
\end{equation}
At $k=1$ the sum in~\eqref{eq:cg-stratified-stela} is empty,
$\bar\mu^{\textsc{Stela}}_1=\mu^{\textsc{Stela}}_1$, and
\eqref{eq:cg-rc-stela-k}--\eqref{eq:cg-SP-stela-k} reduce to the rank-1
reduced cost~\eqref{eq:cg-rc-stela} and pricing
problem~\eqref{eq:cg-SP-generic-path}.
\end{proposition}

\begin{proof}
\[
\begin{array}{:c}
\begin{minipage}{0.97\textwidth}
The reduced cost of a column is its direct objective cost minus its dual
activity. The direct cost is $0$, the objective $\min z$ containing no
selection variable. The activity involves only the duals of the rows
containing $\xi_{d,\pi}$, namely $\alpha^{\textsc{Stela}}_d$,
$\eta^{\textsc{Stela}}$ and $\psi^{\textsc{Stela}}$: the duals
$\mu^{\textsc{Stela}}_k,\nu^{\textsc{Stela}}_j,\rho^{\textsc{Stela}}_j,
\sigma^{\textsc{Stela}},\omega^{\textsc{Stela}}_j$ multiply coefficients of
$\xi_{d,\pi}$ that are \emph{zero}. Hence
$$
\rc^{\textsc{Stela}}_{d,k}(\pi)
=0-\Bigl[\alpha^{\textsc{Stela}}_d
  -\sum_{(a,t)}\eta^{\textsc{Stela}}_{a,t}\,\Gamma_{d,a,t}(\pi)
  +\sum_{t\in T^*}\psi^{\textsc{Stela}}_t\,\beta_{d,t}(\pi)\Bigr],
$$
and substituting $\eta^{\textsc{Stela}}=-\bar\mu^{\textsc{Stela}}_k$
(\Cref{prop:cg-stratified-stela}) and
$\gamma^{\textsc{Stela}}_t=-\psi^{\textsc{Stela}}_t$
gives~\eqref{eq:cg-rc-stela-k}. The rank-$k$ duals thus enter \emph{only}
through this substitution; the functional form is unchanged.

For the pricing, insert $\Gamma_{d,a,t}(\pi)$ from~\eqref{eq:cg-Gamma} and
exchange the order of summation, legitimate because all sums are finite:
$$
\begin{aligned}
\sum_{(a,t)\in\calA_T}\eta^{\textsc{Stela}}_{a,t}\Gamma_{d,a,t}(\pi)
&=\sum_{t\in T}\ \sum_{a:(a,t)\in\calA_T}\eta^{\textsc{Stela}}_{a,t}\,
  \frac{\nu(d,t)}{c(a)}\sum_{i,j\in V}r(i,j,a,t)\,\delta\bigl(p^{(t)},ij\bigr)\\
&=\sum_{t\in T}\ \sum_{i,j\in V}\delta\bigl(p^{(t)},ij\bigr)\,
  \theta^{\textsc{Stela}}_{d,k}(i,j,t)\\
&=\sum_{t\in T}\sum_{(i,j)\in\segs(p^{(t)})}\theta^{\textsc{Stela}}_{d,k}(i,j,t),
\end{aligned}
$$
the last equality using $\delta(p,ij)\in\{0,1\}$ with
$\delta(p,ij)=1\Leftrightarrow(i,j)\in\segs(p)$. Nonnegativity of
$\theta^{\textsc{Stela}}_{d,k}$ follows from
$-\bar\mu^{\textsc{Stela}}_k\ge0$, $\nu(d,t)\ge0$, $c(a)>0$ and
$r(i,j,a,t)\ge0$. Separating the constant $-\alpha^{\textsc{Stela}}_d$
gives~\eqref{eq:cg-SP-stela-k}.

Only $\eta^{\textsc{Stela}}\ge0$ and $\gamma^{\textsc{Stela}}\ge0$ were used; the
total mass of $\eta^{\textsc{Stela}}$ played no role. This is what licenses
running the rank-1 oracles unchanged at rank~$k$, despite
\Cref{prop:cg-stratified-stela}.
\end{minipage}
\end{array}
\]
\end{proof}

\subsubsection{Dual of $\mathrm{RMP}^{\textsc{Carla}}_k$ and rank-$k$ pricing.}
\label{eq:cg-carla-Pk}

Recall from \cref{sec:milp} that \textsc{Carla} solves rank~$k$ by
minimising the $k$-th cumulative ordered load
$\Theta_k=\sum_{r\le k}\lambda^{\downarrow}_r$ through the Ogryczak--Tamir
representation, using continuous auxiliaries $u^{(j)}_{a,t}\ge0$ and
freezing each inherited rank $j<k$ at its certified cumulative value.
Setting
\begin{equation}
\label{eq:cg-carla-cj}
\bar\Theta_j=\sum_{r\le j}\bar\lambda_{r},
\qquad
c_j=\bar\Theta_j-j\,\bar\lambda_{j}
=\sum_{r<j}\bigl(\bar\lambda_{r}-\bar\lambda_{j}\bigr)\ \ge0,
\quad 1\le j<k,
\end{equation}
\begin{equation}
\label{eq:rmp-carla-k}
\mathrm{RMP}^{\textsc{Carla}}_k:\quad
\begin{array}{llll}
\min & k\,z_k+\sum_{(a,t)\in\calA_T}u^{(k)}_{a,t} &&\\[4pt]
\text{s.t.} & (\mathrm S),\ (\mathrm L),\ (\mathrm B)
  && \text{[core, unchanged]}\\[4pt]
& \lambda(a,t)-z_k-u^{(k)}_{a,t}\le0,
  & \forall(a,t)\in\calA_T, & [\mu^{\textsc{Carla}}_k(a,t)\le0]\quad(\mathrm C_k)\\[3pt]
& \lambda(a,t)-u^{(j)}_{a,t}\le\bar\lambda_j,
  & \forall(a,t),\ j<k, & [\nu^{\textsc{Carla}}_j(a,t)\le0]\quad(\mathrm F_j)\\[3pt]
& \sum_{(a,t)}u^{(j)}_{a,t}\le c_j,
  & j<k, & [\rho^{\textsc{Carla}}_j\le0]\quad(\mathrm D_j)\\[3pt]
& z_k\le\bar\lambda_{k-1},
  & & [\sigma^{\textsc{Carla}}\le0]\quad(\mathrm Z)\\[3pt]
& z_k\ge0,\ u^{(k)}_{a,t}\ge0,\ u^{(j)}_{a,t}\ge0,\
  \xi_{d,\pi}\ge0,\ \lambda(a,t)\ \text{free,} &&
\end{array}
\end{equation}
with $c_j=\bar\Theta_j-j\bar\lambda_j
=\sum_{r<j}(\bar\lambda_r-\bar\lambda_j)\ge0$ as in~\eqref{eq:cg-carla-cj}.

\paragraph{Identification of the primal variables.}

\begin{table}[ht]
\centering
\small
\renewcommand{\arraystretch}{1.15}
\begin{tabular}{@{}l c c >{\raggedright\arraybackslash}p{0.50\linewidth}@{}}
\toprule
Variable & Domain & Obj. & Rows and coefficients \\
\midrule
$\xi_{d,\pi}$
  & $\ge 0$ & $0$
  & $(\mathrm S)_d$: $+1$;\ \ $(\mathrm L)_{a,t}$: $-\Gamma_{d,a,t}(\pi)$;\ \
    $(\mathrm B)_t$: $+\beta_{d,t}(\pi)$;\ \ \emph{no row of the rank block}
    \\[2pt]
$\lambda(a,t)$
  & free & $0$
  & $(\mathrm L)_{a,t}$: $+1$;\ \ $(\mathrm C_k)_{a,t}$: $+1$;\ \
    $(\mathrm F_j)_{a,t}$: $+1$ for $j<k$ \\[2pt]
$z_k$
  & $\ge 0$ & $k$
  & $(\mathrm C_k)_{a,t}$: $-1$, for every $(a,t)$;\ \ $(\mathrm Z)$: $+1$
    \\[2pt]
$u^{(k)}_{a,t}$
  & $\ge 0$ & $1$
  & $(\mathrm C_k)_{a,t}$: $-1$ \\[2pt]
$u^{(j)}_{a,t}$, $j<k$
  & $\ge 0$ & $0$
  & $(\mathrm F_j)_{a,t}$: $-1$;\ \ $(\mathrm D_j)$: $+1$ \\
\bottomrule
\end{tabular}
\caption{\small Incidence of the primal variables of
$\mathrm{RMP}^{\textsc{Carla}}_k$~\eqref{eq:rmp-carla-k}. Compare with
\Cref{tab:cg-incidence-stela-k}: the objective coefficient of the threshold is
$k$, not $1$.}
\label{tab:cg-incidence-carla-k}
\end{table}

Three lines of \Cref{tab:cg-incidence-carla-k} are where \textsc{Carla}
genuinely departs from \textsc{Stela}, and where transposing the
\textsc{Stela} derivation would be wrong.

\begin{enumerate}[label=(\arabic*),itemsep=3pt]
\item \emph{The objective coefficient of the threshold is $k$, not $1$}, the
      rank-$k$ objective being $k z_k+\sum u^{(k)}$. This single number changes
      the normalisation from $\sum(-\mu_k)\le1$ to $\sum(-\mu_k)\le k$.
\item \emph{The current excesses $u^{(k)}$ carry objective coefficient $1$},
      producing the box $-\mu^{\textsc{Carla}}_k(a,t)\le1$ which, combined
      with~(1), makes $-\mu^{\textsc{Carla}}_k$ a top-$k$ selector
      (\Cref{prop:topk-selector}). \textsc{Stela} has no analogue.
\item \emph{The inherited excesses $u^{(j)}$ carry objective coefficient $0$}
      and are bounded only through the cumulative cardinality row, producing
      $0\le-\nu^{\textsc{Carla}}_j(a,t)\le-\rho^{\textsc{Carla}}_j$. The
      \textsc{Stela} analogue involves the big-$M$
      $(\bar\lambda_1-\bar\lambda_j)$ and is scaled by it.
\end{enumerate}

\paragraph{Dual derivation.}

\emph{Variable $\lambda(a,t)$.} Free, objective coefficient $0$, occurring
with coefficient $+1$ in $(\mathrm L)_{a,t}$, in $(\mathrm C_k)_{a,t}$ and in
each $(\mathrm F_j)_{a,t}$. Equality relation:
\begin{equation}\label{eq:cg-carla-stationarity-k}
\eta^{\textsc{Carla}}_{a,t}
+\mu^{\textsc{Carla}}_k(a,t)
+\sum_{j<k}\nu^{\textsc{Carla}}_j(a,t)=0 .
\end{equation}

\emph{Variable $z_k$.} \textbf{Objective coefficient $k$}; sign-constrained
($z_k\ge0$); coefficient $-1$ in every $(\mathrm C_k)_{a,t}$ and $+1$ in
$(\mathrm Z)$. Hence the inequality
\[
\sum_{(a,t)\in\calA_T}\bigl(-\mu^{\textsc{Carla}}_k(a,t)\bigr)
+\sigma^{\textsc{Carla}}\ \le\ k,
\]
with equality when $z_k^\star>0$, and $\sigma^{\textsc{Carla}}=0$ when
$z_k^\star<\bar\lambda_{k-1}$.

\emph{Variable $u^{(k)}_{a,t}$.} Nonnegative, objective coefficient $1$,
occurring only in $(\mathrm C_k)_{a,t}$ with coefficient $-1$:
$-\mu^{\textsc{Carla}}_k(a,t)\le1$.

\emph{Variable $u^{(j)}_{a,t}$, $j<k$.} Nonnegative, objective coefficient
$0$, occurring in $(\mathrm F_j)_{a,t}$ with coefficient $-1$ and in
$(\mathrm D_j)$ with coefficient $+1$:
$-\nu^{\textsc{Carla}}_j(a,t)+\rho^{\textsc{Carla}}_j\le0$, i.e.\
$-\nu^{\textsc{Carla}}_j(a,t)\le-\rho^{\textsc{Carla}}_j$.

\emph{Variable $\xi_{d,\pi}$.} It occurs in no row of the rank block: not in
the current rows (which contain $\lambda,z_k,u^{(k)}$), not in the inherited
rows (which contain $\lambda,u^{(j)}$), not in the cumulative cardinality rows
(which contain $u^{(j)}$ only), not in the boxes. Its dual relation is
therefore that of rank~$1$.

\paragraph{Dual objective at rank $k$.}
\[
\max\ \sum_{d\in D}\alpha^{\textsc{Carla}}_d
+\sum_{t\in T^*}\kappa(t)\,\psi^{\textsc{Carla}}_t
+\sum_{j<k}\bar\lambda_j\!\!\sum_{(a,t)}\!\!\nu^{\textsc{Carla}}_j(a,t)
+\sum_{j<k}c_j\,\rho^{\textsc{Carla}}_j
+\bar\lambda_{k-1}\,\sigma^{\textsc{Carla}} .
\]

\begin{proposition}[stratified dual and normalisation, \textsc{Carla}]
\label{prop:cg-stratified-carla}
Define the \emph{stratified dual}
\begin{equation}\label{eq:cg-stratified-carla}
\bar\mu^{\textsc{Carla}}_k(a,t)
\;:=\;\mu^{\textsc{Carla}}_k(a,t)+\sum_{j<k}\nu^{\textsc{Carla}}_j(a,t)\;\le\;0 .
\end{equation}
Then $\eta^{\textsc{Carla}}_{a,t}=-\bar\mu^{\textsc{Carla}}_k(a,t)\ge0$
by~\eqref{eq:cg-carla-stationarity-k}, and
\begin{equation}
\label{eq:mass-carla-k}
\sum_{(a,t)}\bigl(-\mu^{\textsc{Carla}}_k(a,t)\bigr)\ \le\ k-\sigma^{\textsc{Carla}},
\qquad
-\mu^{\textsc{Carla}}_k(a,t)\ \le\ 1\quad\forall(a,t),
\qquad
0\le-\nu^{\textsc{Carla}}_j(a,t)\le-\rho^{\textsc{Carla}}_j .
\end{equation}
At an optimal pair with $0<z_k^\star<\bar\lambda_{k-1}$,
\begin{equation}
\label{eq:mass-carla-k-exact}
\sum_{(a,t)\in\calA_T}\eta^{\textsc{Carla}}_{a,t}
\;=\;k+\sum_{j<k}\bigl\lVert\nu^{\textsc{Carla}}_j\bigr\rVert_1\ \ \ge\ k .
\end{equation}
\end{proposition}

\begin{proof}
\[
\begin{array}{:c}
\begin{minipage}{0.97\textwidth}
The bookkeeping is that of \Cref{prop:cg-stratified-stela} applied to the rows
of~\eqref{eq:rmp-carla-k}: $\bar\mu^{\textsc{Carla}}_k\le0$ as a sum of duals of
``$\le$'' rows of a minimisation, and
$\eta^{\textsc{Carla}}_{a,t}=-\bar\mu^{\textsc{Carla}}_k(a,t)$
is~\eqref{eq:cg-carla-stationarity-k} rearranged. The three inequalities
of~\eqref{eq:mass-carla-k} are the dual relations of $z_k$, of
$u^{(k)}_{a,t}$ and of $u^{(j)}_{a,t}$ derived above. Summing
$\eta^{\textsc{Carla}}_{a,t}=-\mu^{\textsc{Carla}}_k(a,t)-\sum_{j<k}
\nu^{\textsc{Carla}}_j(a,t)$ over $\calA_T$ and using $\sigma^{\textsc{Carla}}=0$ with
equality in the relation of $z_k$ gives~\eqref{eq:mass-carla-k-exact}.
\end{minipage}
\end{array}
\]
\end{proof}

\begin{proposition}[$-\mu^{\textsc{Carla}}_k$ is a top-$k$ selector]
\label{prop:topk-selector}
At a nondegenerate optimum of the \textsc{Carla} rank-$k$ block (no inherited
row active and $0<z_k^\star<\bar\lambda_{k-1}$), the vector
$w:=-\mu^{\textsc{Carla}}_k$ satisfies
\[
0\le w_{a,t}\le 1\quad\forall(a,t)\in\calA_T,
\qquad \sum_{(a,t)\in\calA_T}w_{a,t}=k,
\]
which is exactly the feasible set of the linear characterisation of the sum of
the $k$ largest entries,
\[
\Theta_k(\lambda)=\max\Bigl\{\textstyle\sum_{(a,t)}w_{a,t}\,\lambda(a,t)\ :\
0\le w\le1,\ \sum_{(a,t)}w_{a,t}=k\Bigr\} .
\]
The load dual produced by the \textsc{Carla} block is therefore the optimal
top-$k$ weight vector of the current load vector: weight $1$ on each of the
$k$ most loaded pairs and $0$ elsewhere, fractional weights appearing only
under ties.
\end{proposition}

\begin{proof}
\[
\begin{array}{:c}
\begin{minipage}{0.97\textwidth}
The displayed constraints are~\eqref{eq:mass-carla-k} with
$\sigma^{\textsc{Carla}}=0$ and the inherited duals set to zero. The linear
characterisation of $\Theta_k$ is the LP dual, in the variables $(z,u)$ for
fixed $\lambda$, of the Ogryczak--Tamir minimisation of \Cref{lem:ot},
$\Theta_k(\lambda)=\min_{z,u\ge0}\{kz+\sum_{(a,t)}u_{a,t}:
u_{a,t}\ge\lambda(a,t)-z\}$; dualising gives
$\max\{w^{\top}\lambda:0\le w\le1,\ \sum w=k\}$, whose optimal solutions are
the indicator vectors of the $k$ largest entries.
\end{minipage}
\end{array}
\]
\end{proof}

\begin{remark}
\label{rem:carla-ub-rankk}
At rank~$1$, $\eta^{\textsc{Carla}}\ge0$ and $\sum\eta^{\textsc{Carla}}=1$ make the
bound $\eta^{\textsc{Carla}}_{a,t}\le1$ automatic, so the dual relation of
$u^{(1)}_{a,t}$ carries no information beyond the simplex
(\Cref{prop:cg-simplex-carla}). At rank $k\ge2$ the total mass is $k$ and the
bound becomes the binding description of a genuinely different polytope, the
one of \Cref{prop:topk-selector}.
\end{remark}

\begin{proposition}
\label{prop:cg-rc-carla-k}
For every $\pi\in\Pi_d$, with
$\gamma^{\textsc{Carla}}_t:=-\psi^{\textsc{Carla}}_t\ge0$,
\begin{equation}\label{eq:cg-rc-carla-k}
\rc^{\textsc{Carla}}_{d,k}(\pi)
\;=\;-\alpha^{\textsc{Carla}}_d
+\sum_{(a,t)\in\calA_T}
  \bigl(-\bar\mu^{\textsc{Carla}}_k(a,t)\bigr)\,\Gamma_{d,a,t}(\pi)
+\sum_{t\in T^*}\gamma^{\textsc{Carla}}_t\,\beta_{d,t}(\pi),
\end{equation}
the effective segment cost is
\begin{equation}\label{eq:cg-theta-carla-k}
\theta^{\textsc{Carla}}_{d,k}(i,j,t)
\;:=\;\sum_{a\,:\,(a,t)\in\calA_T}
  \bigl(-\bar\mu^{\textsc{Carla}}_k(a,t)\bigr)
  \,\frac{\nu(d,t)}{c(a)}\,r(i,j,a,t)
\;\ge\;0,
\end{equation}
and the rank-$k$ pricing subproblem is
\begin{equation}\label{eq:cg-SP-carla-k}
(\mathrm{SP}^{\textsc{Carla}}_{d,k})\qquad
\operatorname{val}(\mathrm{SP}^{\textsc{Carla}}_{d,k})
=\min_{\pi\in\Pi_d}
\Bigl\{\sum_{t\in T}\theta^{\textsc{Carla}}_{d,k}\bigl(p^{(t)},t\bigr)
+\sum_{t\in T^*}\gamma^{\textsc{Carla}}_t\,
  \dist\bigl(p^{(t-1)},p^{(t)}\bigr)\Bigr\}.
\end{equation}
At $k=1$ the sum in~\eqref{eq:cg-stratified-carla} is empty,
$\bar\mu^{\textsc{Carla}}_1=\mu^{\textsc{Carla}}_1$, and
\eqref{eq:cg-rc-carla-k}--\eqref{eq:cg-SP-carla-k} reduce to the rank-1
reduced cost~\eqref{eq:cg-rc-carla} and pricing
problem~\eqref{eq:cg-SP-generic-path}.
\end{proposition}

\begin{proof}
\[
\begin{array}{:c}
\begin{minipage}{0.97\textwidth}
The direct objective cost of $\xi_{d,\pi}$ is $0$, and $\xi_{d,\pi}$ occurs in
$(\mathrm S),(\mathrm L),(\mathrm B)$ only, as listed in
\Cref{tab:cg-incidence-carla-k}. The duals
$\mu^{\textsc{Carla}}_k,\nu^{\textsc{Carla}}_j,\rho^{\textsc{Carla}}_j,
\sigma^{\textsc{Carla}}$ therefore multiply zero coefficients in its activity and
drop out, leaving
$\alpha^{\textsc{Carla}}_d-\sum_{(a,t)}\eta^{\textsc{Carla}}_{a,t}
\Gamma_{d,a,t}(\pi)+\sum_{t\in T^*}\psi^{\textsc{Carla}}_t\beta_{d,t}(\pi)$.
Subtracting and substituting $\eta^{\textsc{Carla}}=-\bar\mu^{\textsc{Carla}}_k$ and
$\gamma^{\textsc{Carla}}_t=-\psi^{\textsc{Carla}}_t$ gives~\eqref{eq:cg-rc-carla-k}.
The derivation of $\theta^{\textsc{Carla}}_{d,k}$ repeats the exchange of
summations of \Cref{prop:cg-rc-stela-k}, which used no property of
$\eta^{\textsc{Carla}}$ beyond nonnegativity, and in particular not its total
mass, which by \Cref{prop:cg-stratified-carla} is not $1$.
\end{minipage}
\end{array}
\]
\end{proof}

\subsection{Common pricing structure at rank $k$.}
\label{sec:cg-lex}

\begin{theorem}
\label{thm:cg-common}
For $\mathsf X\in\{\textsc{Stela},\textsc{Carla}\}$ and every $k\ge1$, the
rank-$k$ reduced cost of $\pi\in\Pi_d$ in $\mathrm{RMP}^{\mathsf X}_k$ is
\begin{equation}\label{eq:cg-rc-rankk}
\rc^{\mathsf X}_{d,k}(\pi)
\;=\;-\alpha^{\mathsf X}_d
+\sum_{(a,t)\in\calA_T}
  \bigl(-\bar\mu^{\mathsf X}_k(a,t)\bigr)\,\Gamma_{d,a,t}(\pi)
+\sum_{t\in T^*}\gamma^{\mathsf X}_t\,\beta_{d,t}(\pi),
\end{equation}
with $-\bar\mu^{\mathsf X}_k\ge0$ and $\gamma^{\mathsf X}\ge0$, and the
pricing subproblem is
\begin{equation}\label{eq:cg-SP-rankk}
(\mathrm{SP}^{\mathsf X}_{d,k})\qquad
\operatorname{val}(\mathrm{SP}^{\mathsf X}_{d,k})
=\min_{\pi\in\Pi_d}
\Bigl\{\sum_{t\in T}\theta^{\mathsf X}_{d,k}\bigl(p^{(t)},t\bigr)
+\sum_{t\in T^*}\gamma^{\mathsf X}_t\,
  \dist\bigl(p^{(t-1)},p^{(t)}\bigr)\Bigr\},
\end{equation}
a problem whose feasible set $\Pi_d$, bounded-hop structure,
$\texttt{maxSeg}$ restriction and chain coupling do not depend on
$(\mathsf X,k)$, and whose data
$(\theta^{\mathsf X}_{d,k},\gamma^{\mathsf X})\ge0$ do.
\end{theorem}

\begin{proof}
\[
\begin{array}{:c}
\begin{minipage}{0.97\textwidth}
Collect \Cref{prop:cg-rc-stela}, \Cref{prop:cg-rc-carla},
\Cref{prop:cg-rc-stela-k} and \Cref{prop:cg-rc-carla-k}. Each derivation used
exactly three facts, verified separately in each case: (i) $\xi_{d,\pi}$ has
zero objective coefficient; (ii) $\xi_{d,\pi}$ occurs only in
$(\mathrm S),(\mathrm L),(\mathrm B)$, with coefficients $1$,
$-\Gamma_{d,a,t}(\pi)$, $\beta_{d,t}(\pi)$; (iii) the stationarity of the free
variable $\lambda(a,t)$ expresses $\eta^{\mathsf X}_{a,t}$ as the negative of
the sum of the load-bounding duals of its own pair. Facts (i)--(ii) are
structural properties of the \emph{core}, common to all four masters; fact
(iii) is the only place where the rank block enters, and it enters only
through the value of $\eta^{\mathsf X}$.
\end{minipage}
\end{array}
\]
\end{proof}

\begin{corollary}
\label{cor:scales}
Under the hypotheses of \Cref{prop:cg-stratified-stela} and
\Cref{prop:cg-stratified-carla},
\[
\bigl\lVert\eta^{\textsc{Stela}}\bigr\rVert_1
 =1+\sum_{j<k}\lVert\nu^{\textsc{Stela}}_j\rVert_1,
\qquad
\bigl\lVert\eta^{\textsc{Carla}}\bigr\rVert_1
 =k+\sum_{j<k}\lVert\nu^{\textsc{Carla}}_j\rVert_1 .
\]

\end{corollary}

\begin{remark}
\label{rem:no-simplex-needed}
Nothing in the pricing is invalidated. The oracles require only
$\theta^{\mathsf X}_{d,k}\ge0$, for the bounded-hop shortest paths, and
$\gamma^{\mathsf X}\ge0$; both survive at rank~$k$
by \Cref{prop:cg-stratified-stela} and \Cref{prop:cg-stratified-carla}.
Likewise the stopping test
$\operatorname{val}(\mathrm{SP}^{\mathsf X}_{d,k})<\alpha^{\mathsf X}_d$ is
\emph{scale-consistent}: both sides come from the same dual solution, so a
common rescaling of the master objective rescales them identically. What must
be avoided is invoking $\eta^{\mathsf X}\in\Delta_{\calA_T}$ as a hypothesis
beyond rank~$1$; the correct standing hypothesis is
$\eta^{\mathsf X}\ge0$.
\end{remark}

Every oracle of \cref{sec:cg-pricing}; \textsc{TCG-TopK},
\textsc{TCG-Diag}, \textsc{TCG-Exact}; therefore runs unchanged with
$\theta^{\mathsf X}_{d,k}$ and $\gamma^{\mathsf X}_t$ as input.

\subsubsection{Stopping criterion at rank $k$.}

\begin{proposition}\label{prop:certk}
Fix formulation $\mathsf X$, rank $k$, and an optimal dual solution of the
LP relaxation of\/ $\mathrm{RMP}^{\mathsf X}_k$. The minimum rank-$k$
reduced cost of demand~$d$ is
\[
\min_{\pi\in\Pi_d}\rc^{\mathsf X}_{d,k}(\pi)
\;=\;-\alpha^{\mathsf X}_d
+\operatorname{val}(\mathrm{SP}^{\mathsf X}_{d,k}).
\]
A negative reduced-cost column exists if and only if
$\operatorname{val}(\mathrm{SP}^{\mathsf X}_{d,k})<\alpha^{\mathsf X}_d$.
\end{proposition}

\begin{proof}
\[
\begin{array}{:c}
\begin{minipage}{0.97\textwidth}
In the reduced cost~\eqref{eq:cg-rc-rankk}, the term
$-\alpha^{\mathsf X}_d$ is independent of $\pi$: $\alpha^{\mathsf X}_d$ is the
dual of the convexity row of demand $d$, in which every column of $d$ has
coefficient $1$. Minimising over $\Pi_d$ therefore amounts to minimising the
pricing objective of~\eqref{eq:cg-SP-rankk}, whose optimum is
$\operatorname{val}(\mathrm{SP}^{\mathsf X}_{d,k})$; the minimum is attained
since $\Pi_d$ is finite, and the criterion follows.
\end{minipage}
\end{array}
\]
\end{proof}

\begin{corollary}\label{cor:lpk}
If\/ $(\mathrm{SP}^{\mathsf X}_{d,k})$ is solved for every $d\in D$ and
$\operatorname{val}(\mathrm{SP}^{\mathsf X}_{d,k})\ge\alpha^{\mathsf X}_d$
for all~$d$, then the LP relaxation of\/ $\mathrm{RMP}^{\mathsf X}_k$ is
optimal for the full rank-$k$ trajectory LP.
\end{corollary}

\begin{proof}
\[
\begin{array}{:c}
\begin{minipage}{0.97\textwidth}
By \cref{prop:certk}, the condition means
$\rc^{\mathsf X}_{d,k}(\pi)\ge0$ for every $\pi\in\Pi_d$ and every~$d$;
the standard reduced-cost optimality condition gives optimality of the
restricted LP for the full LP.
\end{minipage}
\end{array}
\]
\end{proof}

\subsection{Lexicographic column-generation algorithm}
\label{sec:cg-lex-algo}

The complete procedure nests two mechanisms: an \emph{outer} loop over
lexicographic ranks $k=1,\dots,K$ and, for each rank, an \emph{inner}
column-generation loop that solves the LP relaxation of
$\mathrm{RMP}^{\mathsf X}_k$ and then closes the rank-$k$ MILP over the
pool. Moving from rank $k$ to rank $k{+}1$ requires three pieces of
information:
\begin{itemize}\itemsep2pt
\item the certified level $\bar\lambda_{k}$, recovered from the integer
      incumbent's sorted load vector by evaluation, never by differencing
      cumulative values, which amplifies the numerical error when two
      consecutive $\bar\Theta$ are close (cf.\ \cref{sec:milp});
\item the prefix $(\bar\lambda_{1},\dots,\bar\lambda_{k})$, used to build the
      rank-$(k{+}1)$ block $(\mathcal P^{\mathsf X}_{k+1})$;
\item the accumulated pool $\Pi_d^{\mathrm{pool}}$
      (\cref{prop:cg-monotone}), which carries every column generated at
      ranks $1,\dots,k$ into rank~$k{+}1$.
\end{itemize}
Dual solutions, reduced costs and the LP basis are \emph{not} carried across
ranks: each $\mathrm{RMP}^{\mathsf X}_k$ is solved from scratch over the
accumulated pool. This is consistent with \Cref{cor:scales}: the rank-$k$
duals live at a different scale from the rank-$(k{-}1)$ ones, so warm-starting
the dual would be meaningless.

For each $d\in D$, let
$p_d^{\mathrm{IGP}}:=\langle s_d,t_d\rangle$ denote the zero-waypoint
segment path. Under the standing reachability assumption,
$p_d^{\mathrm{IGP}}$ is feasible in every period. Initialize
\[
\Pi_d^{\mathrm{pool}}
\gets
\{(p_d^{\mathrm{IGP}},\dots,p_d^{\mathrm{IGP}})\}.
\]
Its reconfiguration cost is zero because its segment set is unchanged across
periods; by \Cref{lem:feasible-k} the rank-$k$ restricted master is then
feasible at every rank, so strong duality holds and the duals used above
exist.

\begin{algorithm}[H]
\caption{Lexicographic trajectory column generation for
$T$-\srchallenge{} under $\mathrm{RMP}^{\mathsf X}_k$}
\label{alg:cg-lex}
\begin{algorithmic}[1]
\Require formulation $\mathsf X\in\{\textsc{Stela},\textsc{Carla}\}$,
demands $D$, horizon $T=\{0,\dots,h{-}1\}$, number of ranks $K$, tolerance
$\varepsilon>0$, candidate-list size $K_{\mathrm{cand}}$, segment bound
$\texttt{maxSeg}$, pricing oracle $\mathcal O$.
\Ensure routing scheme over the pool and attained prefix
$(\bar\lambda_{1},\dots,\bar\lambda_{K})$.
\State Initialise
$\Pi_d^{\mathrm{pool}}\gets\{(\mathrm{IGP},\dots,\mathrm{IGP})\}$ for all
$d\in D$.
\Comment{nominal shortest-path trajectory}
\For{$k=1,\dots,K$}
  \State Build $\mathrm{RMP}^{\mathsf X}_k$: core $(\mathcal C)$ plus
         rank block~\eqref{eq:cg-Pk}
         if $\mathsf X=\textsc{Stela}$, or~\eqref{eq:rmp-carla-k}
         if $\mathsf X=\textsc{Carla}$.
  \Repeat
    \State Solve the LP relaxation of $\mathrm{RMP}^{\mathsf X}_k$; extract
           duals $\alpha^{\mathsf X}_d$, $\psi^{\mathsf X}_t$,
           $\mu^{\mathsf X}_k(a,t)$,
           $\nu^{\mathsf X}_j(a,t)$ for $j<k$.
    \State
           $\bar\mu^{\mathsf X}_k(a,t)
           \gets\mu^{\mathsf X}_k(a,t)+\sum_{j<k}\nu^{\mathsf X}_j(a,t)$;
           \quad form $\theta^{\mathsf X}_{d,k}$
           via~\eqref{eq:cg-theta-stela-k}
           or~\eqref{eq:cg-theta-carla-k};
           \quad set
           $\gamma^{\mathsf X}_t\gets-\psi^{\mathsf X}_t$.
    \State $\mathrm{added}\gets\textsc{false}$
    \For{$d\in D$}
      \State $(\pi_d^{\star},\rc_d^{\star})\gets
             \mathcal O\bigl(d,\,
             \{\theta^{\mathsf X}_{d,k}(\cdot,t)\}_{t\in T},\,
             \alpha^{\mathsf X}_d,\,
             \{\gamma^{\mathsf X}_t\}_{t\in T^*},\,
             K_{\mathrm{cand}},\,\texttt{maxSeg}\bigr)$
      \If{$\rc_d^{\star}<-\varepsilon$}
        \State $\Pi_d^{\mathrm{pool}}\gets
               \Pi_d^{\mathrm{pool}}\cup\{\pi_d^{\star}\}$;
               \quad $\mathrm{added}\gets\textsc{true}$
      \EndIf
    \EndFor
  \Until{$\mathrm{added}=\textsc{false}$}
  \State Solve $\mathrm{RMP}^{\mathsf X}_k$ as a MILP over the pool
    
  \State Recover $\bar\lambda_{k}\gets\lambda^{\downarrow}_k(\text{incumbent})$
         by evaluation; freeze it as an inherited level.
\EndFor
\State \Return routing scheme and
$(\bar\lambda_{1},\dots,\bar\lambda_{K})$.
\end{algorithmic}
\end{algorithm}

The oracle $\mathcal O$ is any of \cref{sec:cg-pricing}: \textsc{TCG-Exact}
(EP with $\theta^{\mathsf X}_{d,k}$) for exact certification of the LP,
\textsc{TCG-Diag} (diagonal $\oplus$ top-$K$) or \textsc{TCG-TopK} for
speed. By \Cref{thm:cg-common} all three run unchanged at every rank and for
either formulation: they consume only the nonnegative data
$(\theta^{\mathsf X}_{d,k},\gamma^{\mathsf X}_t)$ and are indifferent to how
those prices were produced. The all-nominal trajectory
$(\mathrm{IGP},\dots,\mathrm{IGP})$ guarantees a feasible starting pool.
The warm-starting scheme of \cref{sec:cg-warm} (\textsc{TCG-Exact-Warm})
applies unchanged: the seed pool is filtered for feasibility, and the
columns are rank- and encoding-independent (\cref{prop:cg-monotone}).

For the theoretical exactness statements, the stopping rule is understood with
$\varepsilon=0$. With a positive numerical tolerance $\varepsilon$, termination
provides a tolerance-dependent LP solution rather than the exact reduced-cost
certificate of \Cref{cor:lpk}.

\subsubsection{Warm-starting the exact oracle from a heuristic pool}
\label{sec:cg-warm}

The oracles of \cref{sec:cg-pricing} trade rigour for speed:
\textsc{TCG-TopK} and \textsc{TCG-Diag} price a demand in (almost)
shortest-path time but do not, on their own, certify optimality over the full
trajectory set, whereas \textsc{TCG-Exact} certifies each rank at the price
of one layered MIP per demand and iteration. The two regimes are
complementary, and the structure of $\mathrm{RMP}^{\mathsf X}_k$ lets us
combine them without weakening any guarantee: a fast heuristic run
\emph{seeds} the column pool, and a single exact pass \emph{certifies} it.
We call this variant \textsc{TCG-Exact-Warm}.

A cold-started exact run spends its first iterations rediscovering, through
expensive layered MIPs, columns that \textsc{TCG-Diag} would have produced
in shortest-path time. Since a column is an intrinsic object; its
coefficients $\Gamma_{d,a,t}(\pi)$ and $\beta_{d,t}(\pi)$ do not depend on
how it was generated; any feasible trajectory
found heuristically is a legitimate member of the exact run's pool. Seeding
therefore replaces the first, most redundant, exact pricing rounds by a cheap
heuristic phase.

\begin{algorithm}[H]
\caption{\textsc{TCG-Exact-Warm}: heuristic seeding then exact certification
under $\mathrm{RMP}^{\mathsf X}_k$}
\label{alg:cg-warm}
\begin{algorithmic}[1]
\Require formulation $\mathsf X\in\{\textsc{Stela},\textsc{Carla}\}$,
demands $D$, ranks $K$, tolerance $\varepsilon>0$, seeding oracle
$\mathcal O_{\mathrm{seed}}\in\{\textsc{TCG-Diag},\textsc{TCG-TopK}\}$.
\Statex\hrulefill\ \textsc{Seeding phase}\ \hrulefill
\State Run \cref{alg:cg-lex} with oracle $\mathcal O_{\mathrm{seed}}$ and
       formulation $\mathsf X$ to convergence of its inner loop; let
       $\Pi_d^{\mathrm{seed}}$ be the pool it accumulates for each
       demand.
\State Discard from $\Pi_d^{\mathrm{seed}}$ any trajectory not feasible in
       every $G_t$ or exceeding $\texttt{maxSeg}$; set
       $\Pi_d^{\mathrm{pool}}\gets
       \{(\mathrm{IGP},\dots,\mathrm{IGP})\}\cup\Pi_d^{\mathrm{seed}}$.
\Statex\hrulefill\ \textsc{Exact certification phase}\ \hrulefill
\For{$k=1,\dots,K$}
  \State Build $\mathrm{RMP}^{\mathsf X}_k$: core $(\mathcal C)$ plus
         rank block
         $(\mathcal P^{\textsc{S}}_k)$~\eqref{eq:cg-Pk} or
         $(\mathcal P^{\textsc{C}}_k)$~\eqref{eq:cg-carla-Pk}.
  \Repeat
    \State Solve the LP relaxation of $\mathrm{RMP}^{\mathsf X}_k$;
           extract duals $\alpha^{\mathsf X}_d$, $\psi^{\mathsf X}_t$,
           $\mu^{\mathsf X}_k(a,t)$,
           $\nu^{\mathsf X}_j(a,t)$ for $j<k$.
    \State $\bar\mu^{\mathsf X}_k(a,t)\gets
           \mu^{\mathsf X}_k(a,t)+\sum_{j<k}\nu^{\mathsf X}_j(a,t)$;
           \quad form $\theta^{\mathsf X}_{d,k}$
           via~\eqref{eq:cg-theta-stela-k}
           or~\eqref{eq:cg-theta-carla-k};
           \quad set
           $\gamma^{\mathsf X}_t\gets-\psi^{\mathsf X}_t$.
    \State $\mathrm{added}\gets\textsc{false}$
    \For{each demand $d\in D$}
      \State Solve the \emph{exact} layered pricing (EP) with costs
             $\theta^{\mathsf X}_{d,k}$ and budget prices
             $\gamma^{\mathsf X}_t$ (\cref{thm:exact})
             $\to(\pi_d^{\star},\,
             \operatorname{val}(\mathrm{SP}^{\mathsf X}_{d,k}))$.
      \If{$\operatorname{val}(\mathrm{SP}^{\mathsf X}_{d,k})
           <\alpha^{\mathsf X}_d-\varepsilon$}
        \State $\Pi_d^{\mathrm{pool}}\gets
               \Pi_d^{\mathrm{pool}}\cup\{\pi_d^{\star}\}$;
               \quad $\mathrm{added}\gets\textsc{true}$
      \EndIf
    \EndFor
  \Until{$\mathrm{added}=\textsc{false}$}
  \State Solve $\mathrm{RMP}^{\mathsf X}_k$ as a MILP over the pool
        
  \State Recover $\bar\lambda_{k}\gets
         \lambda^{\downarrow}_k(\text{incumbent})$ by evaluation;
         freeze it as an inherited level.
\EndFor
\State \Return routing scheme and
$(\bar\lambda_{1},\dots,\bar\lambda_{k})$.
\end{algorithmic}
\end{algorithm}

\subsection{Numerical experiments}

We evaluate both the \textsc{Stela}-driven and \textsc{Carla}-driven TCGs on the complete \textsc{setAP} benchmark, comprising 125 instances with network sizes ranging from $N=100$ to $N=500$ nodes. For each instance, the warm-start procedure is initialized from the solution produced by the corresponding diagnostic TCG, \textsc{TCG-DIAG}. We impose a wall-clock time limit of 3600 seconds per instance, while each rank subproblem in the integer master is given a time limit of 1000 seconds. The two variants are run independently under the same computational setting. To compare the quality of the solutions produced by the two approaches, we reconstruct the complete sorted load vector obtained by each TCG and perform a pairwise lexicographic comparison. An instance is counted as a win for a method if its sorted load vector is lexicographically smaller than that of the other method, while instances for which the two vectors are equal are counted as ties.~\Cref{fig:winner by N} reports, for each network size $N$, the number of instances won by the \textsc{Stela}-driven TCG, the number of ties, and the number of instances won by the \textsc{Carla}-driven TCG.

\subsubsection{Who wins}
\begin{figure}[H]
    \centering
    \includegraphics[width=0.4\linewidth]{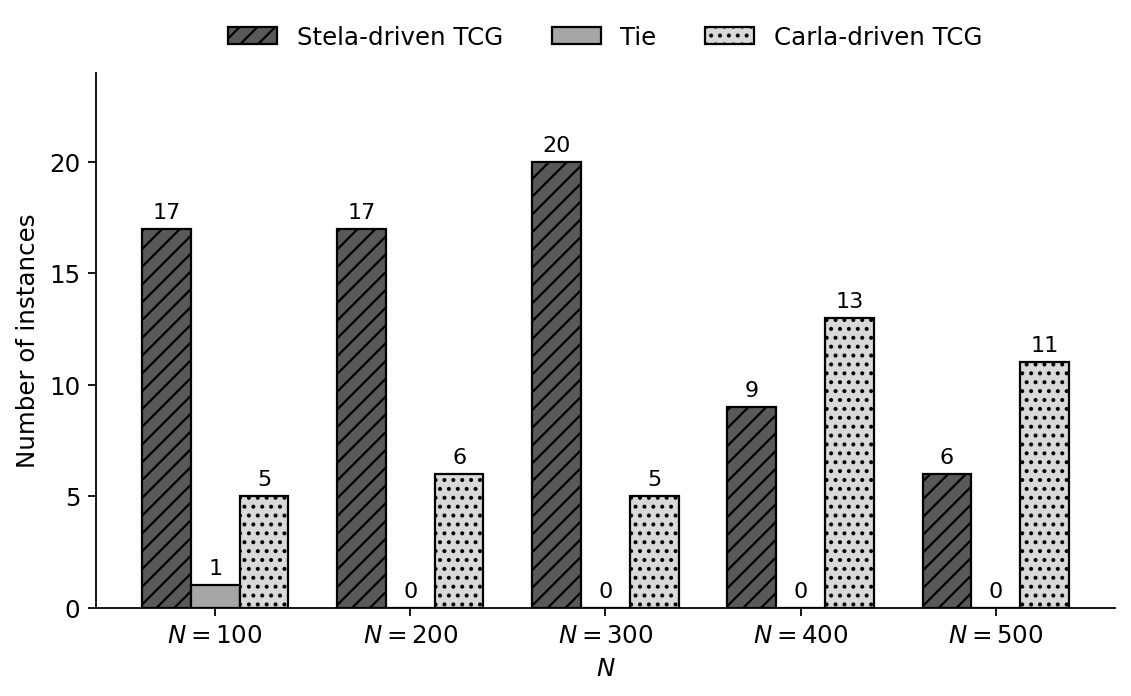}
    \caption{Lexicographic comparison of the \textsc{Stela}-driven and
\textsc{Carla}-driven TCGs by network size $N$. Bars report the number
of instances won by each method or resulting in a tie, based on the sorted load vector.}
    \label{fig:winner by N}
\end{figure}

\subsubsection{Where do they differ}
While the previous comparison of~\Cref{fig:winner by N} identifies which TCG variant yields the lexicographically better solution, it does not indicate where the two solutions start to differ. To further characterize this difference, for each instance we define $k^\star$ as the first rank at which the two complete sorted load vectors are different. Hence, $k^\star=1$ means that the two approaches differ already on the largest load, whereas a larger value of $k^\star$ indicates that their solutions agree on the preceding ranks and differ only at a later level of the lexicographic objective.~\Cref{fig:first_differing_rank_by_N} reports the distribution of $k^\star$ across instances for each network size $N$.
\begin{figure}[H]
    \centering
    \includegraphics[width=0.4\linewidth]{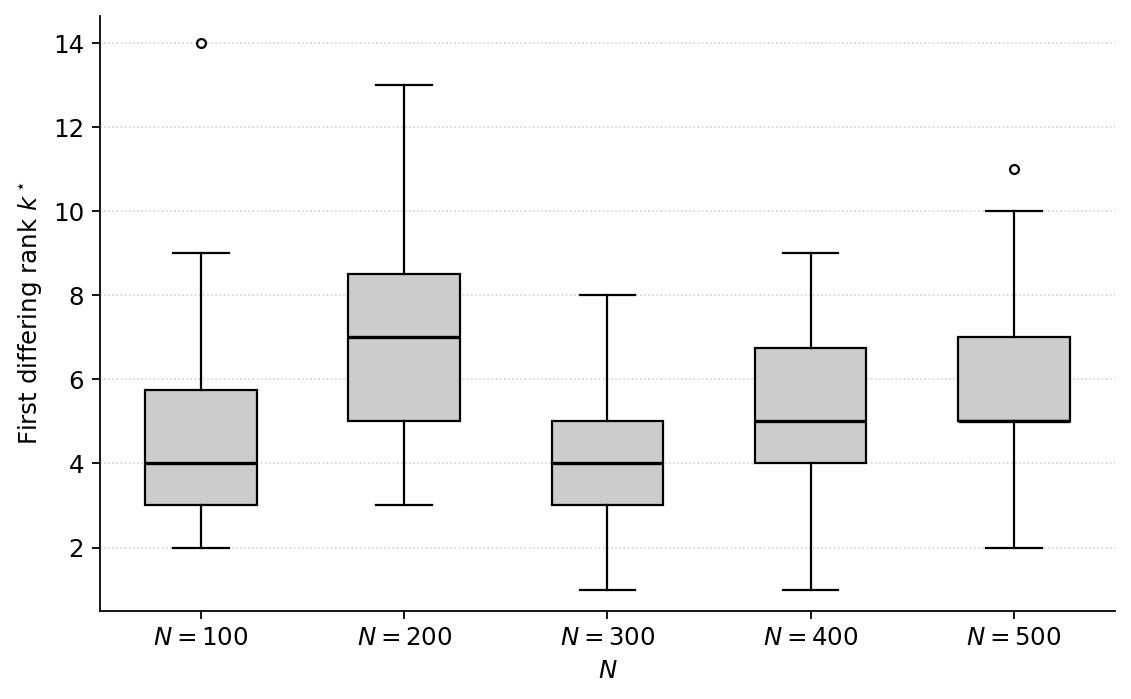}
    \caption{Distribution of the first rank $k^\star$ at which the sorted load vectors produced by the \textsc{Stela}-driven and \textsc{Carla}-driven TCGs differ, reported by network size $N$. Tied instances are excluded since no differing rank exists.}

    \label{fig:first_differing_rank_by_N}
\end{figure}

\subsubsection{By how much}
It is also informative to examine the \emph{magnitude} of the difference
between the two solutions at the first rank $k^\star$ where their
lexicographic load vectors diverge. This complements the previous
analysis, which examines \emph{where} the two approaches first differ,
by showing \emph{by how much} they differ at this rank. ~\Cref{fig:by-how-much}
reports the magnitude of this difference across the different network sizes.
\begin{figure}[H]
    \centering
    \includegraphics[width=0.4\linewidth]{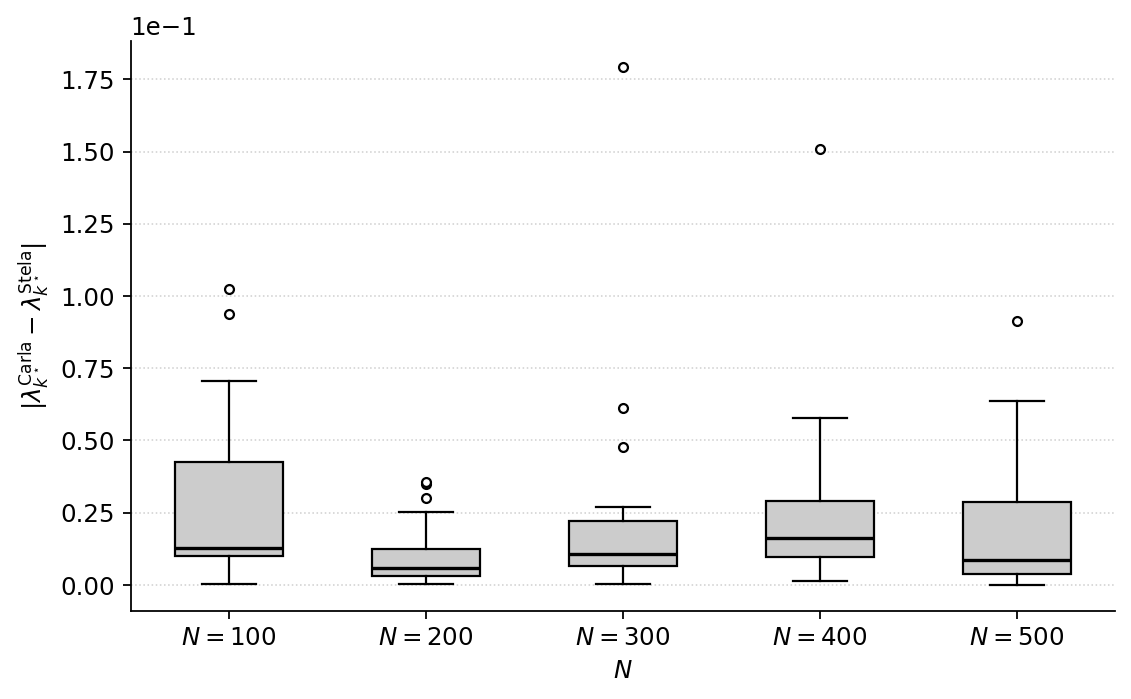}
   \caption{Magnitude of the difference between the \textsc{Stela}-driven
and \textsc{Carla}-driven TCGs at the first rank $k^\star$ where their
lexicographically sorted load vectors differ, reported by network size
$N$. Tied instances are excluded.}
\label{fig:by-how-much}

    \label{fig:by_how_much}
\end{figure}

\subsubsection{Conclusion and discussion}
\label{sec:cg-discussion}

The three views combine into one message. \Cref{fig:winner by N} shows a
size-dependent trend: the \textsc{Stela}-driven TCG wins more often on the small and
medium classes ($N \leq 300$), while the \textsc{Carla}-driven TCG catches up and
slightly leads on the two largest ($N = 400, 500$). But \Cref{fig:first_differing_rank_by_N}
and \Cref{fig:by-how-much} show that this sign says little about how far apart the
solutions actually are. Whichever variant wins, it wins early; the median first
differing rank $k^\star$ stays between $4$ and $7$ regardless of $N$, so the two
approaches agree on the largest loads and in particular almost always on the maximum
link utilisation. And they win by very little: at the decisive rank the two loads
differ by about one point of link utilisation on the median, well under three points
for most instances. Since lexicographic comparison is all-or-nothing, an arbitrarily
small gap at rank $k^\star$ is enough to label an instance a win, so the bar counts
overstate the real distance between the two solutions.

The practical conclusion is therefore simple: the two variants produce routing
schemes of comparable quality. \textsc{Stela}-driven TCG is a marginally safer
default at small and medium sizes, but the gap is small in magnitude even when it is
consistent in sign, and it vanishes on the largest instances. 

Both variants remain heuristics because column generation is applied only at the
root of the integer master, without branching to close the resulting integrality
gap; the next section (\Cref{bp-sec:main}) develops a branch-and-price extension that embeds the
pricing within the search tree to address this.

%% file: sections/06-branchprice.tex
\section{A Branch-and-Price Extension of Trajectory Column Generation}
\label{bp-sec:main}

\subsection{State of the art and positioning}
\label{bp-sec:sota}
\subsubsection{State of the art}
The combination of Dantzig--Wolfe decomposition and column generation is a
standard approach for large-scale linear and integer programs whose natural
formulations contain an exponential number of variables. The decomposition
principle of Dantzig and Wolfe replaces the original formulation by a master
problem whose columns represent feasible solutions of independent
subproblems, while column generation avoids explicit enumeration of these
columns by solving a pricing problem driven by the dual variables of the
restricted master problem~\cite{dantzig1960}. For integer programs, however,
column generation alone provides a certificate only for the linear relaxation
of the Dantzig--Wolfe reformulation. Exact integer optimization requires the
column-generation procedure to be embedded within a branch-and-bound
framework, leading to what is now known as
branch-and-price~\cite{vanderbeck1996,barnhart1998}.

A central difficulty in branch-and-price is that branching decisions cannot
be chosen independently of pricing. A conventional branch on a master
variable may exclude a single column from an exponentially large family and
therefore be difficult to propagate efficiently to the pricing problem. More
generally, an arbitrary branching constraint can substantially modify the
structure of the subproblem and destroy the algorithmic properties that make
column generation attractive. This interaction between branching and pricing
has motivated both problem-specific branching schemes and generic approaches
designed to preserve the structure of the pricing
problem~\cite{barnhart1998,vanderbeck2011}. In particular, Vanderbeck's
generic branching framework emphasises branching decisions that can be
enforced directly in the underlying formulation or in the pricing problem,
avoiding the explicit enumeration of excluded columns and thereby preserving
the decomposition structure~\cite{vanderbeck2011}.

The present work lies within this line of research but considers a different
structural setting. The columns of the Dantzig--Wolfe reformulation are
complete multi-period trajectories rather than single-period paths or
classical vehicle routes. A trajectory simultaneously determines segment
usage at every time period and the resulting inter-period reconfiguration
cost. Consequently, a branching decision on a trajectory variable is
naturally expressed through the underlying segment-usage variables of the
compact formulation. We exploit this correspondence to instantiate an
original-variable branching principle in the trajectory space: branching on
an aggregate segment-usage quantity $X_{d,t,e}$ is equivalent, in the
pricing problem, to fixing the corresponding segment-usage variable $x^t_e$
to zero or one. The resulting pricing problem therefore retains its layered
structure.

The distinction between the generic principle and the present contribution
is important. We do not claim that branching on original variables, or
branching compatible with pricing, is new. These ideas are well established
in the branch-and-price
literature~\cite{vanderbeck1996,barnhart1998,vanderbeck2011}. Our
contribution is instead to characterise their exact realisation for the
trajectory formulation underlying the TCG framework of \cref{sec:cg-lex}. In
particular, under a simple-path assumption, we establish that segment-usage
branching is valid, can be propagated to pricing by variable fixing, is
complete for the trajectory variables, and yields a finite branch-and-bound
tree. We then integrate this branch-and-price mechanism into the sequential
lexicographic decomposition for both $\mathrm{RMP}^{\textsc{Stela}}_k$ and
$\mathrm{RMP}^{\textsc{Carla}}_k$, and distinguish explicitly between the
certificates delivered by root column generation, price-and-branch over a
generated pool, and full branch-and-price.

This distinction is particularly relevant for the computational
interpretation of \textsc{TCG-Exact}. Exact pricing at the root certifies
the LP optimum of the trajectory master, whereas solving the resulting
restricted integer master certifies integer optimality only over the
generated column pool. Global integer optimality requires pricing to be
re-run under the restrictions induced by every branch-and-bound node. The
full branch-and-price algorithm developed in this section is therefore used
primarily as an exact reference method: its role is to identify the gap, if
any, between the restricted-pool integer optimum delivered by
\textsc{TCG-Exact} and the true integer optimum of the trajectory
formulation.

\subsubsection{What column generation certifies}
\label{bp-sec:certif-gap}

The trajectory column generation (TCG) scheme of the previous sections
solves, at each rank of the sequential decomposition, the \emph{linear
relaxation} of the trajectory master $\mathrm{RMP}^{\mathsf X}_k$ to proven
optimality. This section asks the question that this leaves open: \emph{what
exactly does that column generation certify, and what is needed to strengthen
the certificate to global integer optimality?}

The answer separates three claims that are routinely conflated. Column
generation at a fixed rank certifies the LP value of
$\mathrm{RMP}^{\mathsf X}_k$; it does \emph{not} certify integer
optimality, because the root column pool need not contain the trajectories
used by an integer optimum. Price-and-branch: optimising the integer
master over the root pool, certifies integer optimality \emph{over that
pool} and nothing more. Only branch-and-price, in which pricing is re-run
under each node's restrictions, certifies the global integer optimum of the
rank. This section develops that last scheme and states precisely, at each
step, which certificate is delivered.

The contribution over the plain TCG scheme is threefold. First
(\cref{bp-sec:certif-detail}), we identify the exact gap between the LP
certificate TCG delivers and the integer optimum. Second
(\cref{bp-sec:branching}), we show that the natural rule of branching on
master variables, although valid, destroys the pricing structure, and we
replace it by branching on the aggregate segment-usage quantities, which
are exactly the original compact variables $x^{dt}_{ij}$,  proving that
this rule is valid, complete, finite, and absorbed by the pricing problem at
no structural cost.

\subsubsection{Setting recalled}
\label{bp-sec:recall}
 
We recall only what the branch-and-price argument needs; all derivations are
those of the trajectory column generation sections and are not repeated
here. Throughout, $\mathsf X\in\{\textsc{Stela},\textsc{Carla}\}$ indexes
the formulation, and $e=(i,j)$ denotes a segment.
 
\paragraph{Trail restriction.}
We work under the standing modelling assumption of \cref{chap:model}: within
each period, the path of a trajectory is a \emph{trail} in the segment
digraph, i.e.\ no segment $(i,j)$ is used more than once in a period, while a
node may recur as a waypoint. Thus, for example $\langle a,b,c,d,c,e\rangle$ is admissible
(node $c$ repeats, all segments distinct) whereas $\langle a,b,a,b\rangle$ is
not (segment $(a,b)$ repeats). This is exactly what makes the segment-usage
indicator $\delta(p^{(t)},e)\in\{0,1\}$, hence the compact variable
$x^{dt}_{ij}\in\{0,1\}$, so that the compact model. The trail restriction is
enforced in the pricing MIP~(EP) of \cref{sec:cg-exact} by allowing each
segment digraph arc at most once per period. All exactness results below are
stated over this (trail-restricted) T-ASR.
 
\paragraph{Trajectories and column coefficients.}
A \emph{trajectory} of demand $d\in D$ is a choice of one feasible segment
path at every period,
$\pi=(p^{(0)},\dots,p^{(h-1)})\in\Pi_d$ (\cref{def:cg-traj}), where
$\Pi_d$ is finite and exponentially large in $h$ and $\texttt{maxSeg}$. For
a segment $e=(i,j)$ we write $\delta(p^{(t)},e)\in\{0,1\}$ for its usage in
$p^{(t)}$. The column coefficients of $\pi$ are the load footprint
$\Gamma_{d,a,t}(\pi)$ of~\eqref{eq:cg-Gamma} and the reconfiguration cost
$\beta_{d,t}(\pi)$ of~\eqref{eq:cg-beta}; once $\pi$ is fixed, both are
constants. All combinatorial difficulty is confined to the pricing problem.
 
\paragraph{The rank-$k$ restricted master $\mathrm{RMP}^{\mathsf X}_k$.}
As established in \cref{sec:cg-lex}, the rank-$k$ restricted master for
formulation $\mathsf X$ is
\[
\mathrm{RMP}^{\mathsf X}_k
\;=\;(\mathcal C)\;+\;(\mathcal P^{\mathsf X}_k),
\]
where the core $(\mathcal C)$ of~\eqref{eq:cg-core} contains the selection
rows $(\mathrm S)$, the load-assembly rows $(\mathrm L)$ and the budget rows
$(\mathrm B)$, and the rank block $(\mathcal P^{\mathsf X}_k)$ is
$(\mathcal P^{\textsc{S}}_k)$~\eqref{eq:cg-Pk} for $\mathsf X=\textsc{Stela}$ or
$(\mathcal P^{\textsc{C}}_k)$~\eqref{eq:cg-carla-Pk} for
$\mathsf X=\textsc{Carla}$.
 
\paragraph{Reduced cost and pricing at a node.}
At a fixed rank $k$ and a fixed branch-and-bound node $u$, the LP relaxation
of $\mathrm{RMP}^{\mathsf X}_k$ restricted to the trajectories
$\Pi_d(u)\subseteq\Pi_d$ admissible at $u$ yields optimal duals
$\alpha^{\mathsf X}_d$, $\psi^{\mathsf X}_t$, $\mu^{\mathsf X}_k(a,t)$ and
$\nu^{\mathsf X}_j(a,t)$ for $j<k$. The stratified dual
$\bar\mu^{\mathsf X}_k(a,t)$ of~\eqref{eq:cg-stratified-stela}
or~\eqref{eq:cg-stratified-carla}, the effective segment cost
$\theta^{\mathsf X}_{d,k}(e,t)$ of~\eqref{eq:cg-theta-stela-k}
or~\eqref{eq:cg-theta-carla-k}, and the budget price
$\gamma^{\mathsf X}_t:=-\psi^{\mathsf X}_t\ge0$ are formed as in
\cref{sec:cg-lex}. The reduced cost of a trajectory
$\pi\in\Pi_d$ is then~\eqref{eq:cg-rc-rankk}:
\begin{equation}
  \rc^{\mathsf X}_{d,k}(\pi) = -\alpha^{\mathsf X}_d
    + \sum_{t\in T}\sum_{e\in\segs(p^{(t)})}\theta^{\mathsf X}_{d,k}(e,t)
    + \sum_{t\in T^*}\gamma^{\mathsf X}_t\,\dist(p^{(t-1)},p^{(t)}).
  \label{bp-eq:rc}
\end{equation}
Writing $f^{\mathsf X}_{d,k}(\pi)$ for the two rightmost sums, the
\emph{node pricing problem} for demand $d$ restricted to the node's
trajectory set $\Pi_d(u)$ is
\begin{equation}
  \operatorname{val}(\mathrm{SP}^{\mathsf X}_{d,k}(u))
  = \min_{\pi\in\Pi_d(u)} f^{\mathsf X}_{d,k}(\pi),
  \label{bp-eq:val}
\end{equation}
and $d$ admits a negative-reduced-cost trajectory at $u$ if and only if
$\operatorname{val}(\mathrm{SP}^{\mathsf X}_{d,k}(u)) <
\alpha^{\mathsf X}_d$.
 
The pricing problem is solved as the layered MIP (EP)
of \cref{sec:cg-exact}, carrying binary usage variables $x^t_e$ with
$x^t_e=\delta(p^{(t)},e)$ at optimum. The only property we need downstream
is the following: branching decisions enter the pricing problem as
\emph{variable fixings} on these usage variables,
\begin{equation}
  \text{$e$ forbidden for $d$ at $t$}\Rightarrow x^t_e=0,
  \qquad
  \text{$e$ required for $d$ at $t$}\Rightarrow x^t_e=1,
  \label{bp-eq:node-fix}
\end{equation}
each equivalent to deleting or forcing one arc of the segment digraph,
leaving the layered structure intact.
 
Finally, we recall that the column pool is global across ranks.
 
\begin{lemma}
\label{bp-lem:pool}
A trajectory generated at rank $k$ remains feasible and correctly costed at
every rank $k'\neq k$: the coefficients $\Gamma_{d,a,t}(\pi)$
and $\beta_{d,t}(\pi)$ depend on $\pi$ alone
(\cref{prop:cg-monotone}), and the rank index enters
$\mathrm{RMP}^{\mathsf X}_k$ only through the rank block
$(\mathcal P^{\mathsf X}_k)$, which contains no $\xi_{d,\pi}$.
\end{lemma}
 
\subsubsection{The gap between LP and integer certificates}
\label{bp-sec:certif-detail}
 
Column generation at a fixed rank solves the LP relaxation of
$\mathrm{RMP}^{\mathsf X}_k$ exactly: on exit, no trajectory in
$\bigcup_d\Pi_d$ has negative reduced cost, so the restricted-master LP
optimum is the full-master LP optimum (\cref{cor:lpk}). It says nothing
about integrality.
 
The gap to the integer optimum is not a formality. At the root LP optimum a
demand may split fractionally between two trajectories, neither of which is
used by any integer optimum; the trajectory that an integer optimum needs
may never have had negative reduced cost \emph{at the root duals}, and so
may never enter the pool. Optimising the integer master over that pool;
price-and-branch; then provably certifies the LP value and provably
certifies nothing about the integer value. Closing the gap requires
re-pricing under each branching node's restrictions, at which point the duals
differ and the missing columns can appear. That is branch-and-price,
developed next.
 
\subsection{Branching}
\label{bp-sec:branching}
 
To obtain global integer optimality we embed column generation in a
branch-and-bound tree in which pricing is re-run at every node under that
node's restrictions. Everything hinges on a branching rule that partitions
the integer feasible set \emph{without destroying the pricing problem}.

\subsubsection{Why branching on master variables fails}
\label{bp-sec:col-branching}
 
The obvious rule branches on a fractional master variable
$0<\xi^{\star}_{d,\hat\pi}<1$, creating children $\xi_{d,\hat\pi}=1$ and
$\xi_{d,\hat\pi}=0$. It is valid; in any integer solution
$\xi_{d,\hat\pi}\in\{0,1\}$, so exactly one child holds; but unusable,
for two reasons.
 
The tree is unbalanced: the child $\xi_{d,\hat\pi}=1$ fixes a demand
completely, while $\xi_{d,\hat\pi}=0$ excludes a single trajectory out of
the exponential set $\Pi_d$, so the bound barely moves. Worse, the pricing
problem is destroyed: the restriction
$\Pi_d(u')=\Pi_d\setminus\{\hat\pi\}$ is not a variable fixing but a
no-good inequality
\begin{equation}
  \sum_{(t,e):\hat x^t_e=1}(1-x^t_e)
  \;+\; \sum_{(t,e):\hat x^t_e=0}x^t_e \;\ge\; 1,
  \label{bp-eq:nogood}
\end{equation}
and one such dense row accumulates per excluded trajectory along the path to
the node. At depth $\delta$ the pricing MIP carries $\delta$ no-good rows
and the layered structure that made~\eqref{bp-eq:val} tractable is gone.
This is the classical objection to branching on Dantzig--Wolfe master
variables~\cite{barnhart1998}. We record it and do not use it.

\subsubsection{Segment-usage branching}
\label{bp-sec:x-branching}
 
For a demand $d$, a period $t\in T$ and a segment $e=(i,j)$, define the
\emph{aggregate segment-usage}
\begin{equation}
  X_{d,t,e} = \sum_{\pi\in\Pi_d(u)}\delta(p^{(t)}_\pi,e)\,\xi_{d,\pi}.
  \label{bp-eq:X}
\end{equation}
This is not a new variable. With $e=(i,j)$,
\begin{equation}
  X_{d,t,e}
  \;=\;\sum_{\pi\in\Pi_d(u)}\delta(p^{(t)}_\pi,e)\,\xi_{d,\pi}
  \;=\;x^{dt}_{ij},
  \label{bp-eq:X-is-x}
\end{equation}
i.e.\ $X_{d,t,e}$ is exactly the compact segment-usage variable
$x^{dt}_{ij}$ of the original formulation (\cref{chap:model}), read in the
reformulated space; equivalently, $x^{dt}_{ij}$ is the projection of the
Dantzig--Wolfe variables $\xi$ onto the compact space. Segment-usage
branching is therefore branching on the \emph{original decision variables}
themselves, and every result below is stated in their terms. We use
$X_{d,t,e}$ and $x^{dt}_{ij}$ interchangeably, with $e=(i,j)$.
 
\begin{lemma}
\label{bp-lem:x-valid}
In every integer feasible solution of a node, $X_{d,t,e}\in\{0,1\}$; hence
the children $X_{d,t,e}=0$ and $X_{d,t,e}=1$ partition the integer feasible
solutions of the parent.
\end{lemma}
 
\begin{proof} \[ \begin{array}{:c} \begin{minipage}{0.97\textwidth}
By the selection row $(\mathrm S)$~\eqref{eq:cg-core-S}, exactly one
$\pi_0\in\Pi_d(u)$ has $\xi_{d,\pi_0}=1$; then~\eqref{bp-eq:X} reduces to
$X_{d,t,e}=\delta(p^{(t)}_{\pi_0},e)\in\{0,1\}$, so every integer solution
satisfies exactly one child constraint.
\end{minipage} \end{array} \]\end{proof}
 
\begin{lemma}
\label{bp-lem:x-compat}
The constraint $X_{d,t,e}=0$ (resp.\ $=1$) is equivalent to the restriction
$\Pi_d(u')=\{\pi\in\Pi_d(u):\delta(p^{(t)}_\pi,e)=0\}$ (resp.\ $=1$),
imposed in the pricing problem by the single fixing $x^t_e=0$ (resp.\
$x^t_e=1$) of~\eqref{bp-eq:node-fix}. The decision affects only demand~$d$
and does not couple demands.
\end{lemma}
 
\begin{proof} \[ \begin{array}{:c} \begin{minipage}{0.97\textwidth}
Consider $X_{d,t,e}=0$. With $\xi\ge0$,
$\delta(p^{(t)}_\pi,e)\in\{0,1\}$ and
$\sum_\pi\xi_{d,\pi}=1$, the equation
$\sum_\pi\delta(p^{(t)}_\pi,e)\,\xi_{d,\pi}=0$ forces $\xi_{d,\pi}=0$ for
every $\pi$ with $\delta(p^{(t)}_\pi,e)=1$, and any $\xi$ supported on
$\{\pi:\delta(p^{(t)}_\pi,e)=0\}$ satisfies it; the constraint is thus
equivalent to removing those trajectories from $\Pi_d(u)$. Since the pricing
usage variables satisfy $x^t_e=\delta(p^{(t)},e)$, this is expressed exactly
by $x^t_e=0$. The case $X_{d,t,e}=1$ is symmetric. Both~\eqref{bp-eq:X}
and the fixing involve only~$d$.
\end{minipage} \end{array} \]\end{proof}
 
\cref{bp-lem:x-compat} is the whole point: after branching, the pricing
problem is the same layered MIP with one arc deleted or forced; no
no-good rows, no structural degradation, and the same warm-starts apply.
This is exactly what column branching does not provide.

\subsubsection{Completeness}
\label{bp-sec:completeness}
 
Validity alone is not enough; we must know that when no $X$ is fractional,
an integer routing is already in hand. Under the trail restriction, distinct
trajectories may share the same segment sets at every period; e.g.\
$\langle s,h,x,h,y,h,t\rangle$ and $\langle s,h,y,h,x,h,t\rangle$; , so the
master variable $\xi$ may stay fractional even when every $X_{d,t,e}$ is
integral. This is harmless: by~\eqref{bp-eq:X-is-x} these trajectories induce
the \emph{same} values $x^{dt}_{ij}$, hence the same routing, and the object
T-ASR optimises is $x$, not $\xi$. The following lemma makes this precise,
and it is the reason the branching rule targets integrality of $X$ (that is,
of the original variables $x$) rather than of $\xi$.
 
\begin{lemma}
\label{bp-lem:coeff-equiv}
The column coefficients of a trajectory depend only on its \emph{signature}
$\sigma_d(\pi):=\big(\delta(p^{(t)}_\pi,e)\big)_{t\in T,\,e}
\in\{0,1\}^{T\times E}$: if $\sigma_d(\pi)=\sigma_d(\pi')$ then
$\Gamma_{d,a,t}(\pi)=\Gamma_{d,a,t}(\pi')$ for all $a,t$ and
$\beta_{d,t}(\pi)=\beta_{d,t}(\pi')$ for all $t$. Consequently $\pi$ and
$\pi'$ have identical columns in $\mathrm{RMP}^{\mathsf X}_k$ at every rank.
\end{lemma}
 
\begin{proof} \[ \begin{array}{:c} \begin{minipage}{0.97\textwidth}
By~\eqref{eq:cg-Gamma},
$\Gamma_{d,a,t}(\pi)=\sum_{e=(i,j)}r(i,j,a,t)\,\nu(d,t)\,
\delta(p^{(t)}_\pi,e)$, and by~\eqref{eq:cg-beta},
$\beta_{d,t}(\pi)=\sum_e|\delta(p^{(t)}_\pi,e)-\delta(p^{(t-1)}_\pi,e)|$;
both are functions of $\sigma_d(\pi)$ alone. Every core and rank row of
$\mathrm{RMP}^{\mathsf X}_k$ references $\pi$ only through these coefficients
together with the unit selection coefficient, so the two columns coincide.
\end{minipage} \end{array} \]\end{proof}
 
\begin{lemma}
\label{bp-lem:completeness}
Let $\xi$ be a node restricted-master LP solution with
$\sum_{\pi\in\Pi_d(u)}\xi_{d,\pi}=1$ and $\xi\ge0$ for every $d$, and suppose
$X_{d,t,e}(\xi)\in\{0,1\}$ for all $d,t,e$. Then, via~\eqref{bp-eq:X-is-x},
$x^{dt}_{ij}:=X_{d,t,e}$ is an integral routing that is feasible for the
node's rank-$k$ compact problem and whose rank-$k$ objective equals the node
LP objective of $\xi$. Moreover every $\pi\in\Pi_d(u)$ with $\xi_{d,\pi}>0$
has signature $\sigma_d(\pi)=x^{d\cdot}$ and realises this routing for
demand~$d$.
\end{lemma}
 
\begin{proof} \[ \begin{array}{:c} \begin{minipage}{0.97\textwidth}
Fix $d$ and let $S=\{\pi:\xi_{d,\pi}>0\}\neq\emptyset$. For fixed $t,e$,
$X_{d,t,e}=\sum_{\pi\in S}\delta(p^{(t)}_\pi,e)\,\xi_{d,\pi}$ is a convex
combination with strictly positive weights (summing to $1$) of values in
$\{0,1\}$; it lies in $\{0,1\}$ iff all those values are equal. By hypothesis
$X_{d,t,e}\in\{0,1\}$, so $\delta(p^{(t)}_\pi,e)=X_{d,t,e}$ for all
$\pi\in S$. As $t,e$ range, every $\pi\in S$ shares the full signature
$\sigma_d(\pi)=x^{d\cdot}$; in particular a single trajectory of $S$ realises
$x^{d\cdot}$ across all periods.
 
\emph{Feasibility.} Each $\pi\in S$ being a feasible trajectory, its
period-$t$ path is a trail $s_d\!\to\!t_d$ with incidence vector
$(x^{dt}_{ij})_{ij}$; hence $x^{d\cdot}$ satisfies flow conservation
\eqref{ctn:flow}, the segment bound \eqref{ctn:segments}, and; the whole
signature being that of one $\pi\in S$; the budget coupling
\eqref{ctn:budget_3} across periods. The node fixings are respected by every
$\pi\in S$, hence by $x^{d\cdot}$. Integrality of $x$ is the hypothesis.
 
\emph{Objective.} By \cref{bp-lem:coeff-equiv} every $\pi\in S$ contributes
identical $\Gamma,\beta$; since $\sum_{\pi\in S}\xi_{d,\pi}=1$, the
$\xi$-activity of every load-assembly row $(\mathrm L)$, budget row
$(\mathrm B)$ and rank-block row equals its activity at the single
representative $x^{d\cdot}$. As $d$ was arbitrary, the assembled routing $x$
reproduces every row activity of $\xi$; in particular it is feasible for
$\mathrm{RMP}^{\mathsf X}_k$ at the node and attains the same rank-$k$
objective.
\end{minipage} \end{array} \]\end{proof}

\subsubsection{Finiteness}
\label{bp-sec:finiteness}
 
\begin{lemma}
\label{bp-lem:finite}
Branch-and-bound with segment-usage branching, applied to
$\mathrm{RMP}^{\mathsf X}_k$, explores a tree of depth at most
$|D|\cdot h\cdot|V|^2$, and at every leaf either the node is infeasible or
all quantities $X_{d,t,e}$ are integral, in which case
\cref{bp-lem:completeness} induces;  via~\eqref{bp-eq:X-is-x}; an
integral compact routing $x^{dt}_{ij}=X_{d,t,e}$ of the same objective value.
\end{lemma}
 
\begin{proof} \[ \begin{array}{:c} \begin{minipage}{0.97\textwidth}
Each branching decision fixes one of the $|D|\cdot h\cdot|V|^2$ quantities
$X_{d,t,e}=x^{dt}_{ij}$ to $\{0,1\}$ and is never branched on again along the
same root-to-node path, giving the depth bound. A node admits a branching
exactly when some $X_{d,t,e}$ is fractional in its LP solution; so at a leaf
either the node LP is infeasible, or every $X_{d,t,e}\in\{0,1\}$ and
\cref{bp-lem:completeness} induces an integral routing $x$, feasible at the
node with objective equal to the node LP value.
\end{minipage} \end{array} \]\end{proof}
 
Unlike column branching, whose finiteness relies on exhausting the
exponential set $\Pi_d$, \cref{bp-lem:finite} bounds the depth by a
quantity of the same order as the compact model's variable count: large, but
polynomial and independent of $|\Pi_d|$.
 
\subsection{Branch-and-price at one rank}
\label{bp-sec:bp-rank}
 
\begin{algorithm}[H]
\caption{Branch-and-price for $\mathrm{RMP}^{\mathsf X}_k$}
\label{bp-alg:bp-rank}
\begin{algorithmic}[1]
\Require formulation $\mathsf X\in\{\textsc{Stela},\textsc{Carla}\}$,
inherited values $\bar\lambda_{1},\dots,\bar\lambda_{k-1}$, global column pools
$\Pi_d^{\mathrm{pool}}$.
\Ensure $\bar\lambda_{k}$ and an optimal rank-$k$ integral routing.
\State Create root $u_0$ with $\Pi_d(u_0)=\Pi_d$ for all $d$;
  \quad $L\gets\{u_0\}$; \quad $UB\gets+\infty$.
\While{$L\neq\emptyset$}
  \State Select and remove a node $u$ from $L$.
  \Repeat \Comment{column generation at node $u$}
    \State Solve the LP relaxation of $\mathrm{RMP}^{\mathsf X}_k$
           over $\Pi_d^{\mathrm{pool}}\cap\Pi_d(u)$.
    \State Extract duals $\alpha^{\mathsf X}_d$, $\psi^{\mathsf X}_t$,
           $\mu^{\mathsf X}_k(a,t)$, $\nu^{\mathsf X}_j(a,t)$ for $j<k$;
           form $\bar\mu^{\mathsf X}_k$, $\theta^{\mathsf X}_{d,k}$ and
           $\gamma^{\mathsf X}_t$.
    \ForAll{demands $d$ not fixed at $u$}
      \State Solve the pricing MIP (EP) with the node fixings
             $\to\operatorname{val}(\mathrm{SP}^{\mathsf X}_{d,k}(u))$.
      \If{$\operatorname{val}(\mathrm{SP}^{\mathsf X}_{d,k}(u))
           <\alpha^{\mathsf X}_d$}
        \State Add the trajectory to $\Pi_d^{\mathrm{pool}}$.
      \EndIf
    \EndFor
  \Until{$\operatorname{val}(\mathrm{SP}^{\mathsf X}_{d,k}(u))
         \ge\alpha^{\mathsf X}_d$ for every unfixed demand $d$}
  \State $LB(u)\gets$ the certified node LP value.
  \If{$LB(u)\ge UB$} \Comment{fathom by bound}
    \State \textbf{continue}
  \ElsIf{$X_{d,t,e}(\xi^{\star})\in\{0,1\}$ for all $d,t,e$}
    \State $x^{dt}_{ij}\gets X_{d,t,e}(\xi^{\star})$;
           \Comment{integral routing by \cref{bp-lem:completeness}}
    \State $UB\gets LB(u)$; record $x$ as incumbent; \textbf{continue}.
  \Else
    \State Pick $(d,t,e)$ with $X_{d,t,e}(\xi^{\star})=x^{dt}_{ij}$ fractional.
    \State Add to $L$ the children with $x^t_e=0$ and $x^t_e=1$ in the
           pricing of $d$.
  \EndIf
\EndWhile
\State \Return the incumbent $x$ and its value $\bar\lambda_{k}$.
\end{algorithmic}
\end{algorithm}
 
\begin{lemma}
\label{bp-lem:cg-finite}
For a fixed node $u$ and rank $k$, the column-generation loop of lines
4--12 of \cref{bp-alg:bp-rank} terminates after finitely many iterations.
\end{lemma}
 
\begin{proof} \[ \begin{array}{:c} \begin{minipage}{0.97\textwidth}
At an LP optimum every pooled column has nonnegative reduced cost, so a
column added at line~10, having strictly negative reduced cost, is new. Each
non-terminating iteration adds at least one element of the finite set
$\bigcup_d\Pi_d(u)$, bounding the count by $\sum_d|\Pi_d(u)|$.
\end{minipage} \end{array} \]\end{proof}
 
\begin{theorem}
\label{bp-thm:rank-exact}
Assume every pricing problem invoked by \cref{bp-alg:bp-rank} is solved to
global optimality. Then \cref{bp-alg:bp-rank} terminates and returns an
integral routing $x^{dt}_{ij}$, feasible for the rank-$k$ compact problem,
whose objective value $\bar\lambda_{k}$ is optimal over that problem.
\end{theorem}
 
\begin{proof} \[ \begin{array}{:c} \begin{minipage}{0.97\textwidth}
\emph{Termination.} The tree is finite by \cref{bp-lem:finite} and each
node performs finitely many column-generation iterations by
\cref{bp-lem:cg-finite}.
 
\emph{Valid node bound.} On exit from the loop,
$\operatorname{val}(\mathrm{SP}^{\mathsf X}_{d,k}(u))\ge
\alpha^{\mathsf X}_d$ for every unfixed $d$, so no trajectory in $\Pi_d(u)$
has negative reduced cost and the restricted LP optimum is optimal for the
full node LP. This step uses the global-optimality hypothesis on pricing: a
suboptimal pricing solve could report
$\operatorname{val}(\mathrm{SP}^{\mathsf X}_{d,k}(u))\ge
\alpha^{\mathsf X}_d$ while a negative-reduced-cost column exists, and
$LB(u)$ would not be valid. As the node LP relaxes the node integer program,
$LB(u)$ lower-bounds every integer routing feasible at $u$.
 
\emph{Fathoming.} A node is discarded at line~15 only when $LB(u)\ge UB$
(no routing there improves the incumbent), and at lines~17--19 only when all
$X_{d,t,e}$ are integral, in which case \cref{bp-lem:completeness} induces an
integral routing of value $LB(u)$ that is optimal within the node.
 
\emph{Partition.} By \cref{bp-lem:x-valid} the children partition the node's
integer solutions in the variables $X_{d,t,e}=x^{dt}_{ij}$
(\eqref{bp-eq:X-is-x}), and by \cref{bp-lem:x-compat} each restriction is
represented exactly by one pricing fixing, so no feasible child routing is
unreachable. With a valid node bound, a complete partition, and an optimal
integral routing captured at every $X$-integral leaf, standard
branch-and-bound correctness applies: no optimal routing is discarded, and at
termination the incumbent is a global optimum of the rank-$k$ compact problem.
\end{minipage} \end{array} \]\end{proof}
 
\subsection{Sequential lexicographic branch-and-price}
\label{bp-sec:bp-lex}
 
\begin{algorithm}[H]
\caption{Lexicographic branch-and-price under
$\mathrm{RMP}^{\mathsf X}_k$}
\label{bp-alg:bp-lex}
\begin{algorithmic}[1]
\Require formulation $\mathsf X\in\{\textsc{Stela},\textsc{Carla}\}$,
number of ranks $N=|\calA_T|$.
\State Initialise the column pools $\Pi_d^{\mathrm{pool}}$ (e.g.\ from a
       \textsc{TCG-Diag} run).
\For{$k=1,\dots,N$}
  \State Build $\mathrm{RMP}^{\mathsf X}_k$ with inherited
         $\bar\lambda_{1},\dots,\bar\lambda_{k-1}$ and the rank block
         $(\mathcal P^{\mathsf X}_k)$.
  \State Solve it by \cref{bp-alg:bp-rank}, obtaining $\bar\lambda_{k}$ and
         optimal routing $P_k$.
  \State Retain all generated trajectories in the global pools.
    \Comment{valid by \cref{bp-lem:pool}}
\EndFor
\State \Return $P_N$ and the sorted load vector
       $\lambda^{\downarrow}(P_N)$.
\end{algorithmic}
\end{algorithm}

\begin{theorem}
\label{bp-thm:lex-exact}
Assume that for every $k$ the rank-$k$ problem is solved to global optimality
by \cref{bp-alg:bp-rank}. Then \cref{bp-alg:bp-lex} returns a routing whose
sorted load vector is lexicographically minimum over all feasible routings of
the (trail-restricted) T-ASR, and $(\bar\lambda_{1},\dots,\bar\lambda_{N})$
equals $\lambda^{\downarrow}(P_N)$.
\end{theorem}
 
\begin{proof} \[ \begin{array}{:c} \begin{minipage}{0.97\textwidth}
Let $\mathcal{F}$ be the feasible routings of the compact problem (in the
variables $x^{dt}_{ij}$) and $\mathcal{F}_k\subseteq\mathcal{F}$ those
attaining the lexicographic optimum on the first $k$ components. This is the
sequential-lexicographic minimax scheme, whose generic correctness is that of
\cite[Thm.~3]{abernethy2024}; we instantiate it on the rank blocks
$(\mathcal P^{\mathsf X}_k)$ and verify each step.
 
\emph{Base case.} Rank~$1$: $\mathrm{RMP}^{\mathsf X}_1$ minimises the MLU
$\max_{(a,t)}\lambda(a,t)=\lambda^{\downarrow}_1$ over $\mathcal{F}$;
solved globally by \cref{bp-thm:rank-exact},
$\bar\lambda_{1}=\min_{P\in\mathcal{F}}\lambda^{\downarrow}_1(P)$ and the set
of minimisers is $\mathcal{F}_1$.
 
\emph{Step.} Suppose $\bar\lambda_{1},\dots,\bar\lambda_{k-1}$ are the first
$k{-}1$ lexicographic components and $(\mathcal P^{\mathsf X}_{k-1})$
characterises $\mathcal{F}_{k-1}$. The rank block
$(\mathcal P^{\mathsf X}_k)$ inherits the previous levels as the one-sided
bounds $\lambda^{\downarrow}_j(P)\le\bar\lambda_j$ for $j<k$ (never as
equalities) and minimises the $k$-th component: for \textsc{Stela}, the
threshold characterisation
$\lambda^{\downarrow}_k(P)=\min\{z\ge0:|\{(a,t):\lambda(a,t;P)>z\}|\le
k{-}1\}$ is encoded by the current exceedance rows~\eqref{eq:cg-Pk}; for
\textsc{Carla}, the cumulative $\Theta_k=\sum_{r\le k}\lambda^{\downarrow}_r$
is encoded by rows~\eqref{eq:cg-carla-Pk}. Solved globally by
\cref{bp-thm:rank-exact},
$\bar\lambda_{k}=\min_{P\in\mathcal{F}_{k-1}}\lambda^{\downarrow}_k(P)$.
 
\emph{Re-establishing the hypothesis.} We show
$\{P\in\mathcal{F}:\lambda^{\downarrow}_j(P)\le\bar\lambda_j,\ j<k\}
=\mathcal{F}_{k-1}$, so that the one-sided bound at level $k$ cuts out exactly
$\mathcal{F}_k$. Since $\bar\lambda_1=\min_{P\in\mathcal F}
\lambda^{\downarrow}_1(P)$, the bound $\lambda^{\downarrow}_1(P)\le
\bar\lambda_1$ forces equality; inductively,
$\bar\lambda_j=\min\{\lambda^{\downarrow}_j(P):
\lambda^{\downarrow}_i(P)=\bar\lambda_i,\ i<j\}$ forces
$\lambda^{\downarrow}_j(P)=\bar\lambda_j$ for each $j<k$, i.e.\
$P\in\mathcal{F}_{k-1}$; the converse is immediate. Hence
$\mathcal{F}_k=\{P\in\mathcal{F}_{k-1}:\lambda^{\downarrow}_k(P)\le
\bar\lambda_k\}$, which is what $(\mathcal P^{\mathsf X}_k)$ encodes: the
inductive hypothesis holds at level $k$.
 
After rank $N$, $P_N\in\mathcal{F}_N$ attains the lexicographic optimum on
all $N$ components, and $\lambda^{\downarrow}_k(P_N)=\bar\lambda_{k}$ for
every~$k$.
\end{minipage} \end{array} \]\end{proof}

%% file: sections/conclusionanddiscussion.tex
\section{Conclusion}
This report addressed the \textsc{T-ASR} problem posed by the EURO/ROADEF 2026
Challenge: routing each demand through a sequence of segment-routing waypoints,
over a multi-period horizon with scheduled link interventions and a
reconfiguration budget coupling consecutive periods, so as to minimise the vector
of arc--time loads in the leximin sense.

We first cast the problem into a compact formulation and established its
foundational property: once a routing scheme is fixed, the ECMP split coefficients
are constants, so arc loads are linear in the routing variables. On this basis we
developed three compact sequential exact algorithms for the leximin objective,
differing only in how the $k$-th largest load is encoded at each level:
\textsc{Alexa}, which materialises a full permutation matrix; \textsc{Stela}, which
uses threshold-exceedance binaries; and \textsc{Carla}, which uses the
Ogryczak--Tamir cumulative encoding with continuous auxiliaries only. Two
computational campaigns on the \textsc{setAP} benchmark showed that \textsc{Alexa}
does not scale and can be discarded, while \textsc{Stela} and \textsc{Carla} remain
comparable.

We then applied trajectory-based column generation (\textsc{Tcg}) to both surviving
drivers, deriving the master problem and the pricing subproblem for each at rank~$1$
and at a general rank~$k$. The two variants produce routing schemes of comparable
quality: they agree on the largest loads; and almost always on the maximum link
utilisation; and where they differ, they differ early and by little.
\textsc{Stela}-driven \textsc{Tcg} is a marginally safer default at small and medium
sizes, but the gap is small in magnitude and vanishes on the largest instances.
Finally, since column generation was applied only at the root of the integer master,
we developed a branch-and-price extension embedding the pricing within the search
tree, together with the correctness argument relating master columns to the compact
routing variables.

\section{Code Availability}

The full implementation supporting this report is available at
\url{https://github.com/kaoutarbouaachra/stage_3A}.

It includes the exact drivers and the column-generation code, and the experimental pipeline use to produce the figures.